\documentclass[12pt,amstex]{amsbook}

\usepackage{amsfonts}
\usepackage{epsfig}
\usepackage{amsmath}
\usepackage{amssymb}
\usepackage{amscd}
\usepackage{graphicx}

\usepackage{pstricks}

\usepackage{tocvsec2}

\usepackage{stmaryrd}

\usepackage{hyperref}

\input xy
\xyoption{all}

\theoremstyle{definition}

\theoremstyle{remark}

\newtheorem{thm}{Theorem}[section]
\newtheorem{lem}[thm]{Lemma}

\newtheorem{prop}[thm]{Proposition}

\theoremstyle{definition}
\newtheorem{de}[thm]{Definition}

\theoremstyle{remark}
\newtheorem{rem}[thm]{Remark}

\numberwithin{section}{chapter}
\numberwithin{equation}{chapter}

\def \N {\mathbb N}
\def \C {\mathbb C}
\def \Z {\mathbb Z}
\def \R {\mathbb R}

\def \T {\mathbb{T}}
\def \G {\mathbb{G}}

\def \F {\mathcal{F}}

\def \I {\mathcal{I}}

\def \U {\mathcal{U}}

\def \X {\mathcal{X}}

\def \O {\mathcal{O}}

\def \RP {{\bf RP}}
\def \M {{\bf M}}
\def \NN {\mathcal{N}}

\def \id {{\rm id}}

\def \a {\alpha }
\def \b {\beta}
\def \ep {\epsilon}
\def \d {\delta}
\def \D {\Delta}
\def \ll {\lambda}

\def \lra{\longrightarrow}
\def \g {\mathfrak{g}}
\def \exp {\text{exp}}
\def \ad {\text{ad}\ }

\makeindex

\begin{document}

\frontmatter

\title{Nil Bohr-sets and almost automorphy of higher order}

\author{Wen Huang}
\author{Song Shao}
\author{Xiangdong Ye}

\address{Department of Mathematics, University of Science and Technology of China,
Hefei, Anhui, 230026, P.R. China.}

\email{wenh@mail.ustc.edu.cn}\email{songshao@ustc.edu.cn}
\email{yexd@ustc.edu.cn}



\begin{abstract}

Two closely related topics: higher order Bohr sets and higher
order almost automorphy are investigated in this paper. Both of them
are related to nilsystems.

In the first part, the problem which can be viewed as the higher
order version of an old question concerning Bohr sets is studied:
for any $d\in \N$ does the collection of  $\{n\in \Z: S\cap
(S-n)\cap\ldots\cap (S-dn)\neq \emptyset\}$ with $S$ syndetic
coincide with that of Nil$_d$ Bohr$_0$-sets? It is proved that
Nil$_d$ Bohr$_0$-sets could be characterized via generalized
polynomials, and applying this result one side of the problem is
answered affirmatively: for any Nil$_d$ Bohr$_0$-set $A$, there
exists a syndetic set $S$ such that $A\supset \{n\in \Z: S\cap
(S-n)\cap\ldots\cap (S-dn)\neq \emptyset\}.$ Moreover, it is shown
that the answer of the other side of the problem can be deduced from
some result by Bergelson-Host-Kra if modulo a set with zero density.


In the second part, the notion of $d$-step almost automorphic
systems with $d\in\N\cup\{\infty\}$ is introduced and investigated, which is the
generalization of the classical almost automorphic ones. It is worth to mention that some results
concerning higher order Bohr sets will be applied to the investigation. For a
minimal topological dynamical system $(X,T)$ it is shown that the
condition $x\in X$ is $d$-step almost automorphic can be
characterized via various subsets of $\Z$ including the dual sets of
$d$-step Poincar\'e and Birkhoff recurrence sets, and Nil$_d$
Bohr$_0$-sets. Moreover, it turns out that the condition $(x,y)\in
X\times X$ is regionally proximal of order $d$ can also be
characterized via various subsets of $\Z$.

\end{abstract}

\maketitle


\tableofcontents 

\mainmatter

\chapter{Introduction}

In this paper we study two closely related topics: higher order Bohr
sets and higher order almost automorphy. Both of them are related
to nilsystems. In the first part we investigate the higher order
Bohr sets. Then in the second part we study the higher order
automorphy, and explain how these two topics are closely related.

\section{Higher order Bohr problem}

A very old problem from combinatorial number theory and harmonic
analysis, rooted in the classical work of Bogoliuboff, F{\o}lner
\cite{Folner}, Ellis-Keynes \cite{EK}, and Veech \cite{V68}  is the
following. Let $S$ be a syndetic subset of the integers. Is the set
$S -S$ a Bohr neighborhood of zero in $\Z$ (also called
Bohr$_0$-set)? For the equivalent statements and results related to
the problem in combinatorial number theory, group theory and
dynamical systems, see Glasner \cite{G98}, Weiss \cite{W},
Katznelson \cite{Kaz}, Pestov \cite{P}, Boshernitzan-Glasner
\cite{BG}, Huang-Ye \cite{HY12}, Grivaux-Roginskaya \cite{Gri12,
GR12}.

Bohr-sets are fundamentally abelian in nature. Nowadays it has
become apparent that higher order non-abelian Fourier analysis plays
an important role both in combinatorial number theory and ergodic
theory. Related to this, a higher-order version of Bohr$_0$ sets,
namely Nil$_d$ Bohr$_0$-sets, was introduced in \cite{HK10}. For the
recent results obtained by Bergelson-Furstenberg-Weiss and Host-Kra,
see \cite{BFW,HK10}.


\subsection{Nil Bohr-sets}
There are several equivalent definitions for Bohr-sets. Here is the
one easy to understand: a subset $A$ of $\Z$ is a {\em Bohr-set} \index{Bohr-set}
if there exist $m\in \N$, $\a\in \T^m$, and a non-empty open set
$U\subset \T^m$ such that $\{n\in \Z: n\a \in U\}$ is contained in
$A$; the set $A$ is a {\em Bohr$_0$-set} if additionally $0\in U$. \index{Bohr$_0$-set}

It is not hard to see that if $(X,T)$ is a minimal equicontinuous
system, $x\in X$ and $U$ is a neighborhood of $x$, then
$N(x,U)=:\{n\in\Z:T^nx\in U\}$ contains $S-S=:\{a-b: a,b\in S\}$
with $S$ syndetic, i.e. with a bounded gap ($S$ can be chosen as
$N(x,U_1)$, where $U_1\subset U$ is an open neighborhood of $x$).
\index{syndetic} This implies that if $A$ is a Bohr$_0$-set, then
$A\supset S-S$ with $S$ syndetic. The old question concerning
Bohr$_0$-sets is

\medskip

\noindent {\bf Problem A-I:} \ {Let $S$ be a syndetic subset of
$\Z$, is  $S-S$ a Bohr$_0$-set?}

\medskip

Note that Ellis-Keynes \cite{EK} proved that $S-S+S-a$ is a
Bohr$_0$-set for some $a\in S$. Veech showed that it is at least
``almost'' true \cite{V68}. That is, given a syndetic set
$S\subset \Z$, there is an $N\subset \Z$ with density zero such
that $(S-S)\Delta N$ is a Bohr$_0$-set. K$\check{\rm
r}$\'i$\check{\rm z}$ \cite{K} showed that there exists a subset $K$
of $\Z$ with positive upper Banach density such that $K-K$ does not
contains $S-S$ for any syndetic subset $S$ of $\Z$. This implies
that Problem A-I has a negative answer if we replace a syndetic
subset of $\Z$ by a subset of $\Z$ with positive upper Banach
density.
\medskip

A subset $A$ of $\Z$ is a {\em Nil$_d$ Bohr$_0$-set}
\index{Nil$_d$ Bohr$_0$-set} if there exist a $d$-step nilsystem
$(X,T)$, $x_0\in X$ and an open neighborhood $U$ of $x_0$ such that
$N(x_0,U)=:\{n\in \Z: T^n x_0\in U\}$ is contained in $A$. Denote by
$\F_{d,0}$ \index{$\F_{d,0}$} the family\footnote{ A collection $\F$
of subsets of $\Z$ (or $\N$) is {\em a family} \index{family} if it
is hereditary upward, i.e. $F_1 \subset F_2$ and $F_1 \in \F$
imply $F_2 \in \F$. Any nonempty collection $\mathcal{A}$ of subsets
of $\Z$ generates a family $\F(\mathcal{A}) :=\{F \subset \Z:F
\supset A$ for some $A \in \mathcal{A}\}$.} consisting of all
Nil$_d$ Bohr$_0$-sets. We can now formulate a higher order form of
Problem A-I. We note that $$\{n\in \Z: S\cap (S-n)\cap\ldots\cap
(S-dn)\neq \emptyset\}$$ can be viewed as the common differences of
arithmetic progressions with length $d+1$ appeared in the subset
$S$. In fact, $S\cap (S-n)\cap\ldots\cap (S-dn)\neq \emptyset$ if
and only if there is $m\in S$ with $$m,m+n,\ldots,m+dn\in S.$$
Particularly, $S-S=\{n\in\Z: S\cap (S-n)\neq \emptyset\}$.

\medskip

\noindent {\bf Problem B-I:} [Higher order form of Problem A-I] Let
$d\in\N$.

{\em \begin{enumerate} \item For any Nil$_d$ Bohr$_0$-set $A$, is it
true that there is a syndetic subset $S$ of $\Z$ with
$A\supset\{n\in \Z: S\cap (S-n)\cap\ldots\cap (S-dn)\neq
\emptyset\}$?

\item For any syndetic subset $S$ of $\Z$, is $\{n\in \Z:
S\cap (S-n)\cap\ldots\cap (S-dn)\neq \emptyset\}$ a Nil$_d$
Bohr$_0$-set?\end{enumerate}}

\medskip

\subsection{Dynamical version of the higher order Bohr problem}
Sometimes combinatorial questions can be translated into dynamical
ones by the Furstenberg correspondence principle, see Section
\ref{FCP}. Using this principle, it can be shown that Problem A-I is
equivalent to the following version:

\medskip

\noindent {\bf Problem A-II:} {\it For any minimal system $(X,T)$
and any nonempty open subset $U$ of $X$, is the set $\{n\in \Z:
U\cap T^{-n}U\neq \emptyset \}$ a Bohr$_0$-set?}

\medskip

Similarly, Problem B-I has its dynamical version:

\medskip

\noindent {\bf Problem B-II:} [Dynamical version of Problem B-I] Let
$d\in\N$. {\em
\begin{enumerate} \item  For any Nil$_d$ Bohr$_0$-set $A$, is it
true that there are a minimal system $(X,T)$ and a non-empty open
subset $U$ of $X$ with
$$A\supset \{n\in\Z:
U\cap T^{-n}U\cap \ldots \cap T^{-dn}U\not=\emptyset\}?$$

\item For any minimal system $(X,T)$ and any non-empty open subset $U$ of $X$,
is it true that $\{n\in\Z: U\cap T^{-n}U\cap \ldots \cap
T^{-dn}U\not=\emptyset\}$ is a Nil$_d$ Bohr$_0$-set?
\end{enumerate}}

In the next section, we will give the third version of Problem B via recurrence sets. The equivalence of
three versions will be shown in Chapter \ref{section-pre}.

\subsection{Main results on the higher order Bohr problem}

We aim to study Problem B-I or its dynamical version Problem B-II.
We will show that Problem B-II(1) has an affirmative answer, and
Problem B-II(2) has a positive answer if ignoring a set with zero
density. Namely, we will show

\medskip

\noindent {\bf Theorem A}: {\em Let $d\in\N$.
\begin{enumerate} \item
If $A\subset\Z$ is a Nil$_d$ Bohr$_0$-set, then there exist a
minimal $d$-step nilsystem $(X,T)$ and a nonempty open subset $U$ of
$X$ with
\begin{equation*}
    A\supset \{n\in\Z: U\cap T^{-n}U\cap \ldots\cap T^{-dn}U\neq \emptyset \}.
\end{equation*}

\item For any minimal system $(X,T)$ and any non-empty open subset $U$ of $X$,
$I=\{n\in\Z: U\cap T^{-n}U\cap \ldots \cap T^{-dn}U\not=\emptyset\}$
is almost a Nil$_d$ Bohr$_0$-set, i.e. there is $M\subset \Z$ with
zero upper Banach density such that $I\Delta M $ is a Nil$_d$
Bohr$_0$-set\end{enumerate}}

As we said before for $d=1$ Theorem A(1) can be easily proved. To
show Theorem~ A(1) in the general case, we need to investigate the
properties of $\F_{d,0}$. It is interesting that in the process to
do this, generalized polynomials (see \S  \ref{genenal-poly-def} for a
definition) appear naturally. Generalized polynomials have been
studied extensively, see for example the remarkable paper by
Bergelson and Leibman \cite{BL07} and references therein. After
finishing this paper we even find that it also plays an important
role in the recent work by Green, Tao and Ziegler \cite{GTZ}. In
fact the special generalized polynomials defined in this paper are
closely related to the nilcharacters defined there. We remark that
Theorem A(2) was first proved by Veech in the case $d=1$ \cite{V68},
and its proof will be presented in Section \ref{simpleproof}.

Let $\F_{GP_d}$ (resp. $\F_{SGP_d}$) be the family generated by the
sets of forms
$$\bigcap_{i=1}^k\{n\in\Z:P_i(n)\ (\text{mod}\ \Z)\in (-\ep_i,\ep_i)
\},$$ where $k\in\N$, $P_1,\ldots,P_k$ are generalized polynomials
(resp. special generalized polynomials) of degree $\le d$, and
$\ep_i>0$. For the precise definitions see Chapter \ref{section-GP}. We
remark that one can in fact show that $\F_{GP_d}=\F_{SGP_d}$
(Theorem \ref{gpsgp}).

The following theorem illustrates the relation between Nil$_d$
Bohr$_0$-sets and the sets defined above using generalized
polynomials.

\medskip

\noindent {\bf Theorem B}: {\em Let $d\in\N$. Then
$\F_{d,0}=\F_{GP_d}$.}

\medskip
When $d=1$, we have $\F_{1,0}=\F_{SGP_1}$. This is the result of
Katznelson \cite{Kaz}, since $\F_{SGP_1}$ is generated by sets of
forms $\cap_{i=1}^k \{n\in\Z: na_i\ (\text{mod}\ \Z)\in
(-\ep_i,\ep_i)\}$ with $k\in\N$, $a_i\in\R$ and $\ep_i>0$.

Theorem A(1) follows from Theorem B and the following result:

\medskip

\noindent {\bf Theorem C}: {\em Let $d\in\N$.  If $A\in \F_{GP_d}$,
then there exist a minimal $d$-step nilsystem $(X,T)$ and a nonempty
open subset $U$ of $X$ such that
$$A\supset \{n\in\Z: U\cap T^{-n}U\cap \ldots\cap T^{-dn}U\neq \emptyset \}.$$}

The proof of Theorem B is divided into two parts, namely

\medskip

\noindent {\bf Theorem B(1)}: {\em $\F_{d,0}\subset\F_{GP_d}$} and

\medskip

\noindent {\bf Theorem B(2)}: {\em $\F_{d,0}\supset\F_{GP_d}$.}

\medskip

The proof of Theorem B(1) is a theoretical argument using nilpotent
Lie group theory; and the proofs of Theorem B(2) and Theorem C
involve very complicated construction and computation where
nilpotent matrix Lie groups are used.

\begin{rem} Our definition of generalized polynomials is slight different from the ones
defined in \cite{BL07}. In fact we need to specialize the degree of
the generalized polynomials which is not needed in \cite{BL07}.
Moreover, our Theorem B can be compared with Theorem A of Bergelson
and Leibman proved in \cite{BL07}.
\end{rem}

\section{Higher order almost automorphy}

The notion of almost automorphy was first introduced by Bochner in
\cite{Bochner55, Bochner62}, and Veech studied almost automorphic
systems in \cite{V65} using Fourier analysis. To study higher order
almost automorphic systems it is expected that nilpotent Lie groups
and higher order Fourier analysis will be involved and it turns out
that it is the case. We will apply results obtained in the first part to study higher
order almost automorphic systems, namely $d$-step almost automorphic systems
which by the definition are the almost one-to-one extensions of
their maximal $d$-step nilfactors with $d\in\N\cup\{\infty\}$ (an
$\infty$-step nilsystem was defined in \cite{D-Y}). Since for a
minimal system the maximal $d$-step nilfactor is induced by the
regionally proximal relation of order $d$ (which is a closed
invariant equivalence relation \cite{HKM,SY}), the natural way we
study $d$-step almost automorphic systems is that we first show some
characterizations of regionally proximal relation of order $d$, and
then use them to obtain results for $d$-step almost automorphic
systems. In the process of doing above many interesting subsets of $\Z$
including higher order Poincar\'e and Birkhoff recurrence sets
(usual and cubic versions), higher order Bohr-sets, $SG_d$ sets
(introduced in \cite{HK10}) and others are involved. In this section
we introduce some backgrounds of our study and then state the main
results on higher order almost automorphy.

\medskip

First we give some backgrounds.

\subsection{Almost periodicity, almost automorphy and characterizations}

The study of (uniformly) almost periodic functions was initiated by
Bohr in a series of three papers 1924-26 which can be found in
\cite{Bohr}. The literature on almost periodic functions is
enormous, and the notion has been generalized in several directions.
Nowadays the theory of almost periodic functions may be recognized
as the representation theory of compact Hausdorff groups: every
topological group $G$ has a group compactification $\a_G:
G\rightarrow b G$ such that the space of almost periodic functions
on $G$ is just the set of all functions $f\circ \a_G$ with $f\in C(b
G)$. The compactification $(\a_G, b G)$ of $G$ is called the {\em
Bohr compactification} of $G$. \index{Bohr compactification}

A class of functions related to the almost periodic ones is the
class of {\em almost automorphic functions}: \index{almost automorphic function} these functions turn
out to be the ones of the form $h\circ \a_G$ with $h$ a bounded
continuous function on $\a_G(G)$ (if $h$ is uniformly continuous and
bounded on $\a_G(G)$, then it extends to an $f\in C(bG)$, so $h\circ
\a_G=f\circ \a_G$ is almost periodic on $G$).

\medskip

The notion of almost automorphy was first introduced by Bochner in
1955 in a work of differential geometry \cite{Bochner55, Bochner62}.
Taking $G$ for the present to be the group of integers $\Z$, an
almost automorphic function $f$ has the property that from any
sequence $\{n_i'\}\subset \Z$ one may extract a subsequence
$\{n_i\}$ such that both
\begin{equation*}
    \lim_{i\to \infty} f(t+n_i)=g(t)\quad \text{and }\quad \lim_{i\to \infty}
    g(t-n_i)=f(t)
\end{equation*}
hold for each $t\in \Z$ and some function $g$, not necessarily
uniformly. Bochner  \cite{Bochner62} has observed that almost
periodic functions are almost automorphic, but the converse is not
true. Veech \cite{V65} showed that the almost automorphic functions
can be characterized in terms of the almost periodic ones, and vice
versa. In the same paper, Veech considered the system associated
with an almost automorphic function, and introduced the notion of
{\em almost automorphic point} ({\em AA point}, for short) in
topological dynamical systems (t.d.s. for short). \index{almost automorphic point}\index{AA point}
For a t.d.s. $(X, T)$, a point $x\in X$ is said to be {\em almost
automorphic} if from any sequence $\{n_i'\}\subset \Z$ one may
extract a subsequence $\{n_i\}$ such that
$$\lim_{j\to \infty}\lim_{i\to \infty} T^{n_i-n_j}x=x.$$ Moreover, Veech  \cite{V65, V68} gave the structure theorem
for minimal systems with an AA point: each minimal AA system is an
almost one-to-one extension of its maximal equicontinuous factor.

The notion of almost automorphy is very useful in the study of
differential equations, and  see \cite{ShenYi} and references
therein for more information on this topic.

\medskip

To state other characterizations of an AA point we need to introduce
the notion of regionally proximal relations, and Poincar\'e and
Birkhoff recurrence sets.

Let us first discuss regionally proximal relations. For a t.d.s.
$(X,T)$, it was proved in \cite{EG} that there exists a closed
$T$-invariant equivalence relation $S_{eq}$ on $X$ such that
$(X/S_{eq}, T)$ is the maximal equicontinuous factor.
$S_{eq}$ is called the {\em equicontinuous structure relation}. \index{equicontinuous structure relation}
It was also showed in \cite{EG} that $S_{eq}$ is the smallest closed
$T$-invariant equivalence relation containing the regionally
proximal relation $\RP=\RP(X,T)$ \index{regionally proximal relation} (recall that $(x,y)\in \RP$ if
there are sequences $x_i,y_i\in X, n_i\in \Z$ such that $x_i\to x,
y_i\to y$ and $(T\times T)^{n_i}(x_i,y_i)\to (z,z)$, $i\to \infty$,
for some $z\in X$). A natural question was whether $S_{eq}=\RP(X)$
for all minimal t.d.s.? Veech \cite{V68} gave the first positive
answer to this question, i.e. he proved that $S_{eq}=\RP(X)$ for all
minimal t.d.s. under abelian group actions. As a matter of fact,
Veech proved that for a minimal t.d.s. $(x,y)\in S_{eq}$ if and only
if there is a sequence $\{n_i\}\subset \Z$ and $z\in X$ such that
\begin{equation*}
    T^{n_i}x\lra z \quad \text{and}\quad T^{-n_i}z\lra y, \ i\to \infty.
\end{equation*}
As a direct corollary, for a minimal t.d.s. $(X,T)$, {\em a point
$x\in X$ is AA if and only if }$$\RP[x]=\{y\in X: (x,y)\in
\RP\}=\{x\}.$$

Also from Veech's approach, it is easy to show that for a minimal
t.d.s. $(X,T)$, $(x,y)\in \RP$ if and only if for each neighborhood
$U$ of $y$, $N(x,U)=:\{n\in\Z: T^nx\in U\}$ contains some
$\D$-set\footnote{A $\D$-set is obtained by taking an arbitrary
sequence in $\Z$, $\{s_n\}$ and forming the difference
$\{s_n-s_m:n>m\}$. A $\D^*$ set is a subset of $\Z$ intersecting all
$\D$-sets.\index{$\D$-set}\index{$\D^*$-set}}. Hence one can obtain an equivalent condition for an AA
point \cite[Theorem 9.13]{F}: {\em a point $x\in X$ is AA if and
only if it is $\D^*$-recurrent}\footnote{Let $\F$ be a collection of
subsets of $\Z$ and let $(X,T)$ be a t.d.s. A point $x$ of $X$ is
called {\em $\F$-recurrent} \index{$\F$-recurrent} if $N(x,U)\in \F$ for every neighborhood
$U$ of $x$.} (this result will be reproved by a different method as
a special case of our theorems). For other properties related to
$\Delta^*$-sets, see \cite{BFW,HK10}.

\medskip



Now we discuss Poincar\'e and Birkhoff recurrence sets. The
Birkhorff recurrence theorem states that each t.d.s. has a recurrent
point which implies that whenever $(X, T)$ is a minimal t.d.s. and
$U\subset X$ a nonempty open set, $N(U,U)=:\{n\in\Z: U\cap
T^{-n}U\neq \emptyset\}$ is infinite. The measurable version of this
phenomenon is the well known Poincar\'e's recurrence theorem: let
$(X,\X,\mu,T)$ be a measure preserving system and $A\in \X$ with
$\mu(A)>0$, then $N_\mu(A,A)=:\{n\in \Z: \mu(A\cap T^{-n}A)>0\}$ is
infinite.

\medskip
In \cite{F, F81} Furstenberg introduced the notion of Poincar\'e and
Birkhoff recurrence sets. A subset $P$ of $\Z$ is called a {\em
Poincar\'e recurrence set} \index{Poincar\'e recurrence set} if whenever $(X,\X,\mu,T)$ is a measure
preserving system and $A\in \X$ has positive measure, then $P\cap
N_\mu(A,A)\neq \emptyset $. Similarly, a subset $P$ of $\Z$ is
called a {\em Birkhoff recurrence set} \index{Birkhoff recurrence set} if whenever $(X, T)$ is a
minimal t.d.s. and $U\subset X$ a nonempty open set, $P\cap
N(U,U)\neq \emptyset$. Let $\F_{Poi}$ and $\F_{Bir}$ denote the
collections of Poincar\'e  and Birkhoff recurrence sets of $\Z$
respectively.

\medskip

In \cite{HLY}, it was shown that for a minimal t.d.s. $(x,y)\in \RP$
if and only if for each neighborhood $U$ of $y$, $N(x,U)\in
\F_{Poi}$. We will show that one can use $\F_{Poi}$ to get another
equivalent condition for an AA point: {\em a point $x\in X$ is AA if
and only if it is $\F_{Poi}^*$-recurrent,} where $\F_{Poi}^*$ is the
collection of subsets of $\Z$ intersecting all sets from $\F_{Poi}$.
One has similar results for Birkhoff recurrence sets.




\subsection{Nilfactors and higher order almost automorphy}

In the 1970's Furstenberg gave a beautiful proof of Szemer\'edi's
theorem via ergodic theory \cite{F77}. It remains a question if the
multiple ergodic averages $ \frac 1 N\sum_{n=0}^{N-1}f_1(T^nx)\ldots
f_d(T^{dn}x)$ converges in $L^2(X,\mu)$ for  $f_1, \ldots , f_d \in
L^\infty(X,\mu)$. This question was finally answered by Host and Kra
in \cite{HK05} (see also Ziegler in \cite{Z}).

The authors in \cite{HK05} defined for each $d\in\N$ and each
measure-preserving transformation on the probability space
$(X,\mathcal{B},\mu)$ a factor ${\mathcal Z}_d$ which is
characteristic and is an inverse limit of $d$-step nilsytems.  Since
topological dynamics and ergodic theory are `twins', it is natural
to ask how to obtain similar factors in topological dynamics. In the
pioneer paper \cite{HKM} Host-Kra-Maass succeeded doing the job for
minimal distal systems. Namely, for each $d\in\N$ and each t.d.s.
$(X,T)$ they defined $\RP^{[d]}(X,T)$ (the regionally proximal
relation of order $d$) and showed that $\RP^{[d]}(X,T)$ is an
equivalence relation when $(X,T)$ is minimal distal, and
$X/\RP^{[d]}(X,T)$ is the maximal $d$-step nilfactor of $(X,T)$.
Recently, Shao and Ye \cite{SY} proved that the above conclusion
holds for general minimal systems. We note that the counterpart of
the characteristic factors in topological dynamics was also studied
by Glasner \cite{G93,G94}. For the further study and applications of
the factors, see \cite{D-Y,HKM12}.


\medskip


As we said before the notion of the regionally proximal relation of
order $d$  plays an important role in obtaining the maximal $d$-step
nilfactors, see Section 2 for the definition. It is easy to see that
$\RP^{[d]}(X,T)$ is a closed and invariant relation for all $d\in
\N$. When $d=1$, $\RP^{[d]}(X,T)$ is nothing but the classical
regionally proximal relation. Similar to the definition of almost
automorphy, now we give the definition of $d$-step almost automorphy
for all $d\in \N$. Let $(X,T)$ be a minimal t.d.s. and  $d\in \N$. A
point $x\in X$ is called {\em $d$-step almost automorphic} (or
$d$-step AA for short) \index{$d$-step almost automorphic point} \index{$d$-step AA} if $\RP^{[d]}[x]=\{x\}$. A minimal t.d.s.  is
called {\em $d$-step almost automorphic} if it has a $d$-step AA
point. Since $\RP^{[d]}$ is an equivalence relation for minimal
t.d.s. \cite{SY}, by definition it follows that for a minimal system
$(X,T)$, it is a $d$-step AA system for some $d\in \N$ if and only
if it is an almost one-to-one extension of its maximal $d$-step
nilfactor.

\subsection{Higher order recurrence sets}

In this paper we will use higher order recurrence sets to
characterize $d$-step almost automorphy. To define them we need to
state the multiple Poincar\'e and Birkhoof recurrence theorems, see
\cite{F}.

\medskip
\noindent $\bullet$  Let $(X,\X,\mu,T)$ be a measure preserving
system and $d\in\N$. Then for any $A\in \X$ with $\mu(A)>0$ there is
$n\in\N$ such that $\mu(A\cap T^{-n}A\cap \ldots\cap T^{-dn}A)>0$.

\medskip

\noindent $\bullet$ Let $(X,T)$ be a t.d.s. and $d>0$. Then there
are $x\in X$ and a sequence $\{n_i\}$ with $n_i\lra +\infty$ such
that $\lim_{i\lra+\infty}T^{jn_i}x=x$ for each $1\le j\le d$.

\medskip

The facts enable us to get generalizations of Poincar\'e and
Birkhoff recurrence sets (see \cite{FLW}). Let $d\in \N$.
\begin{enumerate}\index{$d$-recurrence}\index{$d$-topological recurrence}
\item We say that $S \subset \Z$ is a set of {\em $d$-recurrence } if
for every measure preserving system $(X,\X,\mu,T)$ and for every
$A\in \X$ with $\mu (A)
> 0$, there exists $n \in S$  such that
$\mu(A\cap T^{-n}A\cap \ldots \cap T^{-dn}A)>0.$


\item We say that $S\subset \Z$ is a set of {\em $d$-topological
recurrence} \index{$d$-topological
recurrence} if for every minimal t.d.s. $(X, T)$ and for every
nonempty open subset $U$ of $X$, there exists $n\in S$ such that
$U\cap T^{-n}U\cap \ldots \cap T^{-dn}U\neq \emptyset.$
\end{enumerate}

\begin{rem}
We remark that in (1) we can require that $(X,\X,\mu,T)$ is ergodic
(see the proof of Theorem \ref{several}). The above definitions are
slightly different from the ones introduced in \cite{FLW}, namely we
do not require $n\not=0$. The main reason we define in this way is
that for each $A\in \F_{d,0}$, $0\in A$. Thus $\{0\}\cup C\in
\F^*_{d,0}$ for each $C\subset \Z$, where $\F_{d,0}^*$ is the dual
family of $\F_{d,0}$, i.e. the collection of sets intersecting every
Nil$_d$ Bohr$_0$-set.
\end{rem}

Let $\F_{Poi_d}$ (resp. $\F_{Bir_d}$) \index{$\F_{Poi_d}$} \index{$\F_{Bir_d}$} be
the family consisting of all sets of $d$-recurrence (resp. sets of
$d$-topological recurrence). It is obvious by the above definition
that $\F_{Poi_d}\subset \F_{Bir_d}$. Moreover, it is known that for
each $d\in\N$, $\F_{Poi_d}\supsetneqq \F_{Poi_{d+1}}$ and
$\F_{Bir_d}\supsetneqq \F_{Bir_{d+1}}$ \cite{FLW}. Now we state a
problem which is related to Problem B-II.

\medskip

\noindent {\bf Problem B-III:} {\em Is it true that $\F_{Bir_d} =
\F^*_{d,0}$?}

\medskip

An immediate corollary of Theorem A(1) is:

\medskip

\noindent {\bf Corollary D}: {\em Let $d\in\N$. Then
$$\F_{Poi_d}\subset\F_{Bir_d}\subset \F^*_{d,0}.$$ }

Note that $\F_{Poi_1}\neq \F_{Bir_1}$ \cite{K}. Though we do not
know if $\F_{Bir_d}=\F^*_{d,0}$, we will show that the two
collections coincide ``dynamically'', i.e. both of them can be used
to characterize higher order almost automorphic points.




\subsection{Main results on higher order almost automorphy}

As we said before, Veech studied AA systems, and Veech and
Furstenberg gave characterizations of AA systems in \cite{V65} and
\cite{F81} respectively. In this paper we aim to define $d$-step AA
systems and obtain their characterizations for
$d\in\N\cup\{\infty\}$. Now we state the main results of the paper.
For the definitions when $d=\infty$ see the following sections.

\subsubsection{$d$-step AA and proximal relations of order $d$}

The following result shows that we can use $\F_{Poi_d}$,
$\F_{Bir_d}$ and $\F_{d,0}^*$ to characterize regionally proximal
pairs of order $d$. Precisely, we show in Theorems \ref{several} and
\ref{rpd} that: for a minimal t.d.s. $(X,T)$, $d\in\N\cup\{\infty\}$
and $x,y\in X$, the following statements are equivalent: (1)
$(x,y)\in \RP^{[d]}(X,T)$. (2) $N(x,U)\in \F_{Poi_d}$ for each
neighborhood $U$ of $y$. (3) $N(x,U)\in \F_{Bir_d}$ for each
neighborhood $U$ of $y$. (4) $N(x,U)\in \F_{d,0}^*$ for each
neighborhood $U$ of $y$.

\subsubsection{$d$-step AA and $SG_d$-sets}

The notion of $SG_d$-sets was introduced by Host and Kra recently
\cite{HK10} to deal with problems related to Nil$_d$ Bohr$_0$-sets.
We show that one may use it to characterize regionally proximal
pairs of order $d$.

\medskip

Let $d\ge 1$ be an integer and let $P=\{p_i\}_i$ be a (finite or
infinite) sequence in $\Z$. The {\em set of sums with gaps of length
less than $d$} of $P$ is the set $SG_d(P)$ of all integers of the
form \index{$SG_d(P)$}\index{set of sums with gaps of length
less than $d$}
$$\ep_1p_1+\ep_2p_2+\ldots +\ep_np_n$$ where $n\ge 1$ is an integer,
$\ep_i\in \{0,1\}$ for $1\le i\le n$, the $\ep_i$ are not all equal
to $0$, and the blocks of consecutive $0$'s between two $1$ have
length less than $d$. A subset $A$ of $\Z$ is an $SG_d$-set
\index{$SG_d$-set} if $A=SG_d(P)$ for some infinite sequence in
$\Z$; and it is an $SG^*_d$-set \index{$SG^*_d$-set} if $A \cap
SG_d(P)\neq \emptyset$ for every infinite sequence $P$ in $\Z$. Let
$\F_{SG_d}$ be the family generated by all $SG_d$-sets.  Note that
each $SG_1$-set is a $\D$-set, and each $SG_1^*$-set is a
$\D^*$-set. The following is the main result of \cite{HK10}: every
$SG_d^*$-set is a PW-Nil$_d$ Bohr$_0$-set. Host and Kra \cite{HK10}
asked the following question: Is every Nil$_d$ Bohr$_0$-set an
$SG_d^*$-set?


Though we can not answer this question, we show in Theorems \ref{huang12} and \ref{rpd} that: for a minimal
t.d.s., $d\in \N\cup\{\infty\}$ and $x,y\in X$,  $(x,y)\in
\RP^{[d]}(X,T)$ if and only if $N(x,U)\in \F_{SG_d}$ for each
neighborhood $U$ of $y$.
Combining Theorems \ref{several} and \ref{huang12} we see that
Nil$_d$ Bohr$_0$-sets and $SG_d^*$-sets are closely related.

\subsubsection{Cubic version of multiple Poincar\'e recurrence sets}

In \cite{HK05} Host and Kra proved the $L^2$ convergence of the
multiple ergodic average of cubic version. Using it one may define
cubic version of multiple Poincar\'e and Birkhoff recurrence sets.
We will show that they can be used to characterize $\RP^{[d]}$.  For
$d\in \N$, a subset $F$ of $\Z$ is a {\em Poincar\'e recurrence set
of order $d$} \index{Poincar\'e recurrence set of order $d$} if for each measure preserving system
$(X,\mathcal{B},\mu,T)$ and $A\in \mathcal{B}$ with positive measure
there are $n_1,\ldots,n_d\in\Z$ such that
$FS(\{n_i\}_{i=1}^d)=:\{n_{i_1}+\ldots+n_{i_k}:1\leq
i_1<\ldots<i_k\leq d\}\subset F$ and
$$\mu\Big(A\cap\big (\bigcap _{n\in
FS(\{n_i\}_{i=1}^d)} T^{-n}A\big )\Big)>0.$$ Similarly, we define
Birkhoff recurrence sets of order $d$. \index{Birkhoff recurrence set of order $d$} For $d\in \N$ let $\F_{P_d}$
(resp. $\F_{B_d}$) be the family of all Poincar\'e recurrence sets
of order $d$ (resp. the family of all Birkhoff recurrence sets of
order $d$). We have the following result proved in Theorems
\ref{birkhoff-thm} and \ref{poincare} (see also Theorem \ref{rpd}):
for a minimal t.d.s. $(X,T)$, $d\in \N\cup\{\infty\}$ and $x,y\in
X$,  $(x,y)\in \RP^{[d]}(X,T)$ if and only if  $N(x,U)\in \F_{P_d}$
for each neighborhood $U$ of $y$ if and only if $N(x,U)\in \F_{B_d}$
for each neighborhood $U$ of $y$.

\subsubsection{Summary of the main results on higher order almost
automorphy}

To sum up we have (see Theorem \ref{rpd}):

\medskip

\noindent {\bf Theorem E:} {\em Let $(X,T)$ be a minimal t.d.s. and
$x,y\in X$. Then the following statements are equivalent for
$d\in\N\cup\{\infty\}$:

\begin{enumerate}

\item $(x,y)\in \RP^{[d]}$.

\item $N(x,U)\in \F_{d,0}^*$ for each
neighborhood $U$ of $y$.

\item $N(x,U)\in \F_{Poi_d}$ for each
neighborhood $U$ of $y$.

\item $N(x,U)\in \F_{Bir_d}$ for each neighborhood $U$ of $y$.

\item $N(x,U)\in \F_{SG_d}$ for each neighborhood $U$ of $y$.


\item $N(x,U)\in \F_{B_d}$ for each
neighborhood $U$ of $y$.

\item $N(x,U)\in \F_{P_d}$ for each neighborhood
$U$ of $y$.

\end{enumerate}}

\medskip

Using the Ramsey property of the families, we can show that one may
use $\F_{Poi_d}^*$, $\F_{Bir_d}^*$ and $\F_{d,0}$ to characterize
$d$-step AA. That is, we show in Theorem \ref{AAgeneral} that:

\medskip

\noindent {\bf Theorem F:} {\em Let $(X,T)$ be a minimal t.d.s.,
$x\in X$ and $d\in\N\cup\{\infty\}$. Then the following statements
are equivalent:
\begin{enumerate}
\item $x$ is a $d$-step AA point.

\item $N(x,V)\in \F_{d,0}$ for each neighborhood $V$ of $x$.

\item $N(x,V)\in \F_{Poi_d}^*$ for each
neighborhood $V$ of $x$.

\item $N(x,V)\in \F_{Bir_d}^*$ for each
neighborhood $V$ of $x$.

\end{enumerate}}


\section{Further questions}

It is believed that a ``big'' subset of integers should contain
``good'' linear structures, for example arbitrarily long arithmetic
progressions. Szeme\'rdi's Theorem \cite{Sz} asserts that every positive
density subset has this property. In the spirit of Theorem A it is
natural to consider the structure of the set of all common
differences for a ``big'' set. To be precise let $d \ge 1$ and
assume that $S$ is a ``big''  subset of integers. Then all common
differences $n$ of arithmetic progressions
$$a, a+n, a+2n, \ldots, a+ d n$$ with length $d+1$ appearing in $S$
form a set
$$ C_d(S):=\{n\in \N : S\cap (S-n)\cap (S-2n)\cap \ldots \cap
(S- d n)\neq \emptyset \}. $$  What can we say about the structure
of the set $C_d (S)$?

By Green-Tao's result \cite{GT08}, the primes contain arbitrarily
long arithmetic progressions. So we ask the question:

\medskip
\noindent {\bf Question 1:} {\em  Is $C_d(\mathbb{P})$ a Nil$_d$
Bohr$_0$-set for prime numbers $\mathbb{P}$?}

\medskip

A direct corollary of Theorems \ref{huang12} and \ref{rpd} is the
following. Assume $(X,T)$ is minimal, $x\in X$ and $d\in \N$. If $x$
is $\F^*_{SG_d}$-recurrent, or $\F^*_{P_d}$-recurrent, or
$\F^*_{B_d}$-recurrent then it is $d$-step AA. Thus, we have the
following question.

\medskip

\noindent{\bf Question 2:} Let $(X,T)$ be a minimal t.d.s., $x\in
X$, and $d\in \N$. Is it true that $x$ is $d$-step AA if and only if
it is $\F_{SG^*_d}$-recurrent if and only if it is
$\F_{P_d}^*$-recurrent if and only if it is $\F_{B_d}^*$-recurrent?

\medskip

Since $\F_{SG_d}$ does not have the Ramsey property (Appendix
\ref{appendix:Ramsey}), and we do not know if $\F_{P_d}$ and
$\F_{P_d}$ have the Ramsey property, we can not apply the methods in
the proof of Theorem \ref{AAgeneral} to solve Question 2. Note that if
the question by Host-Kra in \cite{HK10} has a positive answer, then
by using Theorem \ref{AAgeneral} Question 2 has a positive answer for
$\F_{SG_d}$. We note that there are two possible ways to get
positive answer of Question 2 for $\F_{P_d}$ and $\F_{P_d}$: (1)
prove $\F_{P_d}$ and $\F_{B_d}$ have the Ramsey property, (2) prove
$\F_{P_d}\subset \F_{B_d}\subset \F^*_{d,0}$. Unfortunately, at this
moment we can not prove neither of them.

\medskip

Recall that Veech \cite{V65} showed that for a t.d.s. $(X, T)$, a
point $x\in X$ is AA if from any sequence $\{n_i'\}\subset \Z$ one
may extract a subsequence $\{n_i\}$ such that \break $\lim_{j\to
\infty}\lim_{i\to \infty} T^{n_i-n_j}x=x.$ So we have

\medskip

\noindent{\bf Question 3:} Is there a similar characterization for a
$d$-step AA point?

\medskip
Let $C(X,Y)$ be the collection of all continuous maps from a
topological space $X$ to a topological space $Y$. Let $V$ be a
finite dimensional vector space over the complex field $\C$.  A
function $f\in C(\R, V)$ is said to be {\it admissible} if it is
bounded and uniformly continuous on $\R$. Let $H(f)$ denote the {\it
hull} of $f$, i.e. the closure of $\{f_{\tau}| \tau \in \R\}$ in the
compact open topology, where $f_{\tau} (t)=f(t+\tau)\ (\tau\in \R)$.
Then by Ascoli's theorem, $H(f)$ is compact and the time translation
$\Pi_{t}g=g_t\ (g\in H(f))$ induces a compact flow $(H(f), \R)$. For
$d\in\N\cup\{\infty\}$, we say $f\in C(\R,\C)$ is a $d$-step AA
function if $(H(f), \R)$ is an almost one-to-one extension of a
minimal $d$-step nilflow (see \cite{Z2} for the definition) and
$f\in H(f)$ is a $d$-step AA point.

\medskip

\noindent{\bf Question 4:} Is there a differential equation which
has a $2$-step AA solution and does not have an AA one?


\medskip
Recall that Host and Kra \cite{HK10} asked: Is every Nil$_d$
Bohr$_0$-set an $SG_d^*$-set? Using Theorem B, the question of Host
and Kra can be reformulated in the following way:

\medskip

\noindent {\bf Question 5:} {\em Let $d\in\N$ and $S$ be an
SG$_d$-set. Is it true that for any $k\in\N$, any $P_1,\ldots,
P_k\in {SGP_d}$ and any $\ep_i>0$, there is $n\in S$ such that
$$P_i(n)\ (\text{mod}\ \Z)\in (-\ep_i,\ep_i)$$
for all $i=1,\ldots,k$? }

\medskip

We remark that since a system of order $d$ is distal, the above
question has an affirmative answer for any IP-set.

\section{Organization of the paper}

We organize the paper as follows: In Chapter \ref{section-pre}, we
give some basic definitions, and particularly we show the
equivalence of the Problems B-I, B-II and B-III. In Chapter
\ref{section-nilsystem} we recall some basic facts related to
nilpotent Lie groups and nilmanifolds, and study the properties of
the metric on nilpotent matrix Lie groups. In Chapter
\ref{section-GP}, we introduce the notions related to generalized
polynomials and special generalized polynomials, and give some basic
properties. In the next two chapters we prove the main results.

In Chapter \ref{chapter-rp}, we study
Nil$_d$-Bohr$_0$ sets and higher order recurrence sets, and use them
to characterize $\RP^{[d]}$. In the final chapter, we introduce the
notion of $d$-step almost automorphy and obtain various
characterizations. In the Appendix, we show $\F_{SG_2}$ does not
have the Ramsey property, Theorem \ref{ShaoYe} holds for general
compact Hausdorff systems and the cubic version of the multiple
Poincar\'e and Birkhoff recurrence sets can be interpreted using
intersectiveness.

\medskip
\noindent{\bf Acknowledgments:} We thank V. Bergelson, N. Frantzikinakis and
E. Glasner for useful comments. We would like to thank Jian Li
for very careful reading which helps the writing of our paper.
{ In particular, we thank the referee for the very careful reading
and many useful comments, which help us to improve the writing of the paper and
simplify some proofs.}


\chapter{Preliminaries}\label{section-pre}

In this chapter we introduce some basic notions related to dynamical
systems, explain how Bergelson-Host-Kra's result is related to
Problem B-II and show the equivalence of the Problems B-I, B-II and
B-III.

\section{Basic notions}

\subsection{Measurable and topological dynamics}

In this subsection we give some basic notions in ergodic theory and
topological dynamics.

\subsubsection{Measurable systems}
In this paper, a {\em measure preserving system} \index{measure preserving system} is a quadruple
$(X,\X, \mu, T)$, where $(X,\X,\mu )$ is a Lebesgue probability
space and $T : X \rightarrow X$ is an invertible measure preserving
transformation.


We write $\I=\I (T)$ for the $\sigma$-algebra $\{A\in \X : T^{-1}A =
A\}$ of invariant sets. A system is {\em ergodic} \index{ergodic} if every
$T$-invariant set has measure either $0$ or $1$.


\subsubsection{Topological dynamical systems}

A {\em transformation} of a compact metric space X is a
homeomorphism of X to itself. A {\em topological dynamical system},
referred to more succinctly as just a t.d.s. or a {\em system}, \index{topological dynamical system} is a
pair $(X, T)$, where $X$ is a compact metric space and $T : X
\rightarrow X$ is a transformation. We use $\rho (\cdot, \cdot)$ to
denote the metric on $X$.

A t.d.s. $(X, T)$ is {\em transitive} \index{transitive} if
there exists some point $x\in X$ whose orbit
$\O(x,T)=\{T^nx: n\in \Z\}$ is dense in $X$.\index{$\O(x,T)$}\index{orbit}
The system is {\em minimal} \index{minimal} if the
orbit of any point is dense in $X$. This property is equivalent to
saying that X and the empty set are the only closed invariant sets
in $X$.

A {\em factor} \index{factor} of a t.d.s. $(X, T)$ is another t.d.s. $(Y, S)$ such
that there exists a continuous and onto map $\phi: X \rightarrow Y$
satisfying $S\circ \phi = \phi\circ T$. In this case, $(X,T)$ is
called an {\em extension } \index{extension} of $(Y,S)$ and the map $\phi$ is called a {\em factor map}. \index{factor map}

\subsubsection{}

We also make use of a more general definition of a
measurable or topological system. That is, instead of just a single
transformation $T$, we consider commuting homeomorphisms $T_1,\ldots
, T_k$ of $X$ or a countable abelian group of transformations.
We summarize some basic definitions and properties of systems in the
classical setting of one transformation. Extensions to the general
case are straightforward.

\subsection{Families and filters}

Since many statements of the paper are better stated using the
notion of a family, we now give the definition. See \cite{A, F} for
more details.

\subsubsection{Furstenberg families}

Recall that a collection $\F$ of subsets of $\Z$  is  {\em
a family} \index{family} if it is hereditary upward, i.e. $F_1 \subset F_2$ and
$F_1 \in \F$ imply $F_2 \in \F$. A family $\F$ is called {\em
proper} \index{proper family} if it is neither empty nor the entire power set of $\Z$, or,
equivalently if $\Z \in \F$ and $\emptyset \not\in \F$. Any nonempty
collection $\mathcal{A}$ of subsets  of $\Z$ generates a family
$\F(\mathcal{A}) :=\{F \subset \Z:F \supset A$ for some $A \in
\mathcal{A}\}$. \index{$\F(\mathcal{A})$}

For a family $\F$ its {\em dual} \index{dual family} is the family $\F^{\ast}
:=\{F\subset \Z  : F \cap F' \neq \emptyset \ \text{for  all} \ F'
\in \F \}$. It is not hard to see that $\F^*=\{F\subset\Z:
\Z\setminus F\not\in \F\}$, from which we have that if $\F$ is a
family then $(\F^*)^*=\F.$

\subsubsection{Filter and Ramsey property}

If a family $\F$ is closed under finite intersections and is proper,
then it is called a {\em filter}. \index{filter}

A family $\F$ has the {\em Ramsey property} \index{Ramsey property} if $A=A_1\cup A_2\in \F$
implies that $A_1\in \F$ or $A_2\in \F$. It is well known that a proper family has
the Ramsey property if and only if its dual $\F^*$ is a filter
\cite{F}.

\subsubsection{Some important families}

A subset $S$ of $\Z$ is {\it syndetic} \index{syndetic} if it has a bounded gap, i.e.
there is $N\in \N$ such that $\{i,i+1,\ldots,i+N\} \cap S \neq
\emptyset$ for every $i \in {\Z}$. The collection of all syndetic
subsets is denoted by $\F_s$. \index{$\F_s$}



The {\it upper Banach density} and {\it lower Banach density} of $S$
are \index{upper Banach density} \index{lower Banach density}
$$BD^*(S)=\limsup_{|I|\to \infty}\frac{|S\cap I|}{|I|},\ \text{and}\
BD_*(S)=\liminf_{|I|\to \infty}\frac{|S\cap I|}{|I|},$$ where $I$
ranges over intervals of $\mathbb{Z}$, while the {\it upper density}
of $S$ and the {\it lower density} of $S$ are \index{upper density}\index{lower density}
$$D^*(S)=\limsup_{n\to \infty}\frac{|S\cap [-n,n]|}{2n+1},\ \text{and}\ D_*(S)=\liminf_{n\to \infty}\frac{|S\cap [-n,n]|}{2n+1}.$$
If $D^*(S)=D_*(S)$, then we say the {\it density} of $S$ is
$D(S)=D^*(S)=D_*(S)$. \index{density}
Let $\F_{pubd}=\{S\subset \Z_+: BD^*(S)>0\}$ and
$\F_{pd}=\{S\subset \Z_+: D^*(S)>0\}$. \index{$\F_{pubd}$}\index{$\F_{pd}$}

Let $\{ b_i \}_{i\in I}$ be a finite or infinite sequence in
$\mathbb{Z}$. One defines $$FS(\{ b_i \}_{i\in I})=\Big\{\sum_{i\in
\alpha} b_i: \alpha \text{ is a finite non-empty subset of } I\Big
\}.$$ $F$ is an {\it IP set} \index{IP set} if it contains some
$FS({\{p_i\}_{i=1}^{\infty}})$, where $p_i\in\Z$. The collection of
all IP sets is denoted by $\F_{ip}$. A subset of $\Z$ is an $IP^*$-set if it intersects any $IP$-set. \index{ $IP^*$ set} It
is known that the family of all $IP^*$-sets is a filter and each
$IP^*$-set is syndetic \cite{F}. \index{$\F_{ip}$}

If $I$ is finite, then one says $FS(\{ p_i \}_{i\in I})$ is an {\em
finite IP set}. \index{finite IP set} The collection of all sets containing finite IP sets
with arbitrarily long lengths is denoted by $\F_{fip}$. \index{$\F_{fip}$}

\section{Bergelson-Host-Kra' Theorem and the proof of Theorem A(2)}\label{simpleproof}

In this section we explain how Bergelson-Host-Kra's result in \cite{BHK05} is
related to Problem B-II. First we need some definitions.

\begin{de} \index{basic $d$-step nilsequence} \index{$d$-step nilsequence}
Let $d\ge 1$ be an integer and let $X = G/\Gamma$ be a
$d$-step nilmanifold. Let $\phi$ be a continuous real (or complex)
valued function on $X$ and let $a\in G$ and $b\in X$. The sequence
$\{\phi(a^n\cdot b)\}$ is called a {\em  basic $d$-step nilsequence}. A
{\em $d$-step nilsequence} is a uniform limit of basic $d$-step
nilsequences.
\end{de}

For the definition of nilmanifolds see Chapter
\ref{section-nilsystem}.

\begin{de}
Let $\{a_n : n\in \Z\}$ be a bounded sequence. We say
that $a_n$ {\em tends to zero in uniform density}, and we write
$\text{UD-Lim}\ a_n = 0,$ if $$\lim_{N\lra+\infty}
\sup_{M\in\Z}\cfrac 1N\sum_{n=M}^{M+N-1}|a_n| = 0.$$
\end{de}

Equivalently, $\text{UD-Lim}\ a_n = 0$ if and only if for any $\ep>
0,$ the set $\{n\in \Z : |a_n| >\ep\}$ has upper Banach density
zero. Now we state their result.

\begin{thm}[Bergelson-Host-Kra]\cite[Theorem 1.9]{BHK05}\label{important}
Let $(X,\X,\mu, T )$ be an ergodic
system, let $f \in L^\infty(\mu)$ and let $d\ge 1$ be an integer.
The sequence $\{I_f(d,n)\}$ is the sum of a sequence tending to zero
in uniform density and a $d$-step nilsequence, where
\begin{equation}\label{}
  I_f(d,n)=\int f(x)f(T^nx)\ldots f(T^{dn}x)\ d\mu(x).
\end{equation}
\end{thm}

{ Note that in Theorem \ref{important} the decomposition of $\{I_f(d,n)\}$ is unique, and if $f$ is a real-valued function
then the corresponding nilsequence is also a real-valued sequence.}
By Theorem \ref{important}, for any $A\in \X$
\begin{equation}\label{noteasy}
\{I_{1_A}(d,n)\}=\{\mu(A\cap T^{-n}A\cap \ldots \cap
T^{-dn}A)\}=F_d+N, \end{equation} where $F_d$ is a $d$-step
nilsequence and $N$ tending to zero in uniform density. { Regard $F_d$
as a function $F_d: \Z\rightarrow \R$.} By \cite{HM} there is a
system $(Z, S)$ of order $d$, $x_0\in Z$ and a continuous function
$\phi\in C(Z)$ such that
\begin{equation*}
    F_d(n)=\phi(S^nx_0).
\end{equation*}
We claim that $\phi(x_0)>0$ if $\mu(A)>0$. Assume that contrary that
$\phi(x_0)\le 0$. By \cite{FK} or \cite[Theorem 6.15]{BM20} there is
$c>0$ such that
$$\{n\in \Z: \mu(A\cap T^{-n}A\cap \ldots \cap T^{-dn}A)>c\}$$ is an
$IP^*$-set. On the other hand there is a small neighborhood $V$ of
$x_0$ such that $\phi(x)< \frac{1}{2}c$ for each $x\in V$ by the
continuity of $\phi$. It is known that $N(x_0,V)$ is an $IP^*$-set
(\cite{F}) since $(Z,S)$ is distal (\cite[Ch 4, Theorem 3]{AGH} or
\cite{L}). This contradicts  (\ref{noteasy}) by the facts that the
family of $IP^*$-sets is a filter, each $IP^*$-set is syndetic and
$N(n)$ tends to zero in uniform density. That is, we have shown that
$\phi(x_0)>0$ if $\mu(A)>0$.

\medskip

For each $\ep>0$, $\{n\in\Z:
\phi(S^nx_0)>\phi(x_0)-\frac{1}{2}\ep\}$ is a Nil$_d$ Bohr$_0$-set.
Since $\{n\in\Z: |N(n)|>\frac{1}{2}\ep\}$ has zero upper Banach
density we have the following corollary

\begin{thm}\label{d-recurrence}
Let $(X,\X,\mu, T )$ be an ergodic system and $d\in \N$. Then for
all $A\in \X$ with $\mu(A)>0$ and $\ep>0$, the set $$I=\{n\in \Z:
\mu(A\cap T^{-n}A\cap \ldots \cap T^{-dn}A)>\phi(x_0)-\ep\}$$ is an
almost Nil$_d$ Bohr$_0$-set, i.e. there is some subset $M$ of $\Z$
with $BD^*(M)=0$ such that $I \Delta M$ is a Nil$_d$ Bohr$_0$-set.
\end{thm}

\noindent{\bf Proof of Theorem A(2):} It follows by Theorem
\ref{d-recurrence} that Theorem A(2) holds, since for a minimal
system $(X,T)$, each invariant measure of $(X,T)$ is fully
supported.

\section{Equivalence of Problems B-I,II,III}

In this subsection we explain why Problems
B-I,II,III are equivalent. We need Furstenberg correspondence principle.

\subsection{Furstenberg correspondence principle}\label{FCP}

Let $\F(\Z)$ denote the collection of finite non-empty subsets of
$\Z$. The following is the well known Furstenberg correspondence
principle \cite{F}. \index{Furstenberg correspondence principle}

\begin{thm}[Topological case]\label{topocase}

\begin{enumerate}

\item Let $E\subset \Z$ be a syndetic set. Then there
exist a minimal system $(X,T)$ and a non-empty open subset $U$ of
$X$ such that
\begin{equation*}
\{\a\in \F(\Z): \bigcap_{n\in \a}T^{-n}U\neq \emptyset\}\subset
\{\a\in \F(\Z): \bigcap_{n\in \a}(E-n)\neq \emptyset \}.
\end{equation*}

\item For any minimal system $(X,T)$ and any open non-empty subset $U$ of $X$, there is a syndetic
set $E$ of $\Z$ such that
\begin{equation*}
\{\a\in \F(\Z): \bigcap_{n\in \a}(E-n)\neq \emptyset \}\subset
\{\a\in \F(\Z): \bigcap_{n\in \a}T^{-n}U\neq \emptyset\}.
\end{equation*}

\end{enumerate}

\end{thm}

\begin{thm}[Measurable case]
\begin{enumerate}
  \item Let $E\subset \Z$ with $BD^*(E)>0$. Then there
  exists a measurable system $(X,\X,\mu,T)$ and $ A\in \X$ with $\mu(A)=BD^*(E)$ such that
  for all $\a\in \F(\Z)$
  \begin{equation*}
    BD^*(\bigcap_{n\in \a}(E-n))\ge \mu(\bigcap_{n\in \a}T^{-n}A).
  \end{equation*}

  \item Let $(X,\X,\mu,T)$ be a measurable system and $ A\in \X$ with $\mu(A)>0$.
  There is a subset $E$ of $\Z$ with $D^*(E)\ge \mu(A)$ such that
  \begin{equation*}
    \{\a\in \F(\Z): \bigcap_{n\in \a}(E-n)\neq \emptyset
    \}\subset \{\a\in \F(\Z): \mu(\bigcap_{n\in \a}T^{-n}A )>0\}.
  \end{equation*}

\end{enumerate}

\end{thm}

\subsection{Equivalence of Problems B-I,II,III}

Let $\F$ be the family generated by all
sets of forms $\{n\in\Z: U\cap T^{-n}U\cap \ldots \cap
T^{-dn}U\not=\emptyset\},$ with $(X,T)$ a minimal system, $U$ a
non-empty open subset of $X$. Then it is clear from the definition
that
$$\F_{Bir_d}=\F^*.$$
\begin{prop}
For any $d\in\N$ the following statements are equivalent.
\begin{enumerate}
\item For any Nil$_d$ Bohr$_0$-set $A$, there is a syndetic subset $S$ of $\Z$ with
$A\supset\{n\in \Z: S\cap (S-n)\cap\ldots\cap (S-dn)\neq
\emptyset\}$.
\item  For any Nil$_d$ Bohr$_0$-set $A$, there are a minimal system $(X,T)$ and a
non-empty open subset $U$ of $X$ with $A\supset \{n\in\Z: U\cap
T^{-n}U\cap \ldots \cap T^{-dn}U\not=\emptyset\}.$
\item $\F_{Bir_d}\subset \F^*_{d,0}$.
\end{enumerate}
\end{prop}
\begin{proof} Let $d\in\N$ be fixed. (1)$\Rightarrow$(2). Let $A$ be a Nil$_d$
Bohr$_0$-set, then there is a syndetic subset $S$ of $\Z$ with
$A\supset\{n\in \Z: S\cap (S-n)\cap\ldots\cap (S-dn)\neq
\emptyset\}$. For such $S$ using Theorem \ref{topocase}, we get that
there exist a minimal system $(X,T)$ and a non-empty open set $
U\subset X$ such that $\{n\in \Z: S\cap (S-n)\cap\ldots\cap
(S-dn)\neq \emptyset\}\supset\{n\in\Z: U\cap T^{-n}U\cap \ldots \cap
T^{-dn}U\not=\emptyset\}.$ Thus $A\supset\{n\in\Z: U\cap T^{-n}U\cap
\ldots \cap T^{-dn}U\not=\emptyset\}.$ (2)$\Rightarrow$(1) follows
similarly by the above argument. (2)$\Rightarrow$(3) follows by the
definition. (3)$\Rightarrow$(2). Since $\F_{Bir_d}\subset
\F^*_{d,0}$ and $\F_{Bir_d}=\F^*$, we have that $\F^*\subset
\F_{d,0}^*$ which implies that $\F\supset \F_{d,0}$.
\end{proof}

\begin{prop} For any $d\in\N$ the following statements are equivalent.
\begin{enumerate}
\item For any syndetic set $S$, $\{n\in \Z:
S\cap (S-n)\cap\ldots\cap (S-dn)\neq \emptyset\}$ is a Nil$_d$
Bohr$_0$-set.

\item For any minimal system $(X,T)$, and any  non-empty open subset $U$ of $X$, $\{n\in\Z: U\cap T^{-n}U\cap \ldots \cap
T^{-dn}U\not=\emptyset\}$ is a Nil$_d$ Bohr$_0$-set.

\item $\F_{Bir_d}\supset \F^*_{d,0}$.
\end{enumerate}
\end{prop}
\begin{proof} Let $d\in\N$ be fixed. (1)$\Rightarrow$(2). Let
$(X,T)$ be a minimal system and $U$ be a non-empty open set of $X$.
By Theorem \ref{topocase}, there is a syndetic set $S$ such that
$$\{n\in
\Z: S\cap (S-n)\cap\ldots\cap (S-dn)\neq \emptyset\}\subset\{n\in\Z:
U\cap T^{-n}U\cap \ldots \cap T^{-dn}U\not=\emptyset\}.$$ By (1),
$\{n\in \Z: S\cap (S-n)\cap\ldots\cap (S-dn)\neq \emptyset\}$ is a
Nil$_d$ Bohr$_0$-set, and so is $\{n\in\Z: U\cap T^{-n}U\cap \ldots
\cap T^{-dn}U\not=\emptyset\}.$ Similarly, we have
(2)$\Rightarrow$(1). (2)$\Rightarrow$(3) follows by the definition.
{ (3)$\Rightarrow$(2) follows by taking $*$ on both sides of (3).}
\end{proof}





\chapter{Nilsystems}\label{section-nilsystem}

In this chapter we recall some basic facts concerning nilpotent Lie
groups and nilmanifolds. Since in the proofs of our main results we
need to use the metric of the nilpotent matrix Lie group, we state
some basic properties related to the metric. Notice that we follow
Green and Tao \cite{GT} to define such a metric.

\section{Nilmanifolds and nilsystems}

\subsection{Nilmanifolds and nilsystems}

\subsubsection{Nilpotent groups}

Let $G$ be a group. For $g, h\in G$, we write $[g, h] =
ghg^{-1}h^{-1}$ for the commutator of $g$ and $h$ and we write
$[A,B]$ for the subgroup spanned by $\{[a, b] : a \in A, b\in B\}$.
The commutator subgroups $G_j$, $j\ge 1$, are defined inductively by
setting $G_1 = G$ and $G_{j+1} = [G_j ,G]$. Let $d \ge 1$ be an
integer. We say that $G$ is {\em $d$-step nilpotent} if $G_{d+1}$ is
the trivial subgroup. \index{$d$-step nilpotent}

\subsubsection{Nilmanifolds}

Let $G$ be a $d$-step nilpotent Lie group and $\Gamma$ a discrete
cocompact subgroup of $G$, i.e. a uniform subgroup of $G$. The
compact manifold $X = G/\Gamma$ is called a {\em $d$-step
nilmanifold}. \index{$d$-step nilmanifold} The group $G$ acts on $X$ by left translations and we
write this action as $(g, x)\mapsto gx$. The Haar measure $\mu$ of
$X$ is the unique probability measure on $X$ invariant under this
action. Let $\tau\in G$ and $T$ be the transformation $x\mapsto \tau
x$ of $X$, i.e the nilrotation induced by $\tau\in G$. Then $(X, T,
\mu)$ is called a {\em $d$-step nilsystem}. See \cite{CG,M}
for the details. \index{$d$-step nilsystem}

\subsubsection{Systems of order $d$}

We also make use of inverse limits of nilsystems and so we recall
the definition of an inverse limit of systems (restricting ourselves
to the case of sequential inverse limits). If $(X_i,T_i)_{i\in \N}$
are systems with $diam(X_i)\le M<\infty$ and $\phi_i:
X_{i+1}\rightarrow X_i$ are factor maps, the {\em inverse limit} \index{inverse limit} of
the systems is defined to be the compact subset of $\prod_{i\in
\N}X_i$ given by $\{ (x_i)_{i\in \N }: \phi_i(x_{i+1}) = x_i,
i\in\N\}$, which is denoted by $\displaystyle
\lim_{\longleftarrow}\{X_i\}_{i\in \N}$. It is a compact metric
space endowed with the distance $\rho(x, y) =\sum_{i\in \N} 1/2^i
\rho_i(x_i, y_i )$. We note that the maps $\{T_i\}$ induce a
transformation $T$ on the inverse limit.

\begin{de}\label{de-nilsystem} [Host-Kra-Maass] \cite{HKM}\index{system of order $d$}
A system $(X,T)$ is called
a {\em system of order $d$}, if it is an inverse limit of
$d$-step minimal nilsystems.
\end{de}

{\it An $\infty$-step nilsystem or a system of order $\infty$} is an
inverse limit of $d_i$-step nilsystem, see \cite{D-Y}. \index{$\infty$-step nilsystem} \index{system of order $\infty$}


\medskip
Recall that a subset $A$ of $\Z$ is a {\em Nil$_d$ Bohr$_0$-set} \index{Nil$_d$ Bohr$_0$-set}
if there exist a $d$-step nilsystem $(X,T)$, $x_0\in X$ and an open
neighborhood $U$ of $x_0$ such that $N(x_0,U)$ is contained in $A$.
As each $d$-step nilsystem is distal, so is a system of order $d$.
Note that each point in a distal system is minimal. Hence by
Definition~\ref{de-nilsystem}, it is not hard to see that a subset
$A$ of $\Z$ is a {\em Nil$_d$ Bohr$_0$-set} if and only if there
exist a $d$-step (minimal) nilsystem $(X,T)$ (or a {\em system
$(X,T)$ of order $d$}), $x_0\in X$ and an open neighborhood $U$ of
$x_0$ such that $N(x_0,U)$ is contained in $A$. { Note that here we
need the fact that
the orbit closure of any
point in a $d$-step nilsystem is a $d$-step nilmanifold
\cite[Theorem 2.21]{L}.}

\subsection{Reduction}\label{reduction}
Let $X=G/\Gamma$ be a nilmanifold. Then there exists a connected,
simply connected nilpotent Lie group $\widehat{G}$ and
$\widehat{\Gamma}\subset \widehat{G}$ a co-compact subgroup such
that $X$ with the action of $G$ is isomorphic to a submanifold
$\widetilde{X}$ of $\widehat{X}=\widehat{G}/\widehat{\Gamma}$
representing the action of $G$ in $\widehat{G}$. See \cite{L} for
more details.

Thus a subset $A$ of $\Z$ is a {\em Nil$_d$ Bohr$_0$-set} if and
only if there exist a  $d$-step nilsystem $(G/\Gamma,T)$ with $G$ is
a connected, simply connected nilpotent Lie group and $\Gamma$ a
co-compact subgroup of $G$, $x_0\in X$ and an open neighborhood $U$
of $x_0$ such that $N(x_0,U)$ is contained in $A$.

\subsection{Nilpotent Lie group and Mal'cev basis}

\subsubsection{}

We will make use of the Lie algebra $\mathfrak{g}$ of a $d$-step
nilpotent Lie group $G$ together with the exponential map $\exp: \g
\lra G$. When $G$ is a connected, simply-connected $d$-step
nilpotent Lie group the exponential map is a diffeomorphism
\cite{CG, M}. In particular, we have a logarithm map $\log: G\lra
\g$. Let $$\exp(X*Y)=\exp(X)\exp(Y), \ X,Y\in \g.$$

\subsubsection{Campbell-Baker-Hausdorff formula}

The following Campbell-Baker-Hausdorff formula (CBH formula) will be
used frequently \index{Campbell-Baker-Hausdorff formula}
\begin{equation*}
\begin{aligned}
X*Y&=\sum_{n>0}\frac{(-1)^{n+1}}{n}\sum_{p_i+q_i>0,1\le i\le
n}\frac{(\sum_{i=1}^n(p_i+q_i))^{-1}}{p_1!q_1!\ldots p_n!q_n!}\\
&\times (\ad X)^{p_1}(\ad Y)^{q_1}\ldots (\ad X)^{p_n}(\ad
Y)^{q_n-1}Y,
\end{aligned}
\end{equation*}
where $(\ad X)Y=[X,Y]$. (If $q_n=0$, the term in the sum is $\ldots
(\ad X)^{p_n-1}X$; of course if $q_n>1$, or if $q_n=0$ and $p_n>1$,
then the term is zero.) The low order nonzero terms are well known,
\begin{equation*}
\begin{aligned}
X*Y=& X+Y+\frac{1}{2}[X,Y]+\frac{1}{12}[X,[X,Y]]-\frac{1}{12}[Y,[X,Y]]\\
&-\frac{1}{48}[Y,[X,[X,Y]]]-\frac{1}{48}[X,[Y,[X,Y]]]\\
&+ (\text{ commutators in five or more terms}).
\end{aligned}
\end{equation*}

\subsubsection{}

We assume $\mathfrak{g}$ is the Lie algebra of $G$ over $\R$, and
$\exp: \g \lra G$ is the exponential map. The {\em descending central
series} of $\g$ is defined inductively by \index{descending central series}
\begin{equation*}
   \g^{(1)}=\g; \ \g^{(n+1)}=[\g,\g^{(n)}]={\rm span}\{[X,Y]:
   X\in \g, Y\in \g^{(n)}\}.
\end{equation*}
Since $\g$ is a $d$-step nilpotent Lie algebra, we have
$$ \g=\g^{(1)}\supset\g^{(2)}\supset\ldots\supset \g^{(d)} \supset\g^{(d+1)}=\{0\}.$$
We note that $$[\g^{(i)},\g^{(j)}]\subset \g^{(i+j)}, \forall i,
j\in \N.$$ In particular, each $\g^{(k)}$ is an ideal in $\g$.

\subsubsection{Mal'cev Basis}

\begin{de}(Mal'cev basis)\index{Mal'cev basis}
Let $G/\Gamma$ be an $m$-dimensional nilmanifold (i.e. $G$ is a
$d$-step nilpotent Lie group and $\Gamma$ is a uniform
subgroup of $G$) and let $G=G_1\supset \ldots\supset G_d\supset
G_{d+1}=\{e\}$ be the lower central series filtration. A basis
$\mathcal {X}= \{X_1, \ldots ,X_m\}$ for the Lie algebra $\g$ over
$\R$ is called a {\em Mal'cev basis} for $G/\Gamma$ if the following
four conditions are satisfied:

\begin{enumerate}
\item  For each $j = 0,\ldots,m-1$ the subspace $\eta_j := \text{Span}(X_{j+1}, \ldots
,X_m)$ is a Lie algebra ideal in $\g$, and hence $H_j := \exp\
\eta_j$ is a normal Lie subgroup of $G$.

\item { For every $0< i\le d$ we have $G_i = H_{l_{i-1}}$. Thus
$0=l_0<l_1<\ldots<l_{d-1}\le m-1$.}

\item Each $g\in G$ can be written uniquely as
$\exp(t_1X_1) \exp(t_2X_2)\ldots \exp(t_mX_m)$, for $t_i\in \R$.

\item $\Gamma$ consists precisely of those elements which, when written in
the above form, have all $t_i\in \Z$.
\end{enumerate}
\end{de}
Note that such a basis exists when $G$ is a connected, simply
connected $d$-step nilpotent Lie group \cite{CG, GT, M}.

\subsection{Base points}
The following proposition should be well known.

\begin{prop}\label{replace}
Let $X=G/\Gamma$ be a nilmanifold, and $T$ be a nilrotation induced
by $a\in G$. Let $x\in G$ and $U$ be an open neighborhood of $x\Gamma$
in $X$. Then there are a uniform subgroup $\Gamma_x\subset G$ and an
open neighborhood $V\subset G/\Gamma_x$ of $e\Gamma_x$ such that
$$N_T(x\Gamma,U)=N_{T'}(e\Gamma_x,V),$$
where $T'$ is a nilrotation induced by $a\in G$ in $X'=G/\Gamma_x$.
\end{prop}
\begin{proof} Let $\Gamma_x=x\Gamma x^{-1}$. Then $\Gamma_x$ is also a uniform subgroup of $G$.

Put $V=Ux^{-1}$, where we view $U$ as the collections of equivalence
classes. It is easy to see that $V\subset G/\Gamma_x$ is open, which
contains $e\Gamma_x$. Let $n\in N_T(x\Gamma,U)$ then $a^nx\Gamma\in
U$ which implies that $a^nx\Gamma x^{-1}\in Ux^{-1}=V$, i.e. $n\in
N_{T'}(e\Gamma_x,V)$. The other direction follows similarly.
\end{proof}

\section{Nilpotent Matrix Lie Group}

\subsection{}

Let $M_{d+1}(\mathbb{R})$ denote the space of all $(d+1)\times
(d+1)$-matrices with real entries. For $A=(A_{ij})_{1\le i,j\le
d+1}\in M_{d+1}(\mathbb{R})$, we define
\begin{equation}\label{Mat-eq-1}
\|A\|_\infty= \max \limits_{1\le i,j\le d+1}
|A_{ij}|.
\end{equation}
Then $\|\cdot\|_\infty$ is a norm on $M_{d+1}(\mathbb{R})$ and the norm
satisfies the inequalities
\begin{align*}
\|A+B\|_\infty\le \|A\|_\infty+\|B\|_\infty
\end{align*}
for $A,B\in M_{d+1}(\mathbb{R})$.

\subsection{}

Let ${\bf a}=(a_i^k)_{1\le k\le d, 1\le i\le d-k+1}\in
\mathbb{R}^{d(d+1)/2}$. Then corresponding to ${\bf a}$ we define
$\M({\bf a})$ with \index{$\M({\bf a})$}
{\footnotesize
$$
\M({\bf a})=\left(
  \begin{array}{cccccccc}
    1 & a_1^1 & a_1^2 & a_1^3  &\ldots & a_1^{d-1} &a_1^d \\
    0 & 1     & a_2^1 & a_2^2   &\ldots &  a_2^{d-2}&a_2^{d-1}\\
    0 & 0     &1      & a_3^1   &\ldots &  a_3^{d-3} &a_3^{d-2}\\
   \vdots & \vdots & \vdots &  \vdots & \vdots & \vdots &\vdots\\
    0 & 0    &0       &0           & \ldots & a_{d-1}^1 &a_{d-1}^2 \\
    0 & 0    &0       &0            & \ldots & 1 &a_{d}^1 \\
    0 & 0    &0       &0            & \ldots & 0  & 1
  \end{array}
\right).
$$}

\subsection{}

Let $\mathbb{G}_d$ \index{$\mathbb{G}_d$} be the (full) upper triangular group
$$\G_d=\{\M({\bf a}):a_i^k\in \mathbb{R}, 1\le k\le d, 1\le i\le d-k+1\}.$$
The group $\G_d$ is a $d$-step nilpotent group, and it is clear that
for $A\in \G_d$ there exists a unique ${\bf c}=(c_i^k)_{1\le k\le d,
1\le i\le d-k+1}\in \mathbb{R}^{d(d+1)/2}$ such that $A=\M({\bf
c})$. Let
$$\Gamma=\{\M({\bf h}): h_i^k\in \mathbb{Z}, 1\le k\le d, 1\le i\le d-k+1\}.$$
Then $\Gamma$ is a uniform subgroup of $\G_d$.

\medskip

\subsection{}

Let ${\bf a}=(a_i^k)_{1\le k\le d, 1\le i\le d-k+1}\in
\mathbb{R}^{d(d+1)/2}$ and ${\bf b}=(b_i^k)_{1\le k\le d, 1\le i\le
d-k+1}\in \mathbb{R}^{d(d+1)/2}$. If ${\bf c}=(c_i^k)_{1\le k\le d,
1\le i\le d-k+1}\in \mathbb{R}^{d(d+1)/2}$ such that $\M({\bf
c})=\M({\bf a})\M({\bf b})$, then
\begin{equation}\label{sec8-4-eq-1}
c_i^k=\sum \limits_{j=0}^k a_i^{k-j}b_{i+k-j}^j=a_i^k+(\sum
\limits_{j=1}^{k-1}a_i^{k-j}b_{i+k-j}^j)+b_i^k
\end{equation}
for $1\le k \le d$ and $1\le i \le d-k+1$, where we assume
$a_1^0=a_2^0=\ldots=a_d^0=1$ and $b_1^0=b_2^0=\ldots=b_d^0=1$.

\chapter{Generalized polynomials}\label{section-GP}

Generalized polynomials have been studied extensively, see for
example the remarkable paper by Bergelson and Leibman \cite{BL07}
and references therein. In this chapter we introduce the notions and
study the basic properties of (special) generalized polynomials
which will be used in the following chapters. Note that our
definition of the generalized polynomials is slightly different from
the usual one.

\section{Definitions}\label{genenal-poly-def}

\subsection{}

For a real number $a\in\R$, let $||a||=\inf\{|a-n|:n\in\Z\}$ and
$$\lceil{a}\rceil=\min\{m\in\Z: |a-m|=||a||\}.$$ \index{$\lceil{\cdot}\rceil$}

When studying $\F_{d,0}$ we find that the generalized polynomials
appear naturally. Here is the precise definition. Note that we use
$f(n)$ or $f$ to denote the generalized polynomials.

\subsection{Generalized polynomials}\index{generalized polynomial}

\begin{de} Let $d\in\N$. We define the {\it generalized polynomials} of
degree $\le d$ (denoted by GP$_d$) by induction. For $d=1$, GP$_1$
is the smallest collection of functions from $\Z$ to $\R$ containing $\{h_a:
a\in \R\}$ with $h_a(n)=an$ for any $n\in\Z$, which is closed under
taking $\lceil\ \rceil$, multiplying by a constant and the finite
sums.

Assume that GP$_i$ is defined for $i<d$. Then GP$_d$ is the smallest
collection of functions from $\Z$ to $\R$ containing GP$_i$ with
$i<d$, functions of the forms $$a_0n^{p_0}{\lceil{f_1(n)}
\rceil}\ldots {\lceil{f_k(n)} \rceil}$$ (with $a_0\in\R,$ $p_0\ge
0$, $k\ge 0$, $f_l\in$ GP$_{p_l}$ and $\sum_{l=0}^kp_l=d$), which is
closed under taking $\lceil\ \rceil$, multiplying by a constant and
the finite sums. Let GP$=\cup_{i=1}^\infty$GP$_i$. \index{GP}
\end{de}

For example, $a_1\lceil{a_2\lceil{a_3n}}\rceil \rceil+b_1n\in$
GP$_1$, and
$a_1\lceil{a_2n^2}\rceil+b_1\lceil{b_2n\lceil{b_3n}}\rceil
\rceil+c_1n^2+c_2n\in$ GP$_2$, where $a_i,b_i,c_i\in\R$. Note that
if $f\in$ GP then $f(0)=0$.

\subsection{Special generalized polynomials}\index{special generalized polynomial}

Since generalized polynomials are very complicated, we will specify
a subclass of them, called  the {\it special generalized
polynomials} which will be used in our proofs of the main results.
To do this, we need some notions.

For $a\in \mathbb{R}$, we define $L(a)=a$. For $a_1,a_2\in
\mathbb{R}$ we define $L(a_1,a_2)=a_1\lceil L(a_2)\rceil$.
Inductively, for $a_1,a_2,\ldots,a_{\ell}\in \mathbb{R}$ ($\ell\ge
2$) we define
\begin{equation}\label{eq-de-L}\index{$L(a_1,a_2,ccots, a_l)$}
L(a_1,a_2,\ldots,a_{\ell})=a_1\lceil
L(a_2,a_3,\ldots,a_\ell)\rceil.
\end{equation}

For example, $L(a_1,a_2,a_3)=a_1\lceil a_2\lceil a_3\rceil\rceil$.

\medskip

We give now the precise definition of special generalized
polynomials.

\begin{de} For $d\in\N$ we define {\it special generalized polynomials of degree} $\le d$,
denoted by SGP$_d$ as follows.  SGP$_d$ is the collection of
generalized polynomials of the forms
$L(n^{j_1}a_1,\ldots,n^{j_\ell}a_\ell)$, where $1\le \ell \le d,
a_1,\ldots,a_\ell\in \mathbb{R}, j_1,\ldots,j_\ell\in \mathbb{N}
\text{ with } \sum_{t=1}^\ell j_t\le d.$
\end{de}

Thus SGP$_1=\{an : a\in\R\},$ SGP$_2=\{an^2,
bn\lceil{cn}\rceil,en:a,b,c,e\in\R\}$ and SGP$_3$=SGP$_2\cup
\{an^3,an\lceil bn^2\rceil, an^2\lceil bn\rceil, an\lceil bn\lceil
cn\rceil\rceil: a,b,c\in\R\}$. \index{SGP}

\subsection{$\F_{GP_d}$ and $\F_{SGP_d}$} \index{$\F_{GP_d}$}\index{$\F_{SGP_d}$}

Let $\F_{GP_d}$  be the family generated by the sets of forms
$$\bigcap_{i=1}^k\{n\in\Z:P_i(n)\ (\text{mod}\ \Z)\in (-\ep_i,\ep_i)
\},$$ where $k\in\N$, $P_i\in GP_d$, and $\ep_i>0$, $1\le i\le k$.
Note that $P_i(n)\ (\text{mod}\ \Z)\in (-\ep_i,\ep_i)$ if and only
if $||P_i(n)||<\ep_i.$

Let $\F_{SGP_d}$ be the family generated by the sets of forms
$$\bigcap_{i=1}^k\{n\in\Z:P_i(n)\ (\text{mod}\ \Z)\in (-\ep_i,\ep_i)
\},$$ where $k\in\N$, $P_i\in SGP_d$, and $\ep_i>0$, $1\le i\le k$.
Note that from the definition both $\F_{GP_d}$ and $\F_{SGP_d}$ are
filters; and $\F_{SGP_d}\subset \F_{GP_d}.$

\section{Basic properties of generalized polynomials}

\subsection{}

The following lemmas lead a way to simplify the generalized
polynomials. For $f\in$ GP we let $f^*=-\lceil{f}\rceil.$
\begin{lem}\label{simple2} Let $c\in\R$ and $f_1, \ldots, f_k\in GP$ with $k\in \N$. Then
$$c\lceil{f_1}\rceil\ldots \lceil{f_k}\rceil=c(-1)^k\prod_{i=1}^k(f_i-\lceil{f_i}\rceil)
-c(-1)^k\sum_{i_1,\ldots,i_k\in
\{1,*\}\atop{(i_1,\ldots,i_k)\not=(*,\ldots,*)}}f_{1}^{i_1}\ldots
f_{k}^{i_k}.$$
In particular, if $k=2$ one has that
$$c\lceil{f_1}\rceil\lceil{f_2}\rceil=cf_1\lceil{f_2}\rceil-cf_1f_2+
cf_2\lceil{f_1}\rceil+c(f_1-\lceil{f_1}\rceil)(f_2-\lceil{f_2}\rceil).$$
\end{lem}
\begin{proof}
Expanding $\prod_{i=1}^k(f_i-\lceil{f_i}\rceil)$ we get that
$$\prod_{i=1}^k(f_i-\lceil{f_i}\rceil)=\sum_{i_1,\ldots,i_k\in \{1,*\}}f_{1}^{i_1}\ldots f_{k}^{i_k}.$$
So we have $$c\lceil{f_1}\rceil\ldots
\lceil{f_k}\rceil=c(-1)^k\prod_{i=1}^k(f_i-\lceil{f_i}\rceil)-c(-1)^k\sum_{i_1,\ldots,i_k\in
\{1,*\}\atop{(i_1,\ldots,i_k)\not=(*,\ldots,*)}}f_{1}^{i_1}\ldots
f_{k}^{i_k}.$$
\end{proof}

Let $c=1$ in Lemma \ref{simple2} we have

\begin{lem}\label{simple3} Let $f_1,f_2, \ldots,f_k\in GP$. Then
$$f_1\lceil{f_2}\rceil\ldots\lceil{f_k}\rceil =(-1)^{k-1}\prod_{i=1}^k(f_i-\lceil{f_i}\rceil)
+(-1)^k\sum_{i_1,\ldots,i_k\in
\{1,*\}\atop{(i_1,\ldots,i_k)\not=(1,*,\ldots,*)}}f_{1}^{i_1}\ldots
f_{k}^{i_k}.$$

In particular, if $k=2$ one has that
$$f_1\lceil{f_2}\rceil =\lceil{f_1}\rceil\lceil{f_2}\rceil +f_1f_2-
f_2\lceil{f_1}\rceil-(f_1-\lceil{f_1}\rceil)(f_2-\lceil{f_2}\rceil).$$
\end{lem}

Let $k=1$ in Lemma \ref{simple2} we have

\begin{lem}\label{simple1} Let $c\in\R$ and $f\in GP$. Then $c\lceil{f}\rceil=cf-c(f-\lceil{f}\rceil)$.
\end{lem}


\subsection{}

In the next subsection we will show that $\F_{GP_d}=\F_{SGP_d}.$ To
do this we use induction. To make the proof clear, first we give
some results under the assumption
\begin{equation}\label{firsta}
\F_{GP_{d-1}}\subset\F_{SGP_{d-1}}.
\end{equation}
\begin{de}\label{de-ws}\index{$\mathcal{SW}_r$} \index{$\mathcal{W}_r$}
Let $r\in \mathbb{N}$ with $r\ge 2$. We define
$$\mathcal{SW}_r=\{ \prod \limits_{i=1}^\ell (w_i(n)-\lceil w_i(n) \rceil):
\ell\ge 2, r_i\ge 1, w_i(n)\in GP_{r_i} \text{ and } \sum
\limits_{i=1}^\ell r_i\le r\}$$ and
$$\mathcal{W}_r=\mathbb{R}\text{-Span}\{\mathcal{SW}_r\}, $$
that is,
$$\mathcal{W}_r=\{ \sum \limits_{j=1}^\ell a_jp_j(n): \ell\ge 1, a_j\in \mathbb{R}, p_j(n)\in
\mathcal{SW}_r \text{ for each }j=1,2,\ldots,\ell\}.$$
\end{de}

\begin{lem}\label{lem1-1}
Under the assumption \eqref{firsta}, for any $p(n)\in
\mathcal{W}_d$ and $\epsilon>0$ one has
$$\{ n\in \mathbb{Z}:p(n) \ (\text{mod}\ \Z)\in (-\epsilon,\epsilon)\}\in \F_{SGP_{d-1}}.$$
\end{lem}

\begin{proof}
Since $\F_{SGP_d}$ is a filter, it is sufficient to show that for
any $p(n)=aq(n)$ and $\frac{1}{2}>\delta>0$ with $q(n)\in
\mathcal{SW}_d$ and $a\in \mathbb{R}$,
$$\{ n\in \mathbb{Z}:p(n)\ (\text{mod}\ \Z) \in (-\delta,\delta)\}\in \F_{SGP_{d-1}}.$$
Note that as $q(n)\in \mathcal{SW}_d$, there exist $\ell\ge 2,
r_i\ge 1, w_i(n)\in GP_{r_i} \text{ and } \sum \limits_{i=1}^\ell
r_i\le d$ such that $q(n)=\prod \limits_{i=1}^\ell (w_i(n)-\lceil
w_i(n) \rceil)$. Since $\ell\ge 2$, one has $r_1\le d-1$ and so
$w_1(n)\in GP_{d-1}$. By the assumption \eqref{firsta}, $\{n\in
\mathbb{Z}: w_1(n)\ (\text{mod}\ \Z)\in
(-\frac{\delta}{1+|a|},\frac{\delta}{1+|a|}) \}\in \F_{SGP_{d-1}}$.
By the inequality $|p(n)|\le |a||w_1(n)-\lceil w_1(n) \rceil|$ for
$n\in \mathbb{Z}$, we get that
\begin{align*}
\{ n\in \mathbb{Z}:p(n)\ (\text{mod}\ \Z) \in (-\delta,\delta)\}
&\supset
\{n\in \mathbb{Z}: |w_1(n)-\lceil w_1(n) \rceil|\in (-\tfrac{\delta}{1+|a|},\tfrac{\delta}{1+|a|})\}\\
&=\{n\in \mathbb{Z}: w_1(n)\ (\text{mod}\ \Z)\in
(-\tfrac{\delta}{1+|a|},\tfrac{\delta}{1+|a|})\}.
\end{align*}
Thus $\{ n\in \mathbb{Z}:p(n)\ (\text{mod}\ \Z) \in
(-\delta,\delta)\} \in \F_{SGP_{d-1}}$ since $\{n\in \mathbb{Z}:
w_1(n)\ (\text{mod}\ \Z)\in
(-\frac{\delta}{1+|a|},\frac{\delta}{1+|a|})\}\in \F_{SGP_{d-1}}$.
\end{proof}

\begin{de}\index{$\simeq_r$}
Let $r\in \mathbb{N}$ with $r\ge 2$. For $q_1(n),q_2(n)\in GP_{r}$
we define
$$q_1(n)\simeq_r q_2(n)$$
if there exist $h_1(n)\in GP_{r-1}$ and $h_2(n)\in \mathcal{W}_r$
such that
$$q_2(n)=q_1(n)+h_1(n)+h_2(n) \ (\text{mod} \ \Z)$$ for all $n\in \mathbb{Z}$.
\end{de}

\begin{lem} \label{lem8-13}
Let $p(n)\in GP_r$ and $q(n)\in GP_t$ with $r,t\in \mathbb{N}$.
\begin{enumerate}
\item $p(n)\lceil q(n)\rceil \simeq_{r+t} (p(n)-\lceil p(n) \rceil) q(n)$.

\item If $q_1(n),q_2(n),\ldots,q_k(n)\in GP_t$ such that $q(n)= \sum_{i=1}^k q_i(n)$, then
$$p(n)\lceil q(n)\rceil \simeq_{r+t} \sum_{i=1}^k p(n)\lceil q_i(n)\rceil.$$
\end{enumerate}
\end{lem}

\begin{proof}
(1) follows from Lemma \ref{simple3} and (2) follows from (1).
\end{proof}

\begin{de}\index{$GP_r'$}
For $r\in \mathbb{N}$, we define
$${GP_r'}=\{p\in GP_r: \{ n\in \mathbb{Z}: p(n)\ (\text{mod}\ \Z)\in
(-\epsilon,\epsilon)\}\in \F_{SGP_r}\text{ for any }\epsilon>0\}.$$
\end{de}

\begin{prop}\label{useful}
Let $r,k\in \N$.
\begin{enumerate}
\item For $p(n)\in GP_r$, $p(n)\in GP'_r$ if and only if $-p(n)\in
GP'_r$.

\item If $p_1(n),p_2(n),\ldots,p_k(n)\in GP'_r$ then
$$p(n)=p_1(n)+p_2(n)+\ldots+p_k(n)\in GP'_r.$$

\item $\F_{GP_d}\subset\F_{SGP_d}$ if and only if $GP'_d=GP_d$.
\end{enumerate}
\end{prop}
\begin{proof} (1) can be verified directly. (2) follows from the fact
that for each $\ep>0$, $\{n\in\Z: p(n)\ (\text{mod}\ \Z)\in
(-\epsilon,\epsilon)\}\supset \cap_{i=1}^k \{n\in\Z: p_i(n)\
(\text{mod}\ \Z)\in (-\ep/k,\ep/k)\}$. (3) follows from the
definition of $GP'_d$.
\end{proof}

\begin{lem}\label{lem-8-equi} Let $p(n),q(n)\in GP_d$ with $p(n)\simeq_d q(n)$. Under the assumption \eqref{firsta},
$p(n)\in GP'_d$ if and only if $q(n)\in GP'_d$.
%
\end{lem}
\begin{proof} It follows from Lemma \ref{lem1-1} and the fact that $\F_{SGP_d}$ is a filter.
\end{proof}






\subsection{$\F_{GP_d}=\F_{SGP_d}$}


\begin{thm}\label{gpsgp}
$\F_{GP_d}=\F_{SGP_d}$ for each $d\in \N$.
\end{thm}
\begin{proof} It is easy to see that $\F_{SGP_d}\subset \F_{GP_d}.$
So it remains to show $\F_{GP_d}\subset \F_{SGP_d}.$ That is, if
$A\in\F_{GP_d}$ then there is $A'\in\F_{SGP_d}$ with $A\supset A'$.
We will use induction to show the theorem.

Assume first $d=1$. In this case we let $GP_1(0)=\{g_a:a\in\R\}$,
where $g_a(n)=an$ for each $n\in\Z$. Inductively if
$GP_1(0),\ldots,GP_1(k)$ have been defined then $f\in GP_1(k+1)$ if
and only if $f\in GP_1\setminus  (\bigcup_{j=0}^kGP_1(j))$ and there
are $k+1$ $\lceil {\ \ } \rceil$ in $f$. It is clear that
$GP_1=\cup_{k=1}^\infty GP_1(k)$. If $f\in GP_1(0)$ then it is clear
that $f\in GP'_1$. Assume that $GP_1(0), \ldots, GP_1(k)\subset
GP'_1$ for some $k\in \mathbb{Z}_+$.

Let $f\in GP_1(k+1)$. We are going to show that $f\in GP_1'$. If
$f=f_1+f_2$ with $f_1,f_2\in \bigcup_{i=0}^k GP_1(i)$, then by the
above assumption and Proposition \ref{useful} we conclude that $f\in
GP_1'$. The remaining case is $f=c\lceil f_1 \rceil +f_2$ with $c\in
\mathbb{R}\setminus \{0\}$, $f_1\in GP_1(k)$,  and $f_2\in GP_1(0)$.
By Proposition \ref{useful} and the fact $GP_1(0)\subset GP_1'$,
$f\in GP_1'$ if and only if $c \lceil{f_1}\rceil \in GP_1'$. So it
remains to show $c\lceil f_1 \rceil\in GP'_1$. By Lemma
\ref{simple1} we have $c\lceil f_1\rceil=cf_1-c(f_1-\lceil f_1
\rceil)$. It is clear that $cf_1\in GP_1(k)\subset GP_1'$ since
$f_1\in GP_1(k)\subset GP_1'$. For any $\ep>0$ since
\begin{equation*}
\{n\in\Z:||-c(f_1(n)-\lceil{f_1(n)}\rceil)||<\ep\}\supset
\Big\{n\in\Z: ||f_1(n)||<\tfrac{\ep}{1+|c|}\Big\},
\end{equation*}
it implies that $-c(f_1-\lceil{f_1}\rceil)\in GP_1'$. By
Proposition \ref{useful} again we conclude that
$c\lceil{f_1}\rceil\in GP'_1$. Hence $f\in GP_1'$.
Thus $GP_1\subset GP_1'$ and we
are done for the case $d=1$ by Proposition \ref{useful} (3).

\medskip
Assume that we have proved $\F_{GP_{d-1}}\subset \F_{SGP_{d-1}}\
d\ge 2$, i.e. the assumption (\ref{firsta}) holds.
We define $GP_d(k)$ with $k=0,1,2,\ldots.$ First $f\in GP_d(0)$ if
and only if there is no $\lceil {\  \ } \rceil$ in $f$, i.e. $f$ is
the usual polynomial of degree $\le d$.  Inductively if
$GP_d(0),\ldots,GP_d(k)$ have been defined then $f\in GP_{k+1}$ if
and only if $f\in GP_d\setminus  (\bigcup_{j=0}^kGP_d(j))$ and there are $k+1$
$\lceil {\ \ } \rceil$ in $f$. It is clear that
$GP_d=\cup_{k=0}^\infty GP_d(k)$. We now show $GP_d(k)\subset GP_d'$ by induction on $k$.

Let $f$ be an ordinary polynomial of degree $\le d$. Then
$f(n)=a_0n^d+f_1(n)\simeq_d a_0n^d$ with $f_1\in GP_{d-1}$. By Lemma \ref{lem-8-equi},
$f\in GP_d'$ since $a_0n^d\in \text{SGP}_d\subset GP_d'$. This shows $GP_d(0)\subset GP_d'$.
Now assume that for some
$k\in \mathbb{Z}_+$ we have proved
\begin{equation}\label{aye2}
\bigcup_{i=0}^k GP_d(i)\subset GP_d'.
\end{equation}
Let $f\in GP_d(k+1)$. We are going to show that $f\in GP_d'$. If
$f=f_1+f_2$ with $f_1,f_2\in \bigcup_{i=0}^k GP_d(i)$, then by the
assumption (\ref{aye2}) and Proposition \ref{useful} (2) we conclude
that $f\in GP_d'$. The remaining case is that $f$ can be expressed
as the sum of a function in $GP_d(0)$ and a function $g\in
GP_d(k+1)$ having the form of
\begin{enumerate}
\item $g=c\lceil{f_1}\rceil\ldots \lceil{f_l}\rceil$ with $c\neq 0$, $l\ge 1$ or

\item $g=g_1(n)\lceil{g_2(n)}\rceil\ldots \lceil{g_l(n)}\rceil$ for any $n\in\Z$ with
$g_1(n)\in SGP_r$ and $r<d$.
\end{enumerate}
Since $GP_d(0)\subset GP_d'$, $f\in GP_d'$ if and only if $g\in GP_d'$ by Proposition \ref{useful}.
It remains to show that $g\in GP_d'$. There are two cases.

\medskip
\noindent Case (1): $g=c\lceil{f_1}\rceil\ldots \lceil{f_l}\rceil$ with $c\neq 0$, $l\ge 1$.

If $l=1$, then $g=c\lceil{f_1}\rceil$ with $f_1\in GP_d(k)$.
By Lemma \ref{simple1} we have
$c\lceil f_1\rceil=cf_1-c(f_1-\lceil f_1 \rceil)$. It is
clear that $cf_1\in GP_d(k)\subset GP_d'$ since $f_1\in GP_d(k)\subset GP_d'$. For any $\ep>0$ since
\begin{equation*}
\{n\in\Z:||-c(f_1(n)-\lceil{f_1(n)}\rceil)||<\ep\}\supset
\Big\{n\in\Z: ||f_1(n)||<\tfrac{\ep}{1+|c|}\Big\},
\end{equation*}
it implies that $-c(f_1-\lceil{f_1}\rceil)\in GP_d'$. By
Proposition \ref{useful} again we conclude that
$g=c\lceil{f_1}\rceil\in GP'_d$.

If $l\ge 2$, using Lemmas \ref{simple2} and \ref{lem1-1} we get that
$$c\lceil{f_1}\rceil\ldots \lceil{f_l}\rceil\simeq_d -c(-1)^l\sum_{i_1,\ldots,i_l\in
\{1,*\}\atop{(i_1,\ldots,i_l)\not=(*,\ldots,*)}}f_{1}^{i_1}\ldots
f_{l}^{i_l}.$$
Since each term of the right side is in $GP_d(k)$,  $g\in GP_d'$  by Lemma \ref{lem-8-equi},
the assumption
(\ref{aye2}) and Proposition \ref{useful} (2).

\medskip

\noindent Case (2): $g=g_1(n)\lceil{g_2(n)}\rceil\ldots \lceil{g_l(n)}\rceil$ for any $n\in\Z$ with
$g_1\in SGP_r$ and $1\le r<d$.

In this case using Lemmas \ref{simple3} and
\ref{lem1-1} we get that
$$g_1\lceil{g_2}\rceil\ldots\lceil{g_l}\rceil\simeq_d (-1)^l\sum_{i_1,\ldots,i_l\in
\{1,*\}\atop{(i_1,\ldots,i_l)\not=(1,*,\ldots,*),(*,*,\ldots,*)}}g_{1}^{i_1}\ldots
g_{l}^{i_l}.$$ Assume $i_1,\ldots,i_l\in \{1,*\}$ with
$(i_1,\ldots,i_l)\not=(1,*,\ldots,*),(*,*,\ldots,*)$. If there are at least two $1$ appearing in $(i_1,i_2,\ldots,i_l)$, then $(-1)^\ell
g_{1}^{i_1}\ldots g_{l}^{i_\ell}\in \bigcup_{i=0}^k GP_d(i)$. Hence
$$(-1)^\ell g_{1}^{i_1}\ldots g_{l}^{i_\ell}\in GP_d'$$ by
the assumption (\ref{aye2}). The remaining situation is that $i_1=*$
and there is exact one 1 appearing in $(i_2, \ldots, i_l)$. In this
case,
$(-1)^\ell g_{1}^{i_1}\ldots g_{l}^{i_\ell}\in GP_d(k+1)$ is the
finite sum of the forms $a_1n^{t_1}\lceil{h_1(n)}\rceil \ldots
\lceil{h_{l_1'}(n)}\rceil$ with $t_1\ge 1$ and $h_1(n)=g_1(n)$; or
the forms $c\lceil{h_l}\rceil \ldots \lceil{h_{l_1}}\rceil$ or terms
in $GP_d'$.

If the term has the form $a_1n^{t_1}\lceil{h_1(n)}\rceil \ldots
\lceil{h_{l_1'}(n)}\rceil$ with $t_1\ge 1$ and $h_1(n)=g_1(n)$, we
let
$g_1^{(1)}(n)=a_1n^{t_1}\lceil{h_1(n)}\rceil=a_1n^{t_1}\lceil{g_1(n)}\rceil\in
SGP_{r_1}$. It is clear $d\ge r_1>r$. If $r_1=d$, then
$a_1n^{t_1}\lceil{h_1(n)}\rceil \ldots
\lceil{h_{l_1'}(n)}\rceil=g_1^{(1)}(n)\in GP_d'$ since $SGP_d\subset GP_d'$. If $r_1<d$,
then we write
$$a_1n^{t_1}\lceil{h_1(n)}\rceil\ldots
\lceil{h_{l_1'}(n)}\rceil=g_1^{(1)}(n)\lceil{g_2^{(1)}(n)}\rceil\ldots\lceil{g_{l_1}^{(1)}(n)}\rceil.$$
By using  Case (1) we conclude that

$g\simeq_d$ finite sum of the forms
$g_1^{(1)}(n)\lceil{g_2^{(1)}(n)}\rceil\ldots\lceil{g_{l_1}^{(1)}(n)}\rceil$
and terms in $GP_d'$.

\medskip
Repeating the above process finitely many time (at most $k+1$-times)
we get that $g\simeq_d$ finite sum of terms in $GP_d'$. Thus $g\in GP_d'$ by Lemma \ref{lem-8-equi}
and Proposition \ref{useful} (2). The proof is now completed.
\end{proof}




\chapter{Nil Bohr$_0$-sets and generalized polynomials: Proof of Theorem B}

In this chapter for a given $d\in\N$ we investigate the relationship
between the family of all Nil$_d$ Bohr$_0$-sets and
the family generalized by all generalized polynomials of order $\le
d$, i.e we will prove Theorem B.

\section{Proof of Theorem B(1)} \label{section-proof1}

In this section, we will prove Theorem B(1), i.e. we will show that
if $A\in \F_{d,0}$ then there are $k\in\N$, $P_i\in GP_d$ ($1\le
i\le k$) and $\ep_i>0$ such that
$$A\supset \bigcap_{i=1}^k\{n\in\Z:P_i(n)\ (\text{mod}\ \Z)\in (-\ep_i,\ep_i) \}.$$

We remark that by Section \ref{reduction}, it is sufficient to
consider the case when the group $G$ is a connected,
simply-connected $d$-step nilpotent Lie group.

\subsection{Notations}

Let $X=G/\Gamma$ with $G$ a connected, simply-connected $d$-step
nilpotent Lie group, $\Gamma$ a uniform subgroup. Let $T:X\lra X$ be
the nilrotation induced by $a\in G$.

\medskip
Let $\mathfrak{g}$ be the Lie algebra of $G$ over $\R$, and
let $\exp: \g \lra G$ be the exponential map. Consider
$$ \g=\g^{(1)}\supset\g^{(2)}\supset \ldots\supset \g^{(d)} \supset\g^{(d+1)}=\{0\}.$$
Notice that $$[\g^{(i)},\g^{(j)}]\subset \g^{(i+j)}, \ \forall i,j\in \N.$$ There is a
Mal'cev basis $\mathcal {X}= \{X_1, \ldots ,X_m\}$ for $\g$ with

\begin{enumerate}
\item  For each $j = 0,\ldots, m-1$ the subspace $\eta_j:=\text{Span}(X_{j+1}, \ldots
,X_m)$ is a Lie algebra ideal in $\g$, and hence $H_j := \exp\
\eta_j$ is a normal Lie subgroup of $G$.

\item { For every $0<i\le d$ we have $G_i = H_{l_{i-1}}$, where $0=l_0<l_1<\ldots<l_{d-1}<l_d=m$. }

\item Each $g\in G$ can be written uniquely as
$\exp(t_1X_1) \exp(t_2X_2)\ldots \exp(t_mX_m)$, for $t_i\in \R$.

\item $\Gamma$ consists precisely of those elements which, when written in
the above form, have all $t_i\in \Z$,
\end{enumerate}
where $G=G_1$, $G_{i+1}=[G_i,G]$ with $G_{d+1}=\{e\}$. { Notice that
$\text{Span}$-$\{X_{l_i+1},\ldots,X_{m}\}=\g^{(i+1)}$ for
$i=0,1,\ldots,d-1$.}

\begin{de} Let $\{X_1,\ldots,X_m\}$ be a Mal'cev bases for
$G/\Gamma$. Assume that $P=P(u_1,\ldots,u_m)$ is a polynomial.
Define the {\it weighted degree} $o(u_i)$ of $u_i$ to be the largest integer $k$ such that $X_i$ is contained in
$\g^{(k)}$, i.e. $o(u_i)=j$ if $l_{j-1}+1\le i\le l_j$, $ 1\le j\le
d$. The weighted degree of a monomial in $u_i'$s is the sum of the weighted degree of each term,
i.e. $o(u_1^{k_1}\ldots u_m^{k_m})=\sum_{i=1}^m k_io(u_i).$ As usual the weighted degree of $P$
is the maximum of the weighted degrees of monomials in $P$.
\end{de}


\subsection{Some lemmas}

We need several lemmas. Note that if
$$\exp(t_1X_1)\ldots\exp(t_mX_m)=\exp(u_1X_1+\ldots+u_mX_m)$$ it is
known that \cite{CG, M} each $t_i$ is a polynomial of $u_1,\ldots,
u_m$ and each $u_i$ is a polynomial of $t_1,\ldots, t_m$. For our
purpose we need to know the precise degree of the polynomials.

\begin{lem}\label{d-step} Let $\{X_1,\ldots,X_m\}$ be a Mal'cev bases for
$G/\Gamma$. Assume that
$$\exp(t_1X_1)\ldots\exp(t_mX_m)=\exp(u_1X_1+\ldots+u_mX_m).$$
Then we have

\begin{enumerate}

\item Each $u_i$ is  a polynomial in $t_j$'s with no constant term such that
the weighted degree of the polynomial is no bigger than that of $u_i$ and
the ordinary degree 1 part of this polynomial is exactly $t_i$ (i.e. $u_i=t_i$ for $1\le i\le l_1$ and if $l_{j-1}+1\le i\le l_j,\ 2\le
j\le d$ then
$u_i=t_i+\sum c_{k_1,\ldots, k_m,i}t_{1}^{k_1}\ldots t_{m}^{k_m},$ where the sum is over all
$0\le k_1,\ldots,k_{m}\le m$ with $\sum_{j=1}^m k_jo(t_j)\le o(u_i)$ and there are at least two $j$'s with $k_j\not=0$).

\item Each $t_i$ is  a polynomial in $u_j$'s with no constant term such that
the weighted degree of the polynomial is no bigger than that of $t_i$ and
the ordinary degree 1 part of this polynomial is exactly $u_i$.


\end{enumerate}
\end{lem}
\begin{proof}
(1). It is easy to see that if $m=1$ then $d=1$ and (1) holds. So we
may assume that $m\ge 2$. For $s\in
\{0,1,\ldots,m\}^{\{1,\ldots,m\}}$, let $\{i_1<\ldots<i_n\}$ be the
collection of $p's$ with $s(p)\not=0$. Let
$$X_{s}=[X_{s(i_1)}, [X_{s(i_2)}, \ldots, [X_{s(i_{n-1})},X_{s(i_n)}]]].$$ For
each $0\le p\le m$ let $k_p(s)$ be the number of $p's$ appearing in
$s(1),\ldots, s(m)$ (as usual, the cardinality of the empty set is defined as $0$).
Using the CBH formula $m-1$ times and the condition
$\g^{(d+1)}=\{0\}$ it is easy to see that $(t_1X_1)*\ldots*(t_mX_m)$
is the sum of $\sum_{i=1}^mt_iX_i$ and the terms $$constant \times
t_{q_1}\ldots t_{q_n}[X_{q_1}, [X_{q_2}, \ldots,
[X_{q_{n-1}},X_{q_n}]]],\ m\ge n\ge 2,$$ i.e.
$\exp(t_1X_1)\ldots\exp(t_mX_m)$ can be written as
$$\exp(\sum_{j=1}^mt_jX_j+\sum c'_{s}t_1^{k_1(s)}\ldots t_m^{k_m(s)}X_s),$$ where
the sum is over all  $s\in \{0, 1,\ldots, m\}^{\{1,\ldots,m\}}$
and there are at least two $j$'s with $s(j)\not=0$.
Note that $X_{s}\in \g^{(\sum_{j=1}^m k_j(s)o(t_j))}.$ Let
$X_s=\sum_{j=1}^m c'_{s,j}X_j$. Thus, $c'_{s,1}, \ldots, c_{s,i}'=0$
if $\sum_{j=1}^m k_j(s)o(t_j)>o(t_i)$. Hence, $u_i=t_i$ for $1\le i\le
l_1$ and if $l_{j-1}+1\le i\le l_j,\ 2\le j\le d$ then the
coefficient of $X_i$ is
$$u_i=t_i+\sum c_{k_1,\ldots, k_m,i}t_{1}^{k_1}\ldots t_{m}^{k_m},$$ where the sum is over all
$0\le k_1,\ldots,k_{m}\le m$ with $\sum_{j=1}^m k_jo(t_j)\le o(u_i)$ and there are at least two $j$'s with $k_j\not=0$.

Note that when ${k_1o(t_1)+\ldots+ k_{m}o(t_m)}\le o(u_i)$ and there are at least two
$j$'s with $k_j\not=0$, we have that $k_i=k_{i+1}=\ldots=k_m=0$ and some other
restrictions. For example, when $l_1+1\le i\le l_2$,
$t_{1}^{k_1}\ldots t_{m}^{k_m}=t_{i_1}t_{i_2}$ with $1\le i_1,i_2\le
l_1$; and when $l_2+1\le i\le l_3$, $t_{1}^{k_1}\ldots
t_{m}^{k_m}=t_{i_1}t_{i_2}t_{i_3}$ with $1\le i_1,i_2,i_3\le l_1$ or
$t_{i_1}t_{i_2}$ with $1\le i_1\le l_1$ and $l_1+1\le i_2\le l_2$.

\medskip

(2) It is easy to see that $t_i=u_i$ for $1\le i\le l_1$.
If $d=1$ (2) holds, and thus we assume that $d\ge 2$. We show (2) by
induction. We assume that
\begin{equation}\label{tp-eq}
t_p=u_p+\sum d_{k_1',\ldots, k_m',p}u_1^{k_1'}\ldots u_m^{k_m'},
\end{equation}
where the sum is over all $0\le k_1',\ldots,k_{m}'\le m$ with $\sum_{j=1}^m k_j'o(u_j)\le o(t_p)$
and there are at least two $j$'s with $k_j'\not=0$ for all $p$ with $l_1+1\le
p\le i.$

Since
$$u_{i+1}=t_{i+1}+\sum c_{k_1,\ldots,
k_m,i+1}t_{1}^{k_1}\ldots t_{m}^{k_m},$$ we have that

$$t_{i+1}=u_{i+1}-\sum c_{k_1,\ldots,
k_m,i+1}t_{1}^{k_1}\ldots t_{m}^{k_m},$$
where the sum is over all $0\le k_1,\ldots,k_{m}\le m$ with $\sum_{j=1}^m k_jo(t_j)\le o(t_{i+1})$
and there are at least two $j$'s with $k_j\not=0$.

Since $o(t_{i+1})\le o(t_i)+1$ and there are at least two $j$'s with $k_j\not=0$, we have that if
$k_1o(u_1)+\ldots+ k_{m}o(u_m)\le o(t_{i+1})$ then $k_jo(u_j)\le o(t_i)$ for
each $1\le j\le m$, which implies that $k_{i+1},\ldots,k_m=0$. By
the induction each $t_p$ ($1\le p\le i$) is a polynomial of
$u_1,\ldots,u_m$ of the weighted degree at most $o(t_p)$ (see
Equation \eqref{tp-eq}) thus
$$\sum  c_{k_1,\ldots, k_m,i+1}t_{1}^{k_1}\ldots
t_{m}^{k_m}$$ is a polynomial of $u_1,\ldots,u_m$ of the weighted degree at most
$\sum_{p=1}^m k_p o(u_p)=\sum_{p=1}^m k_p o(t_p)  \le o(t_{i+1}).$  Rearranging the coefficients we
get (2). Note that there are at least two $j$'s with $k_j\not=0$.
\end{proof}


\begin{lem}\label{product}
Assume that
$$x=\exp(x_1X_1+\ldots+x_mX_m)\ \text{ and}\
y=\exp(y_1X_1) \ldots \exp(y_mX_m).$$ Then
$$xy^{-1}=\exp(\sum_{i=1}^{l_1}(x_i-y_i)X_i+\sum_{i=l_1+1}^m((x_i-y_i)+P_{i,1}(\{y_p\})
+P_{i,2}(\{x_p\},\{y_p\}))X_i),$$ where $P_{i,1}(\{y_p\}),
P_{i,2}(\{x_p\},\{y_p\})$ are polynomials of the weighted degree at most $o(y_i)$ for each $l_1+1\le i\le m$. To be precise, we have
\begin{equation}\label{Polyi1}
P_{i,1}(\{y_p\})= -\sum c_{k_1',\ldots, k_m',i}y_{1}^{k_1'}\ldots y_{m}^{k_m'},
\end{equation}
the sum is over all $0\le k_1',\ldots,k_{m}'\le m$ with $\sum_{j=1}^m k_j'o(y_j)\le o(y_i)$
and there are at least two $j$'s with $k_j'\not=0$ for all $i$ with $l_1+1\le
i\le m.$ Moreover,
\begin{equation}\label{Polyi2}
P_{i,2}(\{x_p\},\{y_p\})=\sum
e_{k_1,\ldots,k_m}^{k_1',\ldots,k_m'}x_{1}^{k_1}\ldots
x_{m}^{k_m}y_{1}^{k_1'}\ldots y_{m}^{k_m'},
\end{equation} the sum is over all $0\le k_j, k_j'\le m$ with
$\sum_{j=1}^m (k_j+k_j')o(y_j)\le o(y_i)$,
and there are at least one $j$ with $k_j\not=0$, one $j$ with $k_j'\not=0$
for all $l_1+1\le i\le m.$
\end{lem}

\begin{proof}
By Lemma \ref{d-step} we have
\begin{equation*}
\begin{aligned}xy^{-1}=\exp(X)\exp(Y)
\end{aligned}
\end{equation*}
with $X=\sum_{i=1}^{m}x_iX_i$ and
$$Y=-\sum_{i=1}^{m}y_iX_i+
\sum_{i=l_1+1}^m P_{i,1}(\{y_p\})X_i,$$
where
\begin{equation*}
P_{i,1}(\{y_p\})= -\sum c_{k_1',\ldots, k_m',i}y_{1}^{k_1'}\ldots y_{m}^{k_m'},
\end{equation*}
the sum is over all $0\le k_1',\ldots,k_{m}'\le m$ with $\sum_{j=1}^m k_j'o(y_j)\le o(y_i)$
and there are at least two $j$'s with $k_j'\not=0$ for all $i$ with $l_1+1\le
i\le m.$

Using the CBH formula we get that
\begin{equation*}
\begin{aligned}
xy^{-1}&=\exp(X*Y)=\exp(X+Y+\frac{1}{2}[X,Y]+\frac{1}{12}[X,[X,Y]]+\ldots)\\
&= \exp(\sum_{i=1}^m(x_i-y_i)X_i+\sum_{i=l_1+1}^m (P_{i,1}(\{y_p\})+P_{i,2}(\{x_p\},\{y_p\}))X_i)\\
&=\exp(\sum_{i=1}^{l_1}(x_i-y_i)X_i+\sum_{i=l_1+1}^m((x_i-y_i)+P_{i,1}(\{y_p\})+
P_{i,2}(\{x_p\},\{y_p\}))X_i)
\end{aligned}
\end{equation*}
where
\begin{equation*}
P_{i,2}(\{x_p\},\{y_p\})=\sum
e_{k_1,\ldots,k_m}^{k_1',\ldots,k_m'}x_{1}^{k_1}\ldots
x_{m}^{k_m}y_{1}^{k_1'}\ldots y_{m}^{k_m'},
\end{equation*} the sum is over all $0\le k_j, k_j'\le m$ with
$\sum_{j=1}^m (k_j+k_j')o(y_j)\le o(y_i)$,
and there are at least one $j$ with $k_j\not=0$, one $j$ with $k_j'\not=0$
for all $l_1+1\le i\le m.$

Note that the reason $P_{i,2}$ has the above form follows from the
fact that $[\g^{(i)},\g^{(j)}]\subset \g^{(i+j)}$,
$\g^{(d+1)}=\{0\}$ and a discussion similar to the one used in Lemma
\ref{d-step}.
\end{proof}

\medskip

\subsection{Proof of Theorem B(1)}
Let $X=G/\Gamma$ with $G$ a connected, simply-connected $d$-step
nilpotent Lie group, $\Gamma$ a uniform subgroup. Let $T:X\lra X$ be
the nilrotation induced by $a\in G$. Assume that $A\supset
N(x\Gamma,U)$ with $x\in G$, $x\Gamma\in U$ and $U\subset G/\Gamma$
open. By Proposition \ref{replace} we may assume that $x$ is the unit
element $e$ of $G$, i.e. $A\supset N(e\Gamma,U)$.

Assume that $a=\exp(a_1X_1+\ldots+a_mX_m)$, where $a_1,\ldots a_m\in
\mathbb{R}$. Then $$a^n=\exp(na_1X_1+\ldots+na_mX_m)$$ for any $n\in
\mathbb{Z}$. For $h=\exp(h_1X_1)\ldots \exp(h_mX_m)$, where
$h_1,\ldots,h_m\in \mathbb{R}$, write \begin{equation*}
\begin{aligned}
a^nh^{-1}=\exp(p_1X_1+\ldots+p_mX_m)=\exp(w_1X_1)\ldots\exp(w_mX_m).
\end{aligned}
\end{equation*}
Then by Lemma \ref{product} with $x_j,y_j$ are replaced by $na_j,h_j$ respectively, we have

(1) if $1\le i\le l_1$, then $p_i=na_i-h_i$, and

(2) if $l_{j-1}+1\le i\le l_j,\ 2\le j\le d$ then
\begin{equation}\label{eq-5.2}
\begin{aligned}
p_i =na_i-h_i+P_{i,1}(\{h_p\})+P_{i,2}(\{na_p\}, \{h_p\}),
\end{aligned}
\end{equation}
where $P_{i,1}$ is defined in (\ref{Polyi1}) and $P_{i,2}$ is defined in (\ref{Polyi2}), which
satisfy the properties stated there. It is clear that
$$P_{i,2}(\{na_p\}, \{h_p\})=\sum e_{k_1,\ldots,k_m}^{k_1',\ldots,k_m'}n^{k_1+\ldots+k_m}a_1^{k_1}\ldots a_m^{k_m} h_{1}^{k_1'}\ldots h_{m}^{k_m'}.$$


Changing the exponential coordinates to Mal'sev coordinates (Lemma
\ref{d-step}), we get that

(i) if $1\le i\le l_1$, then $w_i=na_i-h_i$, and

(ii) if $l_{j-1}+1\le i\le l_j,\ 2\le j\le d$, then
\begin{equation*}
w_i=p_i+\sum
d_{k_1,\ldots, k_m,i}p_1^{k_1}\ldots
p_m^{k_m},
\end{equation*}
where the sum is over all $0\le k_1,\ldots,k_{m}\le m$ with $\sum_{j=1}^m k_jo(t_j)\le o(u_i)$
and there are at least two $j$'s with $k_j\not=0$.
In this case using \eqref{eq-5.2} it is not hard to see that
$$w_i=-h_i+Q_i(n,h_1,\ldots,h_{m})$$ such that
$Q_i$ is a polynomial, and each term of $Q_i$ is the form of
$$c(k_1',\ldots,k_{m}', k_1,\ldots,k_{m})n^{k_1'+\ldots+k_{m}'}h_1^{k_1}\ldots h_{m}^{k_{m}}$$
with $\sum_{j=1}^{m}(k_j'+k_j)o(h_j)\le o(h_i)$ (see the argument of Lemma
\ref{d-step}(2)). Note that if $k=k_1'+\ldots+k_{m}'=0$ then there are at least two $j$'s with $k_j\not=0$, and if
$k_1=\ldots=k_m=0$ then $k\ge 1$. This implies that in case (ii) in fact we have
$$w_i=-h_i+Q_i(n,h_1,\ldots,h_{i-1}).$$

For a given $n\in \mathbb{Z}$, let $h_i(n)=\lceil{na_i}\rceil$ if
$1\le i\le l_1$. Moreover, when $h_{v}$ is defined for $1\le v\le i-1$ we let
$$h_i(n)=\lceil{Q_i(n,{h_1(n)},\ldots,h_{i-1}(n))}\rceil$$ if
$l_{j-1}+1\le i\le l_j,\ 2\le j\le d$. Again a similar argument as
in the proof of Lemma~ \ref{d-step}(2) shows that $h_i(n)$ is well
defined and is a generalized polynomial of degree at most $o(h_i)\le
d$. For example, if $l_1+1\le i\le l_2$ then
$$p_i=na_i-h_i+\sum_{1\le i_1<i_2\le
l_1}c(i_1,i_2,i)h_{i_1}h_{i_2}+\sum_{1\le j_1\le l_1}c(j_1,i)nh_{j_1}.$$
So
\begin{equation*}
\begin{aligned}w_i=na_i-h_i+\sum_{1\le i_1<i_2\le
l_1}c(i_1,i_2,i)h_{i_1}h_{i_2}&+\sum_{1\le j_1\le
l_1}c(j_1,i)nh_{j_1}+\sum_{1\le i_1<i_2\le
l_1}d(i_1,i_2,i)p_{i_1}p_{i_2}\\
=na_i-h_i+\sum_{1\le i_1<i_2\le
l_1}c(i_1,i_2,i)h_{i_1}h_{i_2}&+\sum_{1\le j_1\le
l_1}c(j_1,i)nh_{j_1}\\
&+\sum_{1\le i_1<i_2\le
l_1}d(i_1,i_2,i)(na_{i_1}-h_{i_1})(na_{i_2}-h_{i_2}).
\end{aligned}
\end{equation*}
Thus if we let $h_i(n)=\lceil{na_i}\rceil$, $1\le i\le l_1$ then if
$l_1+1\le i\le l_2$
\begin{equation*}
\begin{aligned}
h_i(n)=\lceil na_i&+\sum_{1\le i_1<i_2\le
l_1}c(i_1,i_2,i)\lceil na_{i_1}\rceil \lceil na_{i_2}\rceil+\sum_{1\le j_1\le l_1}c(j_1,i)n\lceil na_{j_1}\rceil\\
&+\sum_{1\le i_1<i_2\le l_1}d(i_1,i_2,i)(na_{i_1}-\lceil
na_{i_1}\rceil)(na_{i_2}-\lceil na_{i_2}\rceil) \rceil.
\end{aligned}
\end{equation*}
That is, $$h_i(n)=\lceil na_i+n^2a_i'+\sum_{1\le i_1<i_2\le
l_1}c'(i_1,i_2,i)\lceil na_{i_1}\rceil \lceil
na_{i_2}\rceil+\sum_{1\le j_1\le l_1}c'(j_1,i)n\lceil
na_{j_1}\rceil\rceil
$$ is a generalized polynomial of degree at most $2$ in $n$.

\medskip
Next we let
 $w_i(n)=na_i-h_i(n)=na_i-\lceil{na_i}\rceil$ for $1\le i\le l_1$ and
  if
$l_{j-1}+1\le i\le l_j,\ 2\le j\le d$, let
\begin{align*}
w_i(n)&=Q_i(n,{h_1(n)},\ldots,h_{i-1}(n))-h_i(n)\\
&=Q_i(n,{h_1(n)},\ldots,h_{i-1}(n))-\lceil{Q_i(n,{h_1(n)},\ldots,h_{i-1}(n))}\rceil.
\end{align*}
The previous argument shows that $w_i(n)$ is a generalized polynomial of degree at most $d$.

\medskip
Let $h(n)=\exp(h_1(n)X_1)\ldots \exp(h_m(n)X_m)$. Then $h(n)\in
\Gamma$ and $$a^nh(n)^{-1}=\exp(w_1(n)X_1)\ldots \exp(w_m(n)X_m).$$
{Denote by $\pi$ the quotient map $\pi: G\rightarrow X$. Since $\pi^{-1}(U)$ is open and contains $e$, there is some $0<\ep < \frac{1}{2}$ such that
$$\pi^{-1}(U)\supset \{\exp(t_1X_1)\ldots \exp(t_mX_m): |t_1|,\ldots, |t_m|<\ep\}=: V.$$
Then
$$A\supset N(e\Gamma,U)\supset \{n\in\Z: a^nh(n)^{-1} \in V \}.$$

So if $n\in \bigcap_{i=1}^m\{n\in\Z:w_i(n)\ (\text{mod}\ \Z)\in
(-\ep,\ep) \}$ then $n\in \{n\in\Z: a^nh(n)^{-1} \in V \}\subset N(e\Gamma,U)\subset A$.} That is,
$$A\supset \bigcap_{i=1}^m\{n\in\Z:w_i(n)\ (\text{mod}\ \Z)\in (-\ep,\ep) \}.$$
This ends the proof of Theorem B(1).

\section{Proof of Theorem B(2)}
\label{section-proof2}

In this section, we aim to prove Theorem B(2), i.e.
$\F_{d,0}\supset\F_{GP_d}$.
To do this first we make some preparations, then derive some results
under the inductive assumption, and finally give the proof. Note
that in the construction the nilpotent matrix Lie group is used.


More precisely, to show $\F_{d,0}\supset\F_{GP_d}$ we need only to
prove $\F_{d,0}\supset\F_{SGP_d}$ by Theorem \ref{gpsgp}. To do this, for a
given $F\in \F_{SGP_d}$ we need to find a $d$-step nilsystem
$(X,T)$, $x_0\in X$ and a neighborhood $U$ of $x_0$ such that
$F\supset N(x_0,U)$. In the process of doing this, we find that it is
convenient to consider a finite sum of specially generalized
polynomials $P(n;\alpha_1,\ldots,\alpha_r)$ (defined in
(\ref{key-poly})) instead of considering a single specially
generalized polynomial. We can prove that $\F_{d,0}\supset\F_{GP_d}$
if and only if $\{ n\in \mathbb{Z}:
||P(n;\alpha_1,\ldots,\alpha_d)||<\epsilon\}\in \mathcal{F}_{d,0}$
for any $\alpha_1,\ldots,\alpha_d\in \mathbb{R}$ and $\epsilon>0$
(Theorem \ref{lem-trans}). We choose $(X,T)$ as the closure of the
orbit of $\Gamma$ in $\G_d/\Gamma$ (the nilrotation is induced by a
matrix $A\in \G_d$), and consider the most right-corner entry
$z_1^d(n)$ in $A^nB_n$ with $B_n\in\Gamma$. We finish the proof by
showing that $P(n;\alpha_1,\ldots,\alpha_d)\simeq_d z_1^d(n)$ and
$\{n\in\Z: ||z_1^d(n)||<\ep\}\in \F_{d,0}$ for any $\ep>0$.

\subsection{Some preparations}

For a matrix $A$ in $\G_d$ we now give a precise formula of $A^n$.
\begin{lem} \label{lem-8-ite} Let ${\bf x}=(x_i^k)_{1\le k\le d, 1\le i\le
d-k+1}\in \mathbb{R}^{d(d+1)/2}$. For $n\in \mathbb{N}$, assume that
${\bf x}(n)=(x_i^k(n))_{1\le k\le d, 1\le i\le d-k+1}\in
\mathbb{R}^{d(d+1)/2}$ satisfies $\M({\bf x}(n))=\M({\bf x})^n$,
then
\begin{equation}\label{8-11-0}
x_i^k(n)=\tbinom{n}{1}P_1({\bf x};i,k)+\tbinom{n}{2}P_2({\bf
x};i,k)+\ldots+\tbinom{n}{k} P_k({\bf x};i,k)
\end{equation}
for $1\le k\le d$ and $1\le i\le d-k+1$, where
$\tbinom{n}{k}=\frac{n(n-1)\ldots (n-k+1)}{k!}$ for $n,k\in
\mathbb{N}$ and
$$P_\ell({\bf x};i,k)=\sum\limits_{(s_1,s_2,\ldots,s_{\ell})\in \{1,2,\ldots,k\}^\ell
\atop s_1+s_2+\ldots+s_\ell=k}x_i^{s_1}x_{i+s_1}^{s_2}
x_{i+s_1+s_2}^{s_3}\ldots x_{i+s_1+s_2+\ldots+s_{\ell-1}}^{s_\ell}$$
for $1\le k\le d$, $1\le i\le d-k+1$ and $1\le \ell \le k$.
\end{lem}
\begin{proof} Let $x_i^0=1$ and $x_i^0(m)=1$ for $1\le i\le d$ and
$m\in \mathbb{N}$. By \eqref{sec8-4-eq-1}, it is not hard to see
that
\begin{equation}\label{sec8-4-eq-2}
x_i^k(m+1)=\sum \limits_{j=0}^k x_i^{k-j}(m)\cdot x_{i+k-j}^j
\end{equation}
for $1\le k\le d$, $1\le i\le d-k+1$ and $m\in \mathbb{N}$.

Now we do induction for $k$. When $k=1$, $x_i^1(1)=x_i^1$ and
$x_i^1(m+1)=x_i^1(m)+x_i^1$ for $m\in \mathbb{N}$ by
\eqref{sec8-4-eq-2}. Hence $x_i^1(n)=nx_i^1=\tbinom{n}{1}P_1({\bf
x};i,1)$. That is, \eqref{8-11-0} holds for each $1\le i\le d$ and
$n\in \mathbb{N}$ if $k=1$.

Assume that $1\le \ell \le d-1$, and \eqref{8-11-0} holds for each
$1\le k\le \ell$, $1\le i\le d-k+1$ and $n\in \mathbb{N}$. For
$k=\ell+1$, we make induction on $n$. When $n=1$ it is clear
$$x_i^k(1)=x_i^k=\tbinom{1}{1}P_1({\bf x};i,k)+\tbinom{1}{2}P_2({\bf
x};i,k)+\ldots+\tbinom{1}{k} P_k({\bf x};i,k)$$ for $1\le i\le
d-k+1$. That is, \eqref{8-11-0} holds for $k=\ell+1$, $1\le i\le
d-k+1$ and $n=1$. Assume for $n=m\ge 1$,  \eqref{8-11-0} holds for
$k=\ell+1$, $1\le i\le d-k+1$ and $n=m$. For $n=m+1$, by
\eqref{sec8-4-eq-2}
\begin{align*}
x_i^k(n)&=x_i^k(m)+\Big(\sum \limits_{j=1}^{k-1} x_i^{k-j}(m)\cdot x_{i+k-j}^j\Big)+x_i^k\\
&=x_i^k(m)+\Big(\sum \limits_{j=1}^{k-1} (\sum \limits_{r=1}^{k-j}\tbinom{m}{r}P_r({\bf x};i,k-j))\cdot x_{i+k-j}^j)\Big)+x_i^k\\
&=x_i^k(m)+\Big(\sum \limits_{r=1}^{k-1} (\sum \limits_{j=1}^{k-r} P_r({\bf x};i,k-j)x_{i+k-j}^{j})\tbinom{m}{r}\Big)+x_i^k\\
&=x_i^k(m)+\Big(\sum \limits_{r=1}^{k-1} (\sum \limits_{j=r}^{k-1}
P_r({\bf x};i,j)x_{i+j}^{k-j})\tbinom{m}{r}\Big)+x_i^k
\end{align*}
for $1\le i\le d-k+1$. Note that

$$\sum \limits_{j=r}^{k-1} P_r({\bf x};i,j)x_{i+j}^{k-j}=\sum
\limits_{j=r}^{k-1} \sum \limits_{(s_1,\ldots,s_r)\in
\{1,2,\ldots,k-1\}^r \atop
s_1+\ldots+s_r=j}x_i^{s_1}x_{i+s_1}^{s_2}\ldots
x_{i+s_1+\ldots+s_{r-1}}^{s_r} x_{i+j}^{k-j}$$ which is equal to
$$\sum \limits_{(s_1,\ldots,s_r,s_{r+1})\in \{1,2,\ldots,k-1\}^{r+1}
\atop s_1+s_2+\ldots+s_r+s_{r+1}=k}x_i^{s_1}x_{i+s_1}^{s_2}\ldots
x_{i+s_1+\ldots+s_{r-1}}^{s_r}
x_{i+s_1+\ldots+s_{r-1}+s_r}^{s_{r+1}}=P_{r+1}({\bf x};i,k)$$ for
$1\le r\le k-1$ and $1\le i\le d-k+1$. Collecting terms we have

\begin{align*}
x_i^k(n)&=x_i^k(m)+\Big(\sum \limits_{r=1}^{k-1} P_{r+1}({\bf x};i,k) \tbinom{m}{r}\Big)+x_i^k\\
&=x_i^k(m)+\Big(\sum \limits_{r=2}^{k} P_{r}({\bf x};i,k) \tbinom{m}{r-1}\Big)+P_1({\bf x};i,k)\\
&=\Big(\sum \limits_{r=1}^m P_r({\bf x};i,k)\tbinom{m}{r}\Big)+
\Big(\sum \limits_{r=2}^{k} P_{r} ({\bf x};i,k)
\tbinom{m}{r-1}\Big)+P_1({\bf x};i,k).
\end{align*}

Rearranging the order we get
\begin{align*}
x_i^k(n)&=(m+1)P_1({\bf x};i,k)+\sum_{r=2}^k \Big(\tbinom{m}{r}+\tbinom{m}{r-1}\Big)P_r({\bf x};i,k)\\
&=\sum \limits_{r=1}^k \tbinom{m+1}{r}P_r({\bf x};i,k)=\sum
\limits_{r=1}^k \tbinom{n}{r}P_r({\bf x};i,k)
\end{align*}
for $1\le i\le d-k+1$. This ends the proof of the lemma.
\end{proof}

\begin{rem} \label{rem-8-ite}
By the above lemma, we have $$P_1({\bf x};i,k)=x_i^k \text{ and
 }P_k({\bf x};i,k)=x_i^1x_{i+1}^1\ldots x_{i+k-1}^1$$ for $1\le k\le
d$ and $1\le i\le d-k+1$.
\end{rem}

\subsection{Consequences under the inductive assumption}
We will use induction to show Theorem B(2). To make the proof
clearer to read, we derive some results under the following
inductive assumption.
\begin{equation}\label{as-d-1}
\mathcal{F}_{d-1,0}\supset \mathcal{F}_{GP_{d-1}},
\end{equation}
where $d\in\mathbb{N}$ with $d\ge 2$. For that purpose, we need more
notions and lemmas. The proof of Lemma \ref{lem-8-11} is similar to
the one of Lemma \ref{lem1-1}, where $\mathcal{W}_d$ is defined in
Definition \ref{de-ws}.

\begin{lem}\label{lem-8-11}
Under the assumption \eqref{as-d-1}, one has for any $p(n)\in
\mathcal{W}_d$ and $\epsilon>0$,
$$\{n\in \mathbb{Z}:p(n)\ (\text{mod}\ \Z) \in (-\epsilon,\epsilon)\}\in \mathcal{F}_{d-1,0}.$$
\end{lem}

\begin{de} \index{${\widetilde{GP}}_r$}
For $r\in \mathbb{N}$, we define
$${\widetilde {GP}}_r=\{ p(n)\in GP_r: \{ n\in \mathbb{Z}: p(n)\ (\text{mod}\ \Z)\in (-\epsilon,\epsilon)\}\in
\mathcal{F}_{r,0} \text{ for any }\epsilon>0\}.$$
\end{de}
\begin{rem}\label{rem-8-15} It is clear that for $p(n)\in GP_r$,
$p(n)\in {\widetilde {GP}}_r$ if and only if $-p(n)\in {\widetilde
{GP}}_r$. Since $\mathcal{F}_{r,0}$ is a filter, if
$p_1(n),p_2(n),\ldots,p_k(n)\in {\widetilde {GP}}_r$ then
$$p_1(n)+p_2(n)+\ldots+p_k(n)\in {\widetilde {GP}}_r.$$ Moreover by the definition
of ${\widetilde {GP}}_d$, we know that $\mathcal{F}_{d,0}\supset
\mathcal{F}_{GP_d}$ if and only if ${\widetilde {GP}}_d=GP_d$.
\end{rem}

\begin{lem}\label{lem-8-equi-new} Let $p(n),q(n)\in GP_d$ with $p(n)\simeq_d q(n)$. Under the assumption \eqref{as-d-1},
$p(n)\in {\widetilde {GP}}_d$ if and only if $q(n)\in {\widetilde
{GP}}_d$.
%
\end{lem}
\begin{proof} This follows from Lemma \ref{lem-8-11} and the fact that $\mathcal{F}_{d,0}$ is a filter.
\end{proof}

For $\alpha_1,\alpha_2,\ldots, \alpha_r\in \mathbb{R}, r\in \N$, we
define \index{$P(n; \alpha_1,\alpha_2, \ldots, \alpha_r)$} $P(n;\alpha_1,\alpha_2,\ldots,\alpha_r)$ as
\begin{align}\label{key-poly}
\sum \limits_{\ell=1}^r \sum \limits_{j_1,\ldots, j_\ell\in
\mathbb{N}\atop j_1+\ldots+j_\ell=r}(-1)^{\ell-1} L\Big(
\frac{n^{j_1}}{j_1!}\prod \limits_{r_1=1}^{j_1}\alpha_{r_1},\,
\frac{n^{j_2}}{j_2!}\prod \limits_{r_2=1}^{j_2}\alpha_{j_1+r_2},\,
\ldots,\, \frac{n^{j_\ell}}{j_\ell!}\prod
\limits_{r_\ell=1}^{j_\ell}\alpha_{a(\ell-1)+r_\ell} \Big)
\end{align}
where the definition of $L$ is given in \eqref{eq-de-L}, and $a(\ell)=\sum_{t=1}^{\ell}j_t$.

\begin{thm} \label{lem-trans}
Under the assumption \eqref{as-d-1}, the following properties are
equivalent:
\begin{enumerate}
\item  $\mathcal{F}_{d,0}\supset \mathcal{F}_{GP_d}$.

\item $P(n;\alpha_1,\alpha_2,\ldots,\alpha_d)\in {\widetilde {GP}}_d$ for any
$\alpha_1,\alpha_2,\ldots,\alpha_d\in \mathbb{R}$, that is
$$\{ n\in \mathbb{Z}: P(n;\alpha_1,\alpha_2,\ldots,\alpha_d)\ (\text{mod}\ \Z)
\in (-\epsilon, \epsilon) \}\in \mathcal{F}_{d,0}$$ for any
$\alpha_1,\alpha_2,\ldots,\alpha_d\in \mathbb{R}$ and $\epsilon>0$.

\item $\text{SGP}_d\subset {\widetilde {GP}}_d$.
\end{enumerate}
\end{thm}
\begin{proof} $(1)\Rightarrow(2)$. Assume $\mathcal{F}_{d,0}\supset \mathcal{F}_{GP_d}$.
By the definition of ${\widetilde {GP}}_d$, we know that
$\mathcal{F}_{d,0}\supset \mathcal{F}_{GP_d}$ if and only if
${\widetilde {GP}}_d=GP_d$. Particularly
$P(n;\alpha_1,\alpha_2,\ldots,\alpha_d)\in {\widetilde {GP}}_d$ for
any $\alpha_1,\alpha_2,\ldots,\alpha_d\in \mathbb{R}$.

\medskip
$(3)\Rightarrow (1)$. Assume that $\text{SGP}_d\subset {\widetilde
{GP}}_d$. Then $\mathcal{F}_{d,0}\supset \mathcal{F}_{SGP_d}$.
Moveover $\mathcal{F}_{d,0}\supset\F_{SGP_d}=\F_{GP_d}$ by Theorem
\ref{gpsgp}.

\medskip
$(2)\Rightarrow (3)$.  Assume that
$P(n;\alpha_1,\alpha_2,\ldots,\alpha_d)\in {\widetilde {GP}}_d$ for
any $\alpha_1,\alpha_2,\ldots,\alpha_d\in \mathbb{R}$. We define
$$\Sigma_d=\{ (j_1,j_2,\ldots,j_\ell): \ell\in \{1,2,\ldots,d\}, j_1,j_2,\ldots,
j_\ell\in \mathbb{N} \text{ and }\sum \limits_{t=1}^\ell j_t=d\}. $$
For $(j_1,j_2,\ldots,j_\ell),(r_1,r_2,\ldots,r_s)\in \Sigma_d$, we
say $(j_1,j_2,\ldots,j_\ell)>(r_1,r_2,\ldots,r_s)$ if there exists
$1\le t\le \ell$ such that $j_t>r_t$ and $j_i=r_i$ for $i<t$.
Clearly $(\Sigma_d,>)$ is a totally ordered set with the maximal
element $(d)$ and the minimal element $(1,1,\ldots,1)$.

For ${\bf j}=(j_1,j_2,\ldots,j_\ell)\in \Sigma_d$, put
$$\mathcal{L}({\bf j})=\{L(n^{j_1}a_1,\ldots,n^{j_\ell}a_\ell):a_1,\ldots,a_\ell\in \mathbb{R}\}.$$
Now, we have

\medskip
\noindent{\bf Claim:} $\mathcal{L}({\bf s})\subset {\widetilde
{GP}}_d$ for each ${\bf s}\in \Sigma_d$.

\begin{proof} We do induction for ${\bf s}$ under the order $>$.
First, consider the case when ${\bf s}=(d)$. Given  $a_1\in
\mathbb{R}$, we take $\alpha_1=1,\alpha_2=2,\ldots ,\alpha_{d-1}=d-1$
and $\alpha_d=d a_1$. Then for any $1\le j_1\le d-1$,
$\frac{n^{j_1}}{j_1!}\prod \limits_{t=1}^{j_1} \alpha_t \in
\mathbb{Z}$ for $n\in \mathbb{Z}$. Thus
$$P(n;\alpha_1,\alpha_2,\ldots,\alpha_d)=L(\frac{n^d}{d!}\prod
\limits_{t=1}^{d} \alpha_t)=L(n^d a_1) \ (\text{mod}\ \Z)$$ for any
$n\in \mathbb{Z}$. Hence $L(n^d a_1)\in {\widetilde {GP}}_d$ since
$P(n;\alpha_1,\alpha_2,\ldots,\alpha_d)\in {\widetilde {GP}}_d$.
Since $a_1$ is arbitrary, we conclude that $\mathcal{L}((d))\subset
{\widetilde {GP}}_d$.

Assume that for any ${\bf s}>{\bf i}=(i_1,\ldots,i_k)\in \Sigma_d$,
we have $\mathcal{L}({\bf s})\subset {\widetilde {GP}}_d$. Now
consider the case when  ${\bf s}={\bf i}=(i_1,\ldots,i_k)$. There
are two cases.

The first case is $k=d$, $i_1=i_2=\ldots =i_d=1$. Given
$a_1,a_2,\ldots,a_d \in \R$, by the assumption we have that for any
$(j_1,j_2,\ldots,j_{\ell})>{\bf i}$,
$\mathcal{L}((j_1,j_2,\ldots,j_{\ell}))\subset {\widetilde {GP}}_d$.
Thus
$$\sum \limits_{\ell=1}^{d-1} \sum \limits_{j_1,\ldots, j_\ell\in
\mathbb{N}\atop j_1+\ldots+j_\ell=r}(-1)^{\ell-1} L\Big(
\frac{n^{j_1}}{j_1!}\prod \limits_{r_1=1}^{j_1}a_{r_1},\,
\frac{n^{j_2}}{j_2!}\prod \limits_{r_2=1}^{j_2}a_{j_1+r_2},\,
\ldots,\, \frac{n^{j_\ell}}{j_\ell!}\prod
\limits_{r_\ell=1}^{j_\ell}a_{a(\ell-1)+r_\ell} \Big)$$
belongs to ${\widetilde {GP}}_d$ by the Remark \ref{rem-8-15}. This
implies  $$P(n;a_1,a_2,\ldots,a_d)-(-1)^{d-1} L(n a_1,n a_2,\ldots,n a_d)\in
{\widetilde {GP}}_d$$ by \eqref{key-poly}. Combining this with
$P(n;a_1,a_2,\ldots,a_d)\in {\widetilde {GP}}_d$, we have $$L(n
a_1,n a_2,\ldots,n a_d)\in {\widetilde {GP}}_d$$ by  Remark
\ref{rem-8-15}. Since $a_1,a_2,\ldots,a_d \in \R$ are arbitrary, we
get $\mathcal{L}({\bf i})\subset {\widetilde {GP}}_d$.

The second case is ${\bf i}>(1,1,\ldots,1)$. Given
$a_1,a_2,\ldots,a_k\in \mathbb{R}$, for $r=1,2,\ldots,k$, we put
$\alpha_{\sum_{t=1}^{r-1}i_t+h}=h$ for $1\le h\le i_r-1$ and
$\alpha_{\sum_{t=1}^{r-1}i_t+i_r}=i_r a_r$.

By the assumption, for $(j_1,j_2,\ldots,j_{\ell})>{\bf i}$,
 $$L\Big(
\frac{n^{j_1}}{j_1!}\prod \limits_{r_1=1}^{j_1}\alpha_{r_1},\,
\frac{n^{j_2}}{j_2!}\prod \limits_{r_2=1}^{j_2}\alpha_{j_1+r_2},\,
\ldots,\, \frac{n^{j_\ell}}{j_\ell!}\prod
\limits_{r_\ell=1}^{j_\ell}\alpha_{a(\ell-1)+r_\ell}
\Big)\in {\widetilde {GP}}_d.$$

For $(j_1,j_2,\ldots,j_\ell)<{\bf i}$, there exists $1\le u\le k$
such that $j_t=i_t$ for $1\le t\le u-1$ and $i_{u}>j_u$. Then
\begin{equation}\label{8-interger}
\frac{n^{j_u}}{j_u!}\prod
\limits_{r_u=1}^{j_u}\alpha_{a(u-1)+r_u}=n^{j_u}.
\end{equation}
When $u=1$,  by \eqref{8-interger}, $$L\Big(
\frac{n^{j_1}}{j_1!}\prod \limits_{r_1=1}^{j_1}\a_{r_1},\ldots,\,
\ldots,\, \frac{n^{j_\ell}}{j_\ell!}\prod
\limits_{r_\ell=1}^{j_\ell}\a_{a(\ell-1)+r_\ell}
\Big) \in \mathbb{Z}$$ for any $n\in \mathbb{Z}$. Hence
$$L\Big(
\frac{n^{j_1}}{j_1!}\prod \limits_{r_1=1}^{j_1}\a_{r_1},\ldots,\,
\ldots,\, \frac{n^{j_\ell}}{j_\ell!}\prod
\limits_{r_\ell=1}^{j_\ell}\a_{a(\ell-1)+r_\ell}
\Big)\in \widetilde{GP}_d.$$ When $u>1$, write
$\beta_v=\frac{1}{j_v!}\prod_{r_v=1}^{j_v}\alpha_{a(v-1)+r_v}$ for
$v=1,2,\ldots,\ell$. Then $\beta_u=1$ and $$\lceil
L(n^{j_u}\beta_u,n^{j_{u+1}}\beta_{u+1}\, \ldots,\,
n^{j_\ell}\beta_\ell)\rceil=L(n^{j_u}\beta_u,n^{j_{u+1}}\beta_{u+1}\,
\ldots,\, n^{j_\ell}\beta_\ell).$$ Moreover,
\begin{align*}&L\Big(\frac{n^{j_1}}{j_1!}\prod \limits_{r_1=1}^{j_1}\a_{r_1},\ldots,\frac{n^{j_u}}{j_u!}\prod
\limits_{r_u=1}^{j_u}\alpha_{a(u-1)+r_\ell},\, \ldots,\, \frac{n^{j_\ell}}{j_\ell!}\prod
\limits_{r_\ell=1}^{j_\ell}\a_{a(\ell-1)+r_\ell} \Big)\\
&=L\left(n^{j_1}\beta_1,\ldots,n^{j_u}\beta_u,\, \ldots,\,n^{j_\ell}\beta_\ell \right)\\
&=L\left(n^{j_1}\beta_1,\ldots,n^{j_{u-1}}\beta_{u-1}\lceil L(n^{j_u}\beta_u,
\ldots,\, n^{j_\ell}\beta_\ell)\rceil \right)
\end{align*} which is equal to
\begin{align*} &L\left( n^{j_1}\beta_1,\ldots,n^{j_{u-1}}\beta_{u-1}
L(n^{j_u}\beta_u,n^{j_{u+1}}\beta_{u+1}\,
\ldots,\, n^{j_\ell}\beta_\ell)\right)\\
&=L\left( n^{j_1}\beta_1,\ldots,n^{j_{u-1}+j_u}\beta_{u-1}\beta_u
\lceil L(n^{j_{u+1}}\beta_{u+1}\, \ldots,\,
n^{j_\ell}\beta_\ell)\rceil \right)\\
&=L\left( n^{j_1}\beta_1,\ldots,n^{j_{u-1}+j_u}\beta_{u-1}\beta_u,
n^{j_{u+1}}\beta_{u+1}\, \ldots,\, n^{j_\ell}\beta_\ell \right) \in
{\widetilde {GP}}_d
\end{align*}
since $(j_1,\ldots,j_{u-2},j_{u-1}+j_u,j_{u+1},\ldots,j_\ell)>{\bf
i}$.

Summing up for any ${\bf j}=(j_1,\ldots,j_\ell)\in \Sigma_d$ with
${\bf j}\neq {\bf i}$, we have
$$L\Big( \frac{n^{j_1}}{j_1!}\prod
\limits_{r_1=1}^{j_1}\a_{r_1},\ldots,\frac{n^{j_u}}{j_u!}\prod
\limits_{r_u=1}^{j_u}\alpha_{a(u-1)+r_\ell},\,
\ldots,\, \frac{n^{j_\ell}}{j_\ell!}\prod
\limits_{r_\ell=1}^{j_\ell}\a_{a(\ell-1)+r_\ell} \Big)
\in {\widetilde {GP}}_d.$$ Combining this with
$P(n;\alpha_1,\ldots,\alpha_d)\in {\widetilde {GP}}_d$, we have
\begin{align*}
&\hskip0.6cm L\Big( n^{i_1}a_1,\, n^{i_2}a_2,\, \ldots,\,
n^{i_k}a_k \Big)\\
&= L\Big( \frac{n^{i_1}}{i_1!}\prod
\limits_{r_1=1}^{i_1}\alpha_{r_1},\, \frac{n^{i_2}}{i_2!}\prod
\limits_{r_2=1}^{i_2}\alpha_{i_1+r_2},\, \ldots,\,
\frac{n^{i_k}}{i_k!}\prod
\limits_{r_k=1}^{i_k}\alpha_{\sum_{t=1}^{k-1}i_t+r_k} \Big)\in
{\widetilde {GP}}_d
\end{align*}
by \eqref{key-poly} and Remark \eqref{rem-8-15}. Since
$a_1,\ldots,a_k \in \R$ are arbitrary, $\mathcal{L}({\bf i})\subset
{\widetilde {GP}}_d$.
\end{proof}

Finally, since $\text{SGP}_d=\bigcup_{{\bf j}\in \Sigma_d}
\mathcal{L}({\bf j})$, we have $\text{SGP}_d\subset {\widetilde
{GP}}_d$ by the above Claim.
\end{proof}


\subsection{Proof of Theorem B(2)}

We are now ready to give the proof of the Theorem B(2). As we said
before, we will use induction to show Theorem B(2). Firstly, for
$d=1$, since $\mathcal{F}_{GP_1}=\mathcal{F}_{SGP_1}$ and
$\mathcal{F}_{1,0}$ is a filter, it is sufficient to show for any
$a\in \mathbb{R}$ and $\epsilon>0$,
$$\{n\in \mathbb{Z}: an \ (\text{mod}\ \Z)\in (-\epsilon,\epsilon)\}\in \mathcal{F}_{1,0}.$$
This is obvious since the rotation on the unit circle is a 1-step
nilsystem.

\medskip
Now we assume that $\mathcal{F}_{d-1,0}\supset
\mathcal{F}_{GP_{d-1}}$, i.e. the the assumption \eqref{as-d-1}
holds. By Theorem \ref{lem-trans}, to show $\mathcal{F}_{d,0}\supset
\mathcal{F}_{GP_d}$, it remains to prove that
$P(n;\alpha_1,\alpha_2,\ldots,\alpha_d)\in {\widetilde {GP}}_d$ for
any $\alpha_1,\alpha_2,\ldots,\alpha_d\in \mathbb{R}$, that is
$$\{ n\in \mathbb{Z}: P(n;\alpha_1,\alpha_2,\ldots,\alpha_d)\ (\text{mod}\ \Z)
\in (-\epsilon, \epsilon) \}\in \mathcal{F}_{d,0}$$ for any
$\alpha_1,\alpha_2,\ldots,\alpha_d\in \mathbb{R}$ and $\epsilon>0$.

Let $\alpha_1,\alpha_2,\ldots,\alpha_d\in \mathbb{R}$ and choose
${\bf x}=(x_i^k)_{1\le k\le d, 1\le i\le d-k+1}\in
\mathbb{R}^{d(d+1)/2}$ with $x_i^1=\alpha_i$ for $i=1,2,\ldots,d$
and $x_i^k=0$ for $2\le k\le d$ and $1\le i\le d-k+1$. Then
{\footnotesize
$$
A=\M({\bf x})=\left(
  \begin{array}{cccccccc}
    1 & \alpha_1 & 0                &\ldots & 0 &0\\
    0 & 1     & \alpha_2            &\ldots &  0&0\\
    \vdots & \vdots  &  \vdots & \vdots &\vdots\\
    0 & 0    &0                   & \ldots & \alpha_{d-1} &0 \\
    0 & 0    &0                   & \ldots & 1 &\alpha_d \\
    0 & 0    &0                 & \ldots & 0  & 1
  \end{array}
\right)
$$}
{ For $n\in \mathbb{\mathbb{Z}}$, if ${\bf x}(n)=(x_i^k(n))_{1\le k\le d, 1\le
i\le d-k+1}\in \mathbb{R}^{d(d+1)/2}$ satisfies $\M({\bf
x}(n))=A^n$, then $x_i^k(n)$ is a polynomial of $n$ for $1\le k\le d$ and $1\le i\le d-k+1$.
Moreover by Lemma \ref{lem-8-ite} and Remark
\ref{rem-8-ite}, when $n\in \mathbb{Z}$
\begin{equation}\label{thm-8-11-0}
x_i^k(n)=\tbinom{n}{k} P_k({\bf
x};i,k)=\tbinom{n}{k}x_i^1x_{i+1}^1\ldots
x_{i+k-1}^1=\tbinom{n}{k}\alpha_i \alpha_{i+1}\ldots \alpha_{i+k-1}
\end{equation}
for $1\le k\le d$ and $1\le i\le d-k+1$, where $\tbinom{n}{k}=\frac{n(n-1)\cdots(n-k+1)}{k!}$.}

Now we define $f_i^1(n)=\lceil x_i^1(n)\rceil=\lceil n\alpha_i
\rceil$ for $1\le i\le d$ and inductively for $k=2,3,\ldots,d$
define
\begin{equation}\label{sec8-dtgs-1}
f_i^k(n)=\bigg\lceil x_i^k(n)-\sum
\limits_{j=1}^{k-1}x_i^{k-j}(n)f_{i+k-j}^j(n)  \bigg\rceil
\end{equation}
for $1\le i \le d-k+1$. Then we define
$$z_i^1(n)=x_i^1(n)-f_i^1(n)$$ for $1\le i\le d$ and inductively for
$k=2,3,\ldots,d$ define
\begin{equation}\label{sec8-thm-eq-1-old}
z_i^k(n)=x_i^k(n)-\Big( \sum
\limits_{j=1}^{k-1}x_i^{k-j}(n)f_{i+k-j}^j(n) \Big)-f_i^k(n)
\end{equation}
for $1\le i \le d-k+1$.

It is clear that $z_i^k(n)\in GP_k$ for $1\le k\le d$ and $1\le i\le
d-k+1$. First, we have

\medskip
\noindent{\bf Claim:} $P(n;\alpha_1,\alpha_2,\ldots,\alpha_d)
\simeq_d z_1^d(n)$.

Since the proof of the Claim is long, the readers find the proof in
the following subsection. Now we are going to show $z_1^d(n) \in
{\widetilde {GP}}_d$.

Let $X=\G_d/\Gamma$  and $T$ be the nilrotation induced by $A\in\G_d$,
i.e. $B\Gamma\mapsto AB\Gamma$ for $B\in\G_d$. Since $\G_d$ is a
$d$-step nilpotent Lie group and $\Gamma$ is a uniform subgroup of
$\G_d$, $(X,T)$ is a $d$-step nilsystem.

Let $I$ be  the $(d+1)\times (d+1)$ identity matrix. For a given $\eta>0$, choose
the open neighborhood $V=\{C\in \G_d: \|C-I\|_\infty< \min\{\frac{1}{2},\eta\}\}$ of $I$ in $\G_d$.
Let $x_0=\Gamma\in X$ and $U=V\Gamma$. Then $U$ is an open neighborhood of $x_0$ in $X$.
Put $$ S=\{ n\in \mathbb{\mathbb{Z}}: A^n\Gamma \in U\}=\{ n\in
\mathbb{Z}:T^n x_0 \in U\}.$$
Then $S\in \mathcal{F}_{d,0}$. In the following we are
going to show that
$$\{ m\in \mathbb{Z}: z_1^d(m)\ (\text{mod} \ \Z)\in (-\eta,\eta) \}\supset S.$$
This clearly implies that $\{ m\in \mathbb{Z}: z_1^d(m)\ (\text{mod}
\ \Z)\in (-\eta,\eta) \}\in \mathcal{F}_{d,0}$ since $S\in
\mathcal{F}_{d,0}$. As $\eta>0$ is arbitrary, we conclude that
$z_1^d(n) \in {\widetilde {GP}}_d$.

Given $n\in S$, one has $A^n\Gamma\in V\Gamma$. Thus there exists
$B_n\in \Gamma$ such that $A^nB_n\in V$, that is,
\begin{equation}\label{me-eq-222}
\|A^nB_n-I\|_\infty<\min\Big\{\frac{1}{2},\eta\Big\}.
\end{equation}
Take
${\bf h}(n)=(-h_i^k(n))_{1\le k\le d, 1\le i\le d-k+1}\in
\mathbb{Z}^{d(d+1)/2}$ with $\M({\bf h}(n))=B_n$.
Let ${\bf y}(n)=(y_i^k(n))_{1\le k\le d, 1\le i\le d-k+1}\in
\mathbb{R}^{d(d+1)/2}$ such that
$$\M({\bf y}(n))=A^nB_n=\M({\bf x}(n))\M({\bf h}(n)).$$
By \eqref{sec8-4-eq-1}
\begin{equation}\label{sec8-thm-eq-1}
y_i^k(n)=x_i^k(n)-\Big( \sum
\limits_{j=1}^{k-1}x_i^{k-j}(n)h_{i+k-j}^j(n) \Big)-h_i^k(n)
\end{equation}
for $1\le k \le d$ and $1\le i \le d-k+1$. Thus
\begin{equation}\label{sec8-thm-eq-2}
|y_i^k(n)|<\min\{\frac{1}{2},\eta\}
\end{equation}
 for $1\le k \le d$ and $1\le i \le d-k+1$ by \eqref{me-eq-222}.
Hence $h_i^1(n)=\lceil x_i^1(n)\rceil=\lceil n\alpha_i \rceil$ for
$1\le i\le d$ and
\begin{equation}\label{sec8-dtgs-2}
h_i^k(n)=\bigg \lceil x_i^k(n)- \sum
\limits_{j=1}^{k-1}x_i^{k-j}(n)h_{i+k-j}^j(n) \bigg \rceil
\end{equation}
for $2\le k \le d$ and $1\le i \le d-k+1$.

Since $h_i^1(n)=\lceil n\alpha_i \rceil=f_i^1(n)$ for $1\le i\le d$,
one has $h_i^k(n)=f_i^k(n)$ for $2\le k \le d$ and $1\le i \le
d-k+1$ by \eqref{sec8-dtgs-1} and \eqref{sec8-dtgs-2}. Moreover by
\eqref{sec8-thm-eq-1-old} and \eqref{sec8-thm-eq-1}, we know
$z_i^k(n)=y_i^k(n)$ for $2\le k \le d$ and $1\le i \le d-k+1$.
Combining this with \eqref{sec8-thm-eq-2},
$|z_i^k(n)|<\min\{\frac{1}{2},\eta\}$ for $1\le k \le d$ and $1\le i
\le d-k+1$. Particularly, $|z_1^d(n)|<\eta$. Thus $$n\in \{ m\in
\mathbb{Z}: z_1^d(m)\ (\text{mod} \ \Z)\in (-\eta,\eta)\},$$ which
implies that $ \{ m\in \mathbb{Z}: z_1^d(m)\ (\text{mod} \ \Z)\in
(-\eta,\eta)\}\supset S$. That is, $z_1^d(n) \in {\widetilde
{GP}}_d$.

Finally using the Claim and the fact that $z_1^d(n) \in {\widetilde
{GP}}_d$ we have $P(n;\alpha_1,\alpha_2,\ldots,\alpha_d)\in
{\widetilde {GP}}_d$ by Lemma \ref{lem-8-equi-new}. This ends the
proof, i.e. we have proved $\mathcal{F}_{d,0}\supset
\mathcal{F}_{GP_d}$.

\subsection{Proof of the Claim} Let $$u_i^k(n)=z_i^k(n)+f_i^k(n)=x_i^k(n)-\sum
\limits_{j=1}^{k-1}x_i^{k-j}(n)f_{i+k-j}^j(n) $$ for
 $1\le k\le d$
and $1\le i\le d-k+1$. Then $$f_i^k(n)=\lceil u_i^k(n) \rceil $$ for
 $1\le k\le d$
and $1\le i\le d-k+1$.

We define $U(n;j_1)=\frac{n^{j_1}}{j_1!}\prod_{r=1}^{j_1}\alpha_r$
for $1\le j_1\le d$ and recall that $a(\ell)=\sum_{t=1}^\ell j_t$. Then inductively for $\ell=2,3,\ldots,d$ we
define
\begin{align*}
U(n;j_1,j_2,\ldots,j_\ell)&=(U(n;j_1,\ldots,j_{\ell-1})-\lceil
U(n;j_1,\ldots,j_{\ell-1})\rceil )
\frac{n^{j_\ell}}{j_\ell!}\prod_{r=1}^{j_\ell}
\alpha_{a(\ell-1)+r}\\
&=(U(n;j_1,\ldots,j_{\ell-1})-\lceil
U(n;j_1,\ldots,j_{\ell-1})\rceil
)L(\frac{n^{j_\ell}}{j_\ell!}\prod_{r_{\ell}=1}^{j_\ell}
\alpha_{a(\ell-1)+r_{\ell}})
\end{align*}
for $j_1,j_2,\ldots,j_\ell\ge 1$ and $j_1+\ldots+j_\ell\le d$ (see
\eqref{eq-de-L} for the definition of $L$).

Next, $U(n;d)=\frac{n^{d}}{d!}\prod
\limits_{r=1}^{d}\alpha_r=L(\frac{n^{d}}{d!}\prod
\limits_{r=1}^{d}\alpha_r)$ and for $2\le \ell \le d$,
$j_1,j_2,\ldots,j_\ell\in \mathbb{N}$ with
$j_1+j_2+\ldots+j_\ell=d$, by Lemma \ref{lem8-13}(1)

\begin{align*}
U(n;j_1,j_2,\ldots,j_\ell)&=(U(n;j_1,\ldots,j_{\ell-1})-\lceil
U(n;j_1,\ldots,j_{\ell-1}) \rceil)
L(\frac{n^{j_\ell}}{j_\ell!}\prod_{r_{\ell}=1}^{j_\ell}
\alpha_{a(\ell-1)+r_{\ell}})\\&\simeq_d
U(n;j_1,\ldots,j_{\ell-1}) \lceil
L(\frac{n^{j_\ell}}{j_\ell!}\prod_{r_{\ell}=1}^{j_\ell}
\alpha_{a(\ell-1)+r_{\ell}})\rceil
\end{align*}
which is equal to
\begin{align*}
&(U(n;j_1,\ldots,j_{\ell-2})-\lceil U(n;j_1,\ldots,j_{\ell-2}) \rceil)\times
L(\frac{n^{j_{\ell-1}}}{j_{\ell-1}!}\prod_{r_{\ell-1}=1}^{j_{\ell-1}}\alpha_{a(\ell-2)+r_{\ell-1}},
\frac{n^{j_\ell}}{j_\ell!}\prod_{r_{\ell}=1}^{j_\ell}
\alpha_{a(\ell-1)+r_{\ell}})\\
&\simeq_d U(n;j_1,\ldots,j_{\ell-2}) \lceil
L(\frac{n^{j_{\ell-1}}}{j_{\ell-1}!}\prod_{r_{\ell-1}=1}^{j_{\ell-1}}\alpha_{a(\ell-2)+r_{\ell-1}},
\frac{n^{j_\ell}}{j_\ell!}\prod_{r_{\ell}=1}^{j_\ell}
\alpha_{a(\ell-1)+r_{\ell}})\rceil.
\end{align*}
Continuing the above argument we have
$$U(n;j_1,j_2,\ldots,j_\ell)\simeq_d L\Big(
\frac{n^{j_1}}{j_1!}\prod \limits_{r_1=1}^{j_1}\alpha_{r_1},\,
\frac{n^{j_2}}{j_2!}\prod \limits_{r_2=1}^{j_2}\alpha_{j_1+r_2},\,
\ldots,\, \frac{n^{j_\ell}}{j_\ell!}\prod
\limits_{r_\ell=1}^{j_\ell}\alpha_{a(\ell-1)+r_\ell}
\Big).$$

That is, for $1\le \ell \le d$, $j_1,j_2,\ldots,j_\ell\in
\mathbb{N}$ with $j_1+j_2+\ldots+j_\ell=d$,

\begin{equation}\label{eq-u=l}
U(n;j_1,j_2,\ldots,j_\ell) \simeq_d L\Big(
\frac{n^{j_1}}{j_1!}\prod \limits_{r_1=1}^{j_1}\alpha_{r_1},\,
\frac{n^{j_2}}{j_2!}\prod \limits_{r_2=1}^{j_2}\alpha_{j_1+r_2},\,
\ldots,\, \frac{n^{j_\ell}}{j_\ell!}\prod
\limits_{r_\ell=1}^{j_\ell}\alpha_{a(\ell-1)+r_\ell}
\Big).
\end{equation}

Thus using (\ref{eq-u=l}) we have
\begin{equation}\label{eq-8-14}
P(n;\alpha_1,\alpha_2,\ldots,\alpha_d)\simeq_d\sum
\limits_{\ell=1}^d \sum \limits_{j_1,\ldots j_\ell\in
\mathbb{N}\atop j_1+\ldots+j_\ell=d}(-1)^{\ell-1}
U(n;j_1,j_2,\ldots,j_{\ell}).
\end{equation}

Next using Lemma \ref{lem8-13}(1), for any $j_1,\ldots,j_\ell\in
\mathbb{N}$ with $a(\ell)\le d-1$, we have
\begin{align*}
&U(n;j_1,\ldots,j_\ell)f_{1+a(\ell)}^{d-a(\ell)}(n)=
U(n;j_1,\ldots,j_\ell) \lceil u_{1+a(\ell)}^{d-a(\ell)} (n) \rceil\\
&\simeq_d \Big( U(n;j_1,\ldots,j_\ell)-\lceil
U(n;j_1,\ldots,j_\ell)\rceil \Big)
u_{1+a(\ell)}^{d-a(\ell)} (n)\\
&=\Big( U(n;j_1,\ldots,j_\ell)-\lceil U(n;j_1,\ldots,j_\ell)\rceil \Big) \times
\Big( x_{1+a(\ell)}^{d-a(\ell)}(n)-\sum \limits_{j_{\ell+1}=1}^{d-(a(\ell))-1}
x_{1+a(\ell)}^{j_{\ell+1}}(n) f_{1+a(\ell+1)}^{d-a(\ell+1)} (n) \Big)\\
&=\Big( U(n;j_1,j_2,\ldots,j_\ell)-\lceil U(n;j_1,j_2,\ldots,j_\ell)\rceil \Big) \times\\
&\hskip1.6cm \Bigg(\tbinom{n}{d-a(\ell)} \prod \limits_{r_{\ell+1}=1}^{d-a(\ell)}
\alpha_{a(\ell)+r_{\ell+1}}-\sum_{j_{\ell+1}=1}^{d-a(\ell)-1} \tbinom{n}{j_{\ell+1}} \prod
\limits_{r_{\ell+1}=1}^{j_{\ell+1}}\alpha_{a(\ell)+r_{\ell+1}} f_{1+a(\ell+1)}^{d-a(\ell+1)}(n) \Bigg)\\
&\simeq_d \Big( U(n;j_1,j_2,\ldots,j_\ell)-\lceil U(n;j_1,j_2,\ldots,j_\ell)\rceil \Big) \times \\
&\hskip1.6cm \Bigg( \frac{n^{d-a(\ell)}}{(d-a(\ell))!} \prod \limits_{r_{\ell+1}=1}^{d-a(\ell)}
\alpha_{a(\ell)+r_{\ell+1}}-\sum_{j_{\ell+1}=1}^{d-a(\ell)-1} \frac{n^{j_{\ell+1}}}{j_{\ell+1}!} \prod
\limits_{r_{\ell+1}=1}^{j_{\ell+1}}\alpha_{a(\ell)+r_{\ell+1}} f_{1+a(\ell+1)}^{d-a(\ell+1)}(n)
\Bigg)\\
&=U(n;j_1,\ldots,j_\ell,d-a(\ell))-\sum_{j_{\ell+1}=1}^{d-a(\ell)-1}U(n;j_1,\ldots,j_\ell,j_{\ell+1})f_{1+a(\ell+1)}^{d-a(\ell+1)}(n).
\end{align*}
Using this fact and Lemma \ref{lem8-13}(1), we have
\begin{align*}
z_1^d(n)&\simeq_d u_1^d(n)=x_1^d(n)- \sum
\limits_{j_1=1}^{d-1}x_1^{j_1}(n)f_{1+j_1}^{d-j_1}(n) \\
&=\tbinom{n}{d}\alpha_1 \alpha_2\ldots \alpha_d- \sum
\limits_{j_1=1}^{d-1}\tbinom{n}{j_1}\alpha_1 \alpha_2\ldots \alpha_{j_1} f_{1+j_1}^{d-j_1}(n)
\simeq_d U(n;d)-\sum
\limits_{j_1=1}^{d-1}U(n;j_1) f_{1+j_1}^{d-j_1}(n) \\
&\simeq_d U(n;d)-\Big( \sum
\limits_{j_1=1}^{d-1}(U(n;j_1,d-j_1)-\sum \limits_{j_2=1}^{d-j_1-1}
U(n;j_1,j_2) f_{1+j_1+j_2}^{d-(j_1+j_2)}(n)) \Big).
\end{align*}

Continuing this argument we obtain
\begin{align*}
z_1^d(n)&\simeq_d \sum \limits_{\ell=1}^d \sum
\limits_{j_1,\ldots,j_\ell\in \mathbb{N}\atop j_1+\ldots+j_{\ell}}
(-1)^{\ell-1} U(n;j_1,\ldots,j_\ell).
\end{align*}

Combining this with \eqref{eq-8-14}, we have proved the Claim.

\chapter{Generalized polynomials and recurrence sets: Proof of Theorem C}
\label{section-proof3}

In this chapter we will prove Theorem C. That is, we will show that
for $d\in\N$ and $F\in \F_{GP_d}$, there exist a minimal $d$-step
nilsystem $(X,T)$ and a nonempty open set $U$ such that
$$F\supset \{n\in\Z: U\cap T^{-n}U\cap \ldots\cap T^{-dn}U\neq \emptyset \}.$$


Let us explain the idea of the proof of Theorem C. Put
\begin{eqnarray}\label{nnd}\index{$\NN_d$}
& &\NN_d=\{B\subset \Z: \text{there are a minimal $d$-step
nilsystem $(X,T)$ and an open} \\ \nonumber & & \text{non-empty set
$U$ of $X$ with $B\supset \{n\in\Z: U\cap T^{-n}U\cap \ldots\cap
T^{-dn}U\neq \emptyset \}$.} \}
\end{eqnarray} Similar to the
proof of Theorem B(2) we first show that
$\mathcal{F}_{GP_d}\subset \NN_d$ if and only if $\{ n\in
\mathbb{Z}: ||P(n;\alpha_1,\ldots,\alpha_d)||<\epsilon\}\in \NN_d$
for any $\alpha_1,\ldots,\alpha_d\in \mathbb{R}$ and $\epsilon>0$.
We choose $(X,T)$ as the closure of the orbit of $\Gamma$ in
$\G_d/\Gamma$ (the nilrotation is induced by a matrix $A\in \G_d$),
define $U\subset X$ depending on a given $\ep>0$, put $S=\{n\in\Z:
\bigcap_{i=0}^dT^{-in}U\neq \emptyset\}$; and consider the most
right-corner entry $z_1^d(m)$ in $A^{nm}BC_m$ with $B\in\G_d$ and
$C_m\in\Gamma$ for a given $n\in S$ with $1\le m\le d$. We finish
the proof by showing $S\subset \{ n\in \mathbb{Z}:
||P(n;\alpha_1,\ldots,\alpha_d)||<\epsilon\}$ which implies that $\{
n\in \mathbb{Z}: ||P(n;\alpha_1,\ldots,\alpha_d)||<\epsilon\}\in
\NN_d$.

\section{A special case and preparation}

\subsection{The ordinary polynomial case}

To illustrate the idea of the proof of Theorem C, we first consider
the situation when the generalized polynomials are the ordinary
ones. That is, we want to explain if $p(n)$ is a polynomial of
degree $d$ with $p(0)=0$ and $\ep>0$, how we can find a $d$-step
nilsystem $(X,T)$, and a nonempty open set $U\subset X$ such that
\begin{equation}\label{polynomial}
\{n\in \Z: p(n)\ (\text{mod}\ \Z)\in (-\ep,\ep)\}\supset \{n\in\Z: U
\cap T^{-n}U \cap \ldots\cap T^{-dn}U\neq \emptyset \}.
\end{equation}
To do this define $T_{\alpha,d}:\T^d\lra \T^d$ by
$$T_{\alpha,d}(\theta_1,\theta_2,\ldots,\theta_d)=(\theta_1+\alpha, \theta_2+
\theta_1,\theta_3+ \theta_2, \ldots,\theta_d+ \theta_{d-1}),$$ where
$\alpha\in \mathbb{R}$. A simple computation yields that
\begin{align}\label{computation}
T_{\alpha,d}^n(\theta_1,\ldots,\theta_d)=(\theta_1+n\alpha,
n\theta_1+\theta_2+ \frac{1}{2}n(n-1)\alpha,\ldots,
\sum_{i=0}^d\tbinom{n}{d-i} \theta_i),
\end{align} where
$\theta_0=\alpha$, $n\in \mathbb{Z}$ and $\tbinom{n}{0}=1$,
$\tbinom{n}{i}:=\frac{\prod_{j=0}^{i-1} (n-j)}{i!}$ for
$i=1,2,\ldots,d$.

\medskip
We now prove (\ref{polynomial}) by induction. The case when $d=1$ is
easy, and we assume that for each polynomial of degree $\le d-1$
(\ref{polynomial}) holds. Now let $p(n)=\sum_{i=1}^d\alpha_in^i$
with $\alpha_i\in\R$. By induction for each $1\le i\le d-1$ there is
an $i$-step nilsystem $(X_i,T_i)$ and an open non-empty subset $U_i$
of $X_i$ such that
$$\{n\in\Z: \alpha_i n^i \ ({\rm mod}\ \Z)\in (-\tfrac{\ep}{d},\tfrac{\ep}{d})\}\supset
\{n\in\Z: U_i\cap T_i^{-n}U_i\cap \ldots\cap T_i^{-dn}U_i\neq
\emptyset \}.$$

By the Vandermonde's formula, we know
$$ \left(
  \begin{array}{ccccc}
    1 & 2 & 3 & \ldots & d \\
    1 & 2^2 & 3^2& \ldots & d^2 \\
    \vdots & \vdots & \vdots & \vdots & \vdots \\
    1 & 2^{d-1} & 3^{d-1} & \ldots & d^{d-1} \\
    1 & 2^d & 3^d & \ldots & d^d \\
  \end{array}
\right)$$ is a non-singular matrix. Hence there are integers
${\lambda} _1, {\lambda}_2,\ldots, {\lambda}_d$ and ${\lambda} \in
\N$ such that the following equation holds:
$$
\left(
  \begin{array}{ccccc}
    1 & 2 & 3 & \ldots & d \\
    1 & 2^2 & 3^2& \ldots & d^2 \\
    \vdots & \vdots & \vdots & \vdots & \vdots \\
    1 & 2^{d-1} & 3^{d-1} & \ldots & d^{d-1} \\
    1 & 2^d & 3^d & \ldots & d^d \\
  \end{array}
\right)\left(
         \begin{array}{c}
           {\lambda}_1 \\
            {\lambda}_2 \\
           \vdots \\
            {\lambda}_{d-1} \\
            {\lambda}_d \\
         \end{array}
       \right)=\left(
                 \begin{array}{c}
                   0 \\
                   0 \\
                   \vdots \\
                   0 \\
                   {\lambda} \\
                 \end{array}
               \right).
$$
That is,
\begin{equation}\label{a1}
\begin{split}
      & \sum_{m=1}^d {\lambda}_m m^j={\lambda}_1+{\lambda}_22^j+\ldots+ {\lambda}_dd^j=0, \ 1\le j\le d-1;\\
       & \sum_{m=1}^d {\lambda}_mm^d={\lambda}_1+{\lambda}_22^d+\ldots+{\lambda}_d d^d={\lambda}.
\end{split}
\end{equation}
Now let $T_d=T_{\frac{\alpha_d}{\lambda},d}$ and $Y_d=\T^d$. Let
$K_d=d!\sum_{i=1}^d|\lambda_i|$, $\ep_1>0$ with $2K_d\ep_1<\ep/d$ and
$U_d=(-\ep_1,\ep_1)^d$.

It is easy to see that if $n\in \{n\in\Z: U_d\cap T_d^{-n}U_d\cap
\ldots\cap T_d^{-dn}U_d\neq \emptyset \}$ then we know that there is
$(\theta_1,\ldots,\theta_d)\in U_d$ such that
$T_d^{in}(\theta_1,\ldots,\theta_d)\in U_d$ for each $1\le i\le d$.
Thus, by (\ref{computation}) considering the last coordinate we ge
that
{\footnotesize
\begin{align*}
\tbinom{n}{d}\theta_0+\tbinom{n}{d-1}\theta_1+\ldots+\tbinom{n}{0}\theta_d
\ (\text{\rm mod} \ \Z)&\in
(-\ep_1,\ep_1)\\
\tbinom{2n}{d}\theta_0+\tbinom{2n}{d-1}\theta_1+\ldots+\tbinom{2n}{0}\theta_d
\ (\text{\rm mod} \ \Z)&\in
(-\ep_1,\ep_1)\\
\ldots \ldots \ldots \ldots \ \ \ \ & \ldots\\
\tbinom{dn}{d}\theta_0+\tbinom{dn}{d-1}\theta_1+\ldots+\tbinom{dn}{0}\theta_d
\ (\text{\rm mod} \ \Z)&\in (-\ep_1,\ep_1),
\end{align*}}
where $\theta_0=\frac{\alpha_d}{\lambda}$. Multiplying
$\tbinom{in}{d}\theta_0+\tbinom{in}{d-1}\theta_1+\ldots+\tbinom{in}{0}\theta_d$
by $\lambda_id!$ and summing over $i=1,\ldots,d$ { we get that
$$ \sum_{j=1}^d\lambda_jd!\sum_{i=0}^d\tbinom{jn}{d-i}\theta_i=\theta_d(\sum_j \lambda_j) d! +\alpha_d n^d
 \ (\text{\rm mod} \ \Z)\in (-K_d\ep_1,K_d\ep_1).$$
Thus $\alpha_d n^d
 \ (\text{\rm mod} \ \Z)\in (-2K_d\ep_1,2K_d\ep_1)\subset (-\ep/d,\ep/d)$. }

Choose $x_i\in U_i$ for $1\le i\le d$. Let
$x=(x_1,x_2,\ldots,x_d)\in X_1\times \ldots\times X_d$ and $X$ be
the orbit closure of $x$ under $T=T_1\times T_2\ldots\times T_d$. Then $(X,T)$
is a $d$-step nilsystem. If we let $U=(U_1\times U_2\times
\ldots\times U_d)\cap X$, then we have (\ref{polynomial}).

\medskip
By the property of nilsystems and the discussion above it is easy to
see
\begin{rem}\label{corr} Let $k\in\N$, $q_i(x)$ be a polynomial of degree $d$ with
$q_i(0)=0$ and $\ep_i>0$ for $1\le i\le k$. Then there are a
$d$-step nilsystem $(X,T,\mu)$ and $B\subset X$ with $\mu(B)>0$ such
that
$$\bigcap_{i=1}^k\{n\in \Z: ||q_i(n)||<\ep_i \}\supset \{n\in\Z: \mu(B\cap
T^{-n}B\cap \ldots\cap T^{-dn}B)>0\}$$
\end{rem}


\subsection{Some preparation}

Recall that for $d\in \N$, $\NN_d$ is defined in (\ref{nnd}). Hence
Theorem C is equivalent to $$\F_{GP_d}\subset \NN_d.$$

\begin{lem}\label{lem-a1}
For each $d\in \N$, $\NN_d$ is a filter.
\end{lem}

\begin{proof}
Let $B_1, B_2\in \NN_d$. To show $\NN_d$ is a filter, it suffices to
show $B_1\cap B_2\in \NN_d$. By definition, there exist minimal
$d$-step nilsystems $(X_i,T_i)$, and nonempty open sets $U_i$ for
$i=1,2$ such that
$$B_i\supset \{n\in\Z: U_i\cap T_i^{-n}U_i\cap
\ldots\cap T_i^{-dn}U_i\neq \emptyset \}.$$ Taking any minimal point
$x=(x_1,x_2)\in X_1\times X_2$, let $X=\overline{{\O}(x,T)}$,
where $T=T_1\times T_2$. Note that $(X,T)$ is also a minimal
$d$-step nilsystem.

Since $(X_i, T_i), i=1,2$, are minimal, there are $k_i\in \N$ such
that $x_i\in T_i^{-k_i}U_i$, $i=1,2$. Let $U=(T_1^{-k_1}U_1\times
T_2^{-k_2}U_2 )\cap X$, then $U$ is an open set of $X$. Note that
\begin{equation*}
\begin{split}&\hskip0.6cm \{n\in\Z: U\cap T^{-n}U\cap \ldots\cap T^{-dn}U\neq \emptyset \}
       \\&=\bigcap_{i=1,2}\{n\in\Z: T_i^{-k_i}U_i\cap T_i^{-(k_i+n)}U_i\cap \ldots
       \cap T_i^{-(k_i+dn)}U_i\neq \emptyset \}\\
       &=\bigcap_{i=1,2} \{n\in\Z: U_i\cap T_i^{-n}U_i\cap \ldots\cap T_i^{-dn}U_i\neq \emptyset \}
\end{split}
\end{equation*}
Hence $$B_1\cap B_2\supset \{n\in\Z: U\cap T^{-n}U\cap \ldots\cap
T^{-dn}U\neq \emptyset \}.$$ That is, $B_1\cap B_2\in \NN_d$ and
$\NN_d $ is a filter.
\end{proof}

\begin{de}\index{$\widehat{GP}_r$}
For $r\in \N$, define
\begin{equation*}
    \widehat{GP}_r=\{p(n)\in GP_r: \{n\in \Z: p(n) \ ({\rm mod}\ \Z) \in (-\ep,\ep)\}\in \NN_r, \forall
    \ep>0\}.
\end{equation*}
\end{de}

\begin{rem}\label{rem-a1}
It is clear that for $p(n)\in GP_r$, $p(n)\in \widehat{GP}_r$ if and
only if $-p(n)\in \widehat{GP}_r$. Since $\NN_r$ is a filter, if
$p_1(n),p_2(n),\ldots,p_k(n)\in \widehat{GP}_r$ then
$$p_1(n)+p_2(n)+\ldots+p_k(n)\in \widehat{GP}_r.$$ Moreover by the definition
of $\widehat{GP}_r$, we know that $\mathcal{F}_{GP_r}\subset \NN_r$
if and only if $\widehat{GP}_r=GP_r$.
\end{rem}

We shall prove Theorem C inductively, thus we need to
obtain some results under the following assumption, that is for some
$d\ge 2$,
\begin{equation}\label{induction-d-1}
\F_{GP_{d-1}}\subset \NN_{d-1}.
\end{equation}

\begin{lem}\label{lem-a2}
Let $p(n),q(n)\in GP_d$ with $p(n)\simeq_d q(n)$. Under the
assumption \eqref{induction-d-1}, $p(n)\in \widehat{GP}_d$ if and
only if $q(n)\in \widehat{GP}_d$.
\end{lem}

\begin{proof}
It follows from Lemma \ref{lem-8-11}, $\NN_d$ being a filter and
$\F_{GP_{d-1}}\subset \NN_{d-1}\subset \NN_d$.
\end{proof}

\begin{thm} \label{thm-a1}
Under the assumption \eqref{induction-d-1}, the following properties
are equivalent:
\begin{enumerate}
\item  $\mathcal{F}_{GP_d}\subset \NN_d$.

\item $P(n;\alpha_1,\alpha_2,\ldots,\alpha_d)\in \widehat{GP}_d$ for any $\alpha_1,\alpha_2,\ldots,\alpha_d\in \mathbb{R}$, that is
$$\{ n\in \mathbb{Z}: P(n;\alpha_1,\alpha_2,\ldots,\alpha_d) \ (\text{\rm mod} \ \Z) \in (-\epsilon, \epsilon)\}\in \NN_d$$ for any
$\alpha_1,\alpha_2,\ldots,\alpha_d\in \mathbb{R}$ and $\epsilon>0$.

\item $\text{SGP}_d\subset \widehat{GP}_d$.
\end{enumerate}
\end{thm}

\begin{proof}
The proof is similar to that of Theorem \ref{lem-trans}.
\end{proof}

\section{Proof of Theorem C}

Now we prove $\F_{GP_d}\subset \NN_d$ by induction on $d$. When
$d=1$, since $\F_{GP_1}=\F_{SGP_1}$ and $\NN_d$ is a filer, it is
sufficient to show that: for any $p(n)=an\in SGP_1$ and $\ep>0$, we
have
\begin{equation}\label{}
    \{n\in \Z: p(n) \ ({\rm mod}\ \Z) \in (-\ep, \ep)\}\in \NN_1.
\end{equation}
This is easy to be verified.

Now we assume that for $d\ge 2$, $\F_{GP_{d-1}}\subset \NN_{d-1}$,
i.e. (\ref{induction-d-1}) holds.
Then it follows from Theorem \ref{thm-a1} that under the assumption
\eqref{induction-d-1},  to show $\mathcal{F}_{GP_d}\subset \NN_d$,
it is sufficient to show that
\begin{equation*}
    P(n;\b_1,\b_2,\ldots,\b_d)\in \widehat{GP}_d,
\end{equation*} for any $\beta_1,\beta_2,\ldots,\beta_d\in \mathbb{R}$.

Fix $\b_1,\b_2,\ldots,\b_d\in \R$. We divide the remainder of the
proof into two steps.

\medskip
\noindent{\bf Step 1.} We are going to show
\begin{equation*}
P(n;\b_1,\b_2,\ldots,\b_d)\simeq_d \sum \limits_{\ell=1}^d \sum
\limits_{j_1,\ldots j_\ell\in \mathbb{N}\atop
j_1+\ldots+j_\ell=d}(-1)^{\ell-1} {\lambda}
U(n;j_1,j_2,\ldots,j_{\ell}),
\end{equation*}
where as in the proof of Theorem B, we define
\begin{equation}\label{}
U(n;j_1)=\frac{n^{j_1}}{j_1!}\prod_{r=1}^{j_1}\alpha_r, \ 1\le
j_1\le d.
\end{equation}
And inductively for $\ell=2,3,\ldots,d$ define
\begin{align*}
U(n;j_1,j_2,\ldots,j_\ell) &=(U(n;j_1,\ldots,j_{\ell-1})-\lceil
U(n;j_1,\ldots,j_{\ell-1})\rceil )
\frac{n^{j_\ell}}{j_\ell!}\prod_{r=1}^{j_\ell}
\alpha_{\sum_{t=1}^{\ell-1}j_t+r}\\
&=(U(n;j_1,\ldots,j_{\ell-1})-\lceil
U(n;j_1,\ldots,j_{\ell-1})\rceil
)L(\frac{n^{j_\ell}}{j_\ell!}\prod_{r_{\ell}=1}^{j_\ell}
\alpha_{\sum_{t=1}^{\ell-1}j_t+r_{\ell}})
\end{align*}
for $j_1,j_2,\ldots,j_\ell\ge 1$ and $j_1+\ldots+j_\ell\le d$ (see
\eqref{eq-de-L} for the definition of $L$).

\medskip

In fact, let ${\lambda} _1, {\lambda}_2,\ldots, {\lambda}_d\in
\mathbb{Z}$ and ${\lambda} \in \N$ satisfying \eqref{a1}. Put
$$\a_1=\b_1/{\lambda}, \a_2=\b_2, \a_3=\b_3,\ldots, \a_d=\b_d.$$ Then
\begin{equation*}
    P(n;\b_1,\b_2,\ldots,\b_d)={\lambda} P(n;\a_1,\a_2,\ldots,\a_d).
\end{equation*}

Note that in proof of Theorem B we have
\begin{equation}\label{}
P(n;\alpha_1,\alpha_2,\ldots,\alpha_d)\simeq_d \sum
\limits_{\ell=1}^d \sum \limits_{j_1,\ldots j_\ell\in
\mathbb{N}\atop j_1+\ldots+j_\ell=d}(-1)^{\ell-1}
U(n;j_1,j_2,\ldots,j_{\ell}).
\end{equation}
Since ${\lambda}$ is an integer, we have
\begin{equation*}
{\lambda} P(n;\alpha_1,\alpha_2,\ldots,\alpha_d)\simeq_d \sum
\limits_{\ell=1}^d \sum \limits_{j_1,\ldots j_\ell\in
\mathbb{N}\atop j_1+\ldots+j_\ell=d}(-1)^{\ell-1} {\lambda}
U(n;j_1,j_2,\ldots,j_{\ell}).
\end{equation*}
That is,
\begin{equation*}
P(n;\b_1,\b_2,\ldots,\b_d)\simeq_d \sum \limits_{\ell=1}^d \sum
\limits_{j_1,\ldots j_\ell\in \mathbb{N}\atop
j_1+\ldots+j_\ell=d}(-1)^{\ell-1} {\lambda}
U(n;j_1,j_2,\ldots,j_{\ell}).
\end{equation*}
Hence, by Lemma \ref{lem-a2}, to show $P(n;\b_1,\b_2,\ldots,\b_d)\in
\widehat{GP}_d$, it suffices to show
\begin{equation}\label{a2}
\sum \limits_{\ell=1}^d \sum \limits_{j_1,\ldots j_\ell\in
\mathbb{N}\atop j_1+\ldots+j_\ell=d}(-1)^{\ell-1} {\lambda}
U(n;j_1,j_2,\ldots,j_{\ell})\in \widehat{GP}_d.
\end{equation}

Now choose ${\bf x}=(x_i^k)_{1\le k\le d, 1\le i\le d-k+1}\in
\mathbb{R}^{d(d+1)/2}$ with $x_i^1=\alpha_i$ for $i=1,2,\ldots,d$
and $x_i^k=0$ for $2\le k\le d$ and $1\le i\le d-k+1$. Let
{\footnotesize
$$
A=\M({\bf x})=\left(
  \begin{array}{cccccccc}
    1 & \alpha_1 & 0                &\ldots & 0 &0\\
    0 & 1     & \alpha_2            &\ldots &  0&0\\
    \vdots & \vdots &   \vdots & \vdots &\vdots\\
    0 & 0    &0                    & \ldots & \alpha_{d-1} &0 \\
    0 & 0    &0                    & \ldots & 1 &\alpha_d \\
    0 & 0    &0                   & \ldots & 0  & 1
  \end{array}
\right).
$$}

{ For $n\in \mathbb{\mathbb{Z}}$, if ${\bf x}(n)=(x_i^k(n))_{1\le k\le d, 1\le
i\le d-k+1}\in \mathbb{R}^{d(d+1)/2}$ satisfies $\M({\bf
x}(n))=A^n$, then $x_i^k(n)$ is a polynomial of $n$ for $1\le k\le d$ and $1\le i\le d-k+1$.
Moreover, by Lemma \ref{lem-8-ite} and Remark
\ref{rem-8-ite}, when $n\in \mathbb{Z}$
\begin{equation}\label{a3}
x_i^k(n)=\tbinom{n}{k}\alpha_i \alpha_{i+1}\ldots \alpha_{i+k-1}
\end{equation}
for $1\le k\le d$ and $1\le i\le d-k+1$.}

\medskip

Let $X=\G_d/\Gamma$  and $T$ be the nilrotation induced by $A\in\G_d$,
i.e. $B\Gamma\mapsto AB\Gamma$ for $B\in\G_d$. Since $\G_d$ is a
$d$-step nilpotent Lie group and $\Gamma$ is a uniform subgroup of
$\G_d$, $(X,T)$ is a $d$-step nilsystem. Let $x_0=\Gamma\in X$ and
$Z$ be the closure of the orbit $\O (x_0,T)$ of $x_0$ in $X$.
Then $(Z,T)$ is a minimal $d$-step nilsystem.

\medskip
\noindent{\bf Step 2.} For any $\ep>0$, we are going to show there
is a nonempty open subset $U$ of $Z$ such that
\begin{equation}\label{a0}
\begin{split}
& \ \{n\in \Z: \sum \limits_{\ell=1}^d \sum \limits_{j_1,\ldots
j_\ell\in \mathbb{N}\atop j_1+\ldots+j_\ell=d}(-1)^{\ell-1}
{\lambda}
U(n;j_1,j_2,\ldots,j_{\ell}) \ ({\rm mod}\ \Z)\in (-\ep,\ep)\}\\
&\supset \{n\in\Z: U\cap T^{-n}U\cap \ldots\cap T^{-dn}U\neq
\emptyset\}.
\end{split}
\end{equation}
That means $\sum \limits_{\ell=1}^d \sum \limits_{j_1,\ldots
j_\ell\in \mathbb{N}\atop j_1+\ldots+j_\ell=d}(-1)^{\ell-1}
{\lambda} U(n;j_1,j_2,\ldots,j_{\ell})\in \widehat{GP}_d$.

\medskip

Fix an $\ep>0$. Take $\displaystyle
\ep_1=\min\{\tfrac{\ep}{2K(\sum_{i=0}^{d-1}d^i)}, \tfrac 14\}$,
where $\displaystyle
K=\sum_{m=1}^d|\lambda_m|\Big(\sum_{t=0}^dm^t\Big)$, and let
{  $V=\{C\in \G_d: \|C-I\|_\infty< \ep_1\}$ be a neighborhood of $I$ in $\G_d$.
Put $U=V\Gamma\cap Z$. Then $U$ is an open neighborhood of $x_0$ in $Z$.
Let
\begin{equation*}
S=\{n\in \Z: U\cap T^{-n}U\cap \ldots \cap T^{-dn}U\neq \emptyset
\}.
\end{equation*}}

Now we show that
\begin{equation*}
   S\subset \Big\{n\in \Z: \sum \limits_{\ell=1}^d \sum \limits_{j_1,\ldots,
j_\ell\in \mathbb{N}\atop j_1+\ldots+j_\ell=d}(-1)^{\ell-1}
{\lambda} U(n;j_1,j_2,\ldots,j_{\ell}) \ ({\rm mod}\ \Z)\in
(-\ep,\ep)\Big\}.
\end{equation*}
Let $n\in S$. Then $U\cap T^{-n}U\cap \ldots \cap T^{-dn}U\neq
\emptyset$. Hence there is some $B\in \G_d$ with
\begin{equation*}
    B\Gamma \in U\cap T^{-n}U\cap \ldots \cap T^{-dn}U.
\end{equation*}
Thus
$A^{mn}B\Gamma\in V\Gamma, \ m=0,1,2,\ldots, d$. We may assume that
$B\in V$.

For each $m\in \{1,2,\ldots, d\}$, since $A^{mn}B\Gamma\in V\Gamma$ there is some $C_m\in \Gamma$ such that
\begin{equation}\label{a5-1}
    \|A^{mn}BC_m-I\|_\infty<\ep_1.
\end{equation}
Let $A^{mn}BC_m=\M({\bf z}(m))$, where ${\bf z}(m)=(z_i^k(m))_{1\le
k\le d, 1\le i\le d-k+1}\in \mathbb{R}^{d(d+1)/2}$. Then from
\eqref{a5-1}, we have
$$|z_i^k(m)|<\ep_1, \quad 1\le k\le d, 1\le
i\le d-k+1. $$

On the one hand, since $|z_1^d(m)|<\ep_1$, we have
\begin{equation}\label{a8}
    \sum_{m=1}^d {\lambda}_m z_1^d(m)\in (-K\ep_1, K\ep_1).
\end{equation}

On the other hand, we have

\medskip
\begin{align}\label{longproof-1}
\begin{split}
\sum_{m=1}^d{\lambda}_mz_1^d(m) \approx&
\Big(\sum_{l=1}^d(-1)^{l-1}\sum_{\substack{j_1,j_2,\ldots,j_l\in \N\\
j_1+j_2+\ldots +j_l=d}} {\lambda} U(n;j_1,j_2,\ldots,j_l)\Big)\\
&+\vartriangle\Big((d+d^2+\ldots+d^{d-1})(2K\ep_1)\Big).
\end{split}
\end{align}
Note that for $a,b\in \R$ and $\d>0$, $a\thickapprox b+
\vartriangle(\d)$ means that $a-b \ ({\rm mod}\ \Z)\in (- \d,\d)$.

\medskip
Since the proof of (\ref{longproof-1}) is long, we put it after
Theorem C. Now we continue the proof. By (\ref{longproof-1}) and
(\ref{a8}), we have {\footnotesize
\begin{equation*}
\begin{split}
\sum_{l=1}^d(-1)^{l-1}\sum_{\substack{j_1,j_2,\ldots,j_l\in \N\\
j_1+j_2+\ldots +j_l=d}}{\lambda} U(n;j_1,j_2,\ldots,j_l) \ ({\rm
mod}\ \Z)&\in\Big(-M(2K\ep_1),M(2K\ep_1)\Big)\subset (-\ep,\ep),
\end{split}
\end{equation*}}
where $M=1+d+\ldots+d^{d-1}$. This means that
\begin{equation*}
    n\in \Big\{q\in \Z:\sum_{l=1}^d\sum_{\substack{j_1,j_2,\ldots,j_l\in \N\\
j_1+j_2+\ldots +j_l=d}}(-1)^{l-1}{\lambda} U(q;j_1,j_2,\ldots,j_l) \
({\rm mod}\ \Z)\in (-\ep,\ep)\Big \}.
\end{equation*}
Hence
\begin{equation*}
   S\subset  \Big\{q\in \Z:\sum_{l=1}^d\sum_{\substack{j_1,j_2,\ldots,j_l\in \N\\
j_1+j_2+\ldots +j_l=d}}(-1)^{l-1}{\lambda} U(q;j_1,j_2,\ldots,j_l) \
({\rm mod}\ \Z)\in (-\ep,\ep)\Big \}.
\end{equation*}
Thus we have proved (\ref{a0}) which means $\sum \limits_{\ell=1}^d
\sum \limits_{j_1,\ldots j_\ell\in \mathbb{N}\atop
j_1+\ldots+j_\ell=d}(-1)^{\ell-1} {\lambda}
U(n;j_1,j_2,\ldots,j_{\ell})\in \widehat{GP}_d$. The proof of
Theorem C is now finished.

\subsection{Proof of (\ref{longproof-1})} Since $B\in V$,
\begin{equation}\label{a4}
    ||B-I||_\infty<\ep_1<1/2.
\end{equation}

Denote $B=\M({\bf y})$, where ${\bf y}=(y_i^k)_{1\le k\le d, 1\le
i\le d-k+1}\in \mathbb{R}^{d(d+1)/2}$. From (\ref{a4}),
$$|y_i^k|<\ep_1, \quad 1\le k\le d,\ 1\le
i\le d-k+1. $$ For $m=1,2,\ldots,d$, recall that $C_m\in \Gamma$
satisfies \eqref{a5-1}. Denote $C_m=\M({\bf h}(m))$, where ${\bf h}(m)=(-h_i^k(m))_{1\le
k\le d, 1\le i\le d-k+1}\in \mathbb{Z}^{d(d+1)/2}$.  Let
$A^{mn}B=\M({\bf w}(m))$, where ${\bf w}(m)=(w_i^k(m))_{1\le k\le d,
1\le i\le d-k+1}\in \mathbb{R}^{d(d+1)/2}$. Then
\begin{equation}\label{a6}
\begin{split}
w_i^k(m)&=x_i^k(mn)+\Big( \sum
\limits_{j=1}^{k-1}x_i^{j}(mn)y_{i+j}^{k-j} \Big)+y_i^k\\
&=\tbinom{mn}{k}\alpha_i \alpha_{i+1}\ldots
\alpha_{i+k-1}+\sum_{j=1}^{k-1}\tbinom{mn}{j}\alpha_i
\alpha_{i+1}\ldots \alpha_{i+j-1}y_{i+j}^{k-j}+y_i^k\\
&\triangleq \frac{(mn)^k}{k!}\a_i\ldots \a_{i+k-1}+\sum_{j=1}^{k-1}
m^j a_i^k(j)+a_i^k(0),
\end{split}
\end{equation}
where $m=1,2,\ldots, d$, $a_i^k(j)$ does not depend on $m$ and
$|a_i^k(0)|=|y_i^k|<\ep_1$.

Recall that ${\bf z}(m)=(z_i^k(m))_{1\le k\le d, 1\le i\le d-k+1}\in
\mathbb{R}^{d(d+1)/2}$ satisfies $A^{mn}BC_m=\M({\bf z}(m))$. Hence
\begin{equation}\label{addye}
z_i^k(m)= w_i^k(m)-\Big( \sum \limits_{j=1}^{k-1} w_i^{j}(m)
h_{i+j}^{k-j}(m) \Big) -h_i^k(m).
\end{equation}
From $\|A^{mn}BC_m-I\|_\infty<\ep_1$, we have
$$|z_i^k(m)|<\ep_1, \quad 1\le k\le d, 1\le
i\le d-k+1. $$ Note that $h_i^k(m)\in \Z$, and we have
$$h_i^k(m)=\Big \lceil w_i^k(m)- \sum
\limits_{j=1}^{k-1}w_i^{j}(m)h_{i+j}^{k-j}(m) \Big
\rceil.$$ Let
$$u_i^k(m)=w_i^k(m)- \sum
\limits_{j=1}^{k-1}w_i^{j}(m)h_{i+j}^{k-j}(m) .$$ Then
\begin{equation*}
    |u_i^k(m)-h_i^k(m)|=|z_i^k(m)|<\ep_1<1/2.
\end{equation*}

\medskip

Recall that for $a,b\in \R$ and $\d>0$, $a\thickapprox b+
\vartriangle(\d)$ means $a-b \ ({\rm mod}\ \Z)\in (- \d,\d)$.

\bigskip

\noindent {\bf Claim:} {\em Let $1\le r\le d-1$ and $v_r(0),
v_r(1),\ldots,v_r(r)\in \R$. Then for each $1\le r_1\le d-r-1$ and
$1\le j\le r_1+r$, there exist $v_{r,r_1}(j)\in \R$ such that
\begin{enumerate}
\item we have
\begin{equation*}
\begin{split}
&\sum_{m=1}^d  {\lambda}_m \Big(\sum_{t=0}^r m^t
v_r(t)\Big)h_{1+r}^{d-r}(m) \thickapprox {\lambda}(v_r(r)-\lceil
v_r(r)\rceil)\frac{n^{d-r}}{(d-r)!}\a_{1+r}\ldots\a_d \\
&\hskip3.cm-\sum_{r_1=1}^{d-r-1}\sum_{m=1}^d{\lambda}_m\Big(\sum_{t=0}^{r_1+r}m^tv_{r,r_1}(t)\Big)h_{1+r+r_1}^{d-r-r_1}(m)+
\vartriangle(2K\ep_1)
\end{split}
\end{equation*}

\item $\displaystyle v_{r,r_1}(r+r_1)=\Big( v_r(r)-\lceil v_r(r)\rceil\Big)\frac{n^{r_1}}{r_1!}
\a_{r+1}\ldots \a_{r+r_1}$ for all $1\le r_1\le d-r-1$.
\end{enumerate}}

\begin{proof}[Proof of Claim]
{ First we have
\begin{equation*}
\begin{split}
\Big | \sum_{m=1}^d {\lambda}_m\Big(\sum_{t=0}^rm^t(v_r(r)-\lceil
v_r(r)\rceil)\Big) \Big| \le \sum_{m=1}^d
|{\lambda}_m|\Big(\sum_{t=0}^rm^t\Big)=K.
\end{split}
\end{equation*}
Since $|u_{1+r}^{d-r}(m)-h_{1+r}^{d-r}(m)|<\ep_1$, we have}
\begin{equation}\label{a7}
\begin{split}&\sum_{m=1}^d {\lambda}_m \Big(\sum_{t=0}^r
m^tv_r(t)\Big)h_{1+r}^{d-r}(m)\\ &\thickapprox  \sum_{m=1}^d
{\lambda}_m \Big(\sum_{t=0}^r m^t (
v_r(t)-\lceil v_r(t)\rceil)\Big)h_{1+r}^{d-r}(m)\\
&\thickapprox  \sum_{m=1}^d {\lambda}_m \Big(\sum_{t=0}^r m^t (
v_r(t)-\lceil v_r(t)\rceil
)\Big)u_{1+r}^{d-r}(m)+\vartriangle(K\ep_1).
\end{split}
\end{equation}

Then we have
\begin{equation*}
\begin{split}
&\sum_{m=1}^d {\lambda}_m \Big(\sum_{t=0}^r m^t ( v_r(t)-\lceil
v_r(t)\rceil )\Big)u_{1+r}^{d-r}(m)\\
&=\sum_{m=1}^d {\lambda}_m \Big(\sum_{t=0}^r m^t ( v_r(t)-\lceil
v_r(t)\rceil)\Big)\Big(w_{1+r}^{d-r}(m)- \sum
\limits_{r_1=1}^{d-r-1}w_{1+r}^{r_1}(m)h_{1+r+r_1}^{d-r-r_1}(m)
\Big).
\end{split}
\end{equation*}

From (\ref{a6}) we have
\begin{equation*}
\begin{split}
&\sum_{m=1}^d {\lambda}_m \Big(\sum_{t=0}^r m^t ( v_r(t)-\lceil
v_r(t)\rceil)\Big)w_{1+r}^{d-r}(m)\\ &= \sum_{m=1}^d {\lambda}_m
\Big(\sum_{t=0}^r m^t ( v_r(t)-\lceil v_r(t)\rceil)\Big)
\Big(\frac{(mn)^{d-r}}{(d-r)!}\a_{1+r}\ldots
\a_{d}+\sum_{j=0}^{d-r-1} m^j a_{1+r}^{d-r}(j)\Big)\\
&=\sum_{t=0}^r \big( \sum_{m=1}^d {\lambda}_m m^{d-r+t}\big) \frac{n^{d-r}}{(d-r)!} \a_{1+r}\ldots\a_{d}\Big( v_r(t)-\lceil v_r(t)\rceil \Big)\\
&\hskip2.cm+ \sum_{h=1}^{d-1}\big(\sum_{m=1}^d {\lambda}_m m^h\big)
\Bigg(\sum_{0\le t\le r\atop{0\le j\le d-r-1\atop{t+j=h}}}(
v_r(t)-\lceil
v_r(t)\rceil )a_{1+r}^{d-r}(j)\Bigg)\\
&\hskip6.cm+\sum_{m=1}^d {\lambda}_m \Big( v_r(0)-\lceil
v_r(0)\rceil\Big)a_{1+r}^{d-r}(0),
\end{split}
\end{equation*}
and so $\sum_{m=1}^d {\lambda}_m \Big(\sum_{t=0}^r m^t (
v_r(t)-\lceil v_r(t)\rceil)\Big)w_{1+r}^{d-r}(m)$ is equal to

\begin{equation*}
\begin{split}
&  {\lambda} \frac{n^{d-r}}{(d-r)!}\a_{1+r}\ldots \a_{d} (
v_r(r)-\lceil v_r(r)\rceil)+(\sum_{m=1}^d {\lambda}_m )
( v_r(0)-\lceil v_r(0)\rceil)y_{1+r}^{d-r}\\
&\approx  {\lambda} \frac{n^{d-r}}{(d-r)!}\a_{1+r}\ldots \a_{d} (
v_r(r)-\lceil v_r(r)\rceil )+ \vartriangle (K\ep_1).
\end{split}
\end{equation*}
The last equation follows from
$$\Big |(\sum_{m=1}^d {\lambda}_m)
( v_r(0)-\lceil v_r(0)\rceil)y_{1+r}^{d-r}\Big |\le \sum_{m=1}^d
|{\lambda}_m|\ep_1<K\ep_1.$$ Then for $1\le r_1\le d-r-1$ and $m=1,2,\ldots, d$, by
(\ref{a6}), we have

\begin{equation*}
\begin{split}
& \Big(\sum_{t=0}^r m^t ( v_r(t)-\lceil
v_r(t)\rceil)\Big)w_{1+r}^{r_1}(m) \\
&=\Big(\sum_{t=0}^r m^t ( v_r(t)-\lceil
v_r(t)\rceil )\Big) \Big(\frac{(mn)^{r_1}}{(r_1)!}\a_{1+r}\ldots
\a_{r+r_1}+\sum_{j=0}^{r_1-1} m^j a_{1+r}^{r_1}(j)\Big)\\
&=\sum_{t=0}^r   m^{r_1+t}
\frac{n^{r_1}}{(r_1)!}\a_{1+r}\ldots \a_{r+r_1} ( v_r(t)-\lceil
v_r(t)\rceil)+ \\ &\hskip 2cm \sum_{h=0}^{r+r_1-1}m^h
\big(\sum_{0\le t\le r\atop{0\le j\le r_1-1\atop{t+j=h}}}\big(
v_r(t)-\lceil v_r(t)\rceil \big ) a_{1+r}^{r_1}(j)\big) \Big).
\end{split}
\end{equation*}
Let
\begin{equation*}
    v_{r,r_1}(h)=\sum_{0\le t\le
r\atop{0\le j\le r_1-1\atop{t+j=h}}}\big( v_r(t)-\lceil v_r(t)\rceil
\big ) a_{1+r}^{r_1}(j)
\end{equation*} for $0\le h\le r_1-1$,
\begin{align*}
    v_{r,r_1}(h)=&
\frac{n^{r_1}}{(r_1)!}\a_{1+r}\ldots \a_{r+r_1} ( v_r(h-r_1)-\lceil
v_r(h-r_1)\rceil)+\\
&\hskip1cm \sum_{0\le t\le
r\atop{0\le j\le r_1-1\atop{t+j=h}}}\big( v_r(t)-\lceil v_r(t)\rceil
\big ) a_{1+r}^{r_1}(j)
\end{align*}
for $r_1\le h\le r+r_1-1$ and \begin{equation*}
    v_{r,r_1}(r+r_1)=\tfrac{n^{r_1}}{(r_1)!}\a_{1+r}\ldots \a_{r+r_1} \left(
v_r(r)-\lceil v_r(r)\rceil\right).
\end{equation*}
Thus
\begin{equation*}
\begin{split}
& \sum_{r_1=1}^{d-r-1} \Big(\sum_{m=1}^d {\lambda}_m \Big(\sum_{t=0}^r m^t ( v_r(t)-\lceil
v_r(t)\rceil)\Big)w_{1+r}^{r_1}(m)h_{1+r+r_1}^{d-r-r_1}(m)\Big) \\
&= \sum_{r_1=1}^{d-r-1} \Big(\sum_{m=1}^d {\lambda}_m \Big(\sum_{t=0}^{r+r_1} m^t v_{r,r_1}(t)
\Big) h_{1+r+r_1}^{d-r-r_1}(m)\Big).
\end{split}
\end{equation*}

To sum up, we have {\footnotesize
\begin{equation*}
\begin{split}
&\sum_{m=1}^d {\lambda}_m \Big(\sum_{t=0}^r m^t \left( v_r(t)-\lceil
v_r(t)\rceil\right)\Big)u_{1+r}^{d-r}(m)\approx {\lambda} \frac{n^{d-r}}{(d-r)!}\a_{1+r}\ldots \a_{d} (v_r(r)-\lceil v_r(r)\rceil)\\
&\hskip4.5cm-\sum_{r_1=1}^{d-r-1}\Bigg(\sum_{m=1}^d{\lambda}_m\Big(\sum_{t=0}^{r_1+r}m^tv_{r,r_1}(t)\Big)h_{1+r+r_1}^{d-(r+r_1)}(m)\Bigg)
+\vartriangle(K\ep_1).
\end{split}
\end{equation*}}
Together with ( \ref{a7}), we conclude {\footnotesize
\begin{equation*}
\begin{split}
\sum_{m=1}^d {\lambda}_m \Big(\sum_{t=0}^rm^tv_r(t)\Big)h_{1+r}^{d-r}(m) &\approx {\lambda}
\frac{n^{d-r}}{(d-r)!}\a_{1+r}\ldots \a_{d} \Big(v_r(r)-\lceil v_r(r)\rceil \Big)\\
&-\sum_{r_1=1}^{d-r-1}\Bigg(\sum_{m=1}^d{\lambda}_m\Big(\sum_{t=0}^{r_1+r}m^tv_{r,r_1}(t)\Big)
h_{1+r+r_1}^{d-(r+r_1)}(m)\Bigg)+\vartriangle(2K\ep_1).
\end{split}
\end{equation*}}
The proof of the claim  is completed. \end{proof}

We will use the claim repeatedly. First using (\ref{addye}) we have
\begin{equation*}
    \sum_{m=1}^d{\lambda}_mz_1^d(m)\approx \sum_{m=1}^d {\lambda}_m
    \bigg (w_1^d(m)-\sum_{j_1=1}^{d-1} w_1^{j_1}(m) h_{1+j_1}^{d-j_1}(m)\bigg).
\end{equation*}
By (\ref{a6}), we have
\begin{equation*}
\begin{split}
\sum_{m=1}^d {\lambda}_mw_1^d(m)& = \sum_{m=1}^d {\lambda}_m m^d
\frac{n^d}{d!}\a_1\ldots \a_d+\sum_{m=1}^d {\lambda}_my_1^d
\approx {\lambda} \frac{n^d}{d!}\a_1\ldots \a_d
+\vartriangle(K\ep_1).
\end{split}
\end{equation*}
Using this, (\ref{a6}) and the claim, we have {\footnotesize
\begin{equation*}
\begin{split} \sum_{m=1}^d {\lambda}_mz_1^d(m)&
\approx {\lambda} \frac{n^d}{d!}\a_1\ldots
\a_d-\sum_{m=1}^d{\lambda}_m\sum_{j_1=1}^{d-1}\Big(m^{j_1}\frac{n^{j_1}}{j_1!}\a_1\ldots
\a_{j_1}+\sum_{t=0}^{j_1-1}m^ta_1^{j_1}(t)\Big)h_{1+j_1}^{d-j_1}(m)+\vartriangle(K\ep_1)\\
&\approx {\lambda} \frac{n^d}{d!}\a_1\ldots
\a_d-\left(\sum_{j_1=1}^{d-1}{\lambda}
\frac{n^{d-j_1}}{(d-j_1)!}\a_{1+j_1}\ldots\a_d\Big(\frac{n^{j_1}}{j_1!}\a_1\ldots\a_{j_1}-
\lceil \frac{n^{j_1}}{j_1!}\a_1\ldots\a_{j_1} \rceil\Big)\right)\\
& +
\sum_{j_1=1}^{d-1}\sum_{j_2=1}^{d-j_1-1}\Bigg(\sum_{m=1}^d{\lambda}_m\bigg(m^{j_1+j_2}
\Big(\frac{n^{j_1}}{j_1!}\a_1\ldots\a_{j_1}- \lceil
\frac{n^{j_1}}{j_1!}\a_1\ldots\a_{j_1}
\rceil\Big)\frac{n^{j_2}}{j_2!}\a_{1+j_1}\ldots
\a_{j_1+j_2}\\
& +
\sum_{t=0}^{j_1+j_2-1}m^tv_{j_1,j_2}(t)\bigg)h_{1+j_1+j_2}^{d-(j_1+j_2)}(m)\Bigg)+\vartriangle(\big(2(d-1)K+K\big)\ep_1).
\end{split}
\end{equation*}}
Note that here we use $v_{j_1}(t)=a_1^{j_1}(t), t=0,1,\ldots, j_1-1$
and $v_{j_1}(j_1)=\frac{n^{j_1}}{j_1!}\a_1\ldots\a_{j_1}$.

Recall the definition of $U(\cdot)$:
\begin{equation*}
    \tfrac{n^d}{d!}\a_1\ldots\a_d=U(n;d),
\end{equation*}
\begin{equation*}
    \Big(\tfrac{n^{j_1}}{j_1!}\a_1\ldots\a_{j_1}-
\lceil \tfrac{n^{j_1}}{j_1!}\a_1\ldots\a_{j_1}
\rceil\Big)\tfrac{n^{j_2}}{j_2!}\a_{1+j_1}\ldots
\a_{j_1+j_2}=U(n;j_1,j_2).
\end{equation*}
Substituting these in the above equation, we have {\footnotesize
\begin{equation*}
\begin{split}
\sum_{m=1}^d{\lambda}_mz_1^d(m)& \approx {\lambda}
U(n;d)-\sum_{j_1=1}^{d-1}{\lambda}
U(n; j_1,d-j_1) +\vartriangle(2d K\ep_1)\\
&\hskip0.5cm
+\sum_{j_1=1}^{d-1}\sum_{j_2=1}^{d-j_1-1}\Bigg(\sum_{m=1}^d{\lambda}_m\bigg(m^{j_1+j_2}U(n;j_1,j_2)
+\sum_{t=0}^{j_1+j_2-1}m^tv_{j_1,j_2}(t)\bigg)h_{1+j_1+j_2}^{d-(j_1+j_2)}(m)\Bigg)
\end{split}
\end{equation*}}
Using the claim again, we have: {\footnotesize
\begin{equation*}
\begin{split}
\sum_{m=1}^d{\lambda}_mz_1^d(m)& \approx {\lambda}
U(n;d)-\sum_{j_1=1}^{d-1}{\lambda} U(n; j_1,d-j_1)+
\sum_{j_1=1}^{d-1}\sum_{j_2=1}^{d-j_1-1}{\lambda}
U(n;j_1,j_2,d-j_1-j_2)
\\
&-\sum_{j_1=1}^{d-1}\sum_{j_2=1}^{d-j_1-1}\sum_{j_3=1}^{d-(j_1+j_2)-1}\Bigg(\sum_{m=1}^d{\lambda}_m
\bigg(m^{j_1+j_2+j_3}U(n;j_1,j_2,j_3) +\\
& \sum_{t=0}^{j_1+j_2+j_3-1}
m^tv_{j_1,j_2,j_3}(t)\bigg)h_{1+j_1+j_2+j_3}^{d-(j_1+j_2+j_3)}(m)\Bigg)+\vartriangle(2dK\ep_1+2d^2K\ep_1).
\end{split}
\end{equation*}}

Inductively, we have  {\footnotesize
\begin{equation*}
\begin{split}
\sum_{m=1}^d{\lambda}_mz_1^d(m) & \approx
\Big(\sum_{l=1}^d(-1)^{l-1}\sum_{\substack{j_1,\ldots,j_l\in \N\\
j_1+\ldots +j_l=d}}{\lambda} U(n;j_1,\ldots,j_l)\Big)
+\vartriangle(2dK\ep_1+2d^2K\ep_1+\ldots+2d^{d-1}K\ep_1)\\
&\approx
\Big(\sum_{l=1}^d(-1)^{l-1}\sum_{\substack{j_1,\ldots,j_l\in \N\\
j_1+\ldots +j_l=d}} {\lambda} U(n;j_1,\ldots,j_l)\Big)
+\vartriangle\Big((d+d^2+\ldots+d^{d-1})(2K\ep_1)\Big).
\end{split}
\end{equation*}}
The proof of (\ref{longproof-1}) is now finished. \hfill $\square$

\chapter{Recurrence sets and regionally proximal relation of order $d$}\label{chapter-rp}

From this chapter we begin the study of higher order almost automorphy. In this chapter we investigate the
relationship between recurrence sets and $\RP^{[d]}$. Then using the results developed in this chapter, one
can characterize higher order almost automorphy in the next chapter.

\section{Regionally proximal relation of order $d$}

\subsection{Cubes and faces}
In the following subsections, we will
introduce notions about cubes, faces and face transformations. For
more details see \cite{HK05, HKM}.

\subsubsection{} Let $X$ be a set, let $d\ge 1$ be an integer, and write
$[d] = \{1, 2,\ldots , d\}$. We view $\{0, 1\}^d$ in one of two
ways, either as a sequence $\ep=\ep_1\ldots \ep_d$ of $0'$s and
$1'$s, or as a subset of $[d]$. A subset $\ep$ corresponds to the
sequence $(\ep_1,\ldots, \ep_d)\in \{0,1\}^d$ such that $i\in \ep$
if and only if $\ep_i = 1$ for $i\in [d]$. For example, ${\bf
0}=(0,0,\ldots,0)\in \{0,1\}^d$ is the same as $\emptyset \subset
[d]$.

If ${\bf n} = (n_1,\ldots, n_d)\in
\Z^d$ and $\ep\in \{0,1\}^d$, we define ${\bf n}\cdot \ep =
\sum_{i=1}^d n_i\ep_i .$ If we consider $\ep$ as $\ep\subset [d]$,
then ${\bf n}\cdot \ep  = \sum_{i\in \ep} n_i .$

\subsubsection{}

We denote $X^{2^d}$ by $X^{[d]}$. \index{$X^{[d]}$} A point ${\bf x}\in X^{[d]}$ can
be written in one of two equivalent ways, depending on the context:
${\bf x} = (x_\ep :\ep\in \{0,1\}^d )= (x_\ep : \ep\subset [d]). $
Hence $x_\emptyset =x_{\bf 0}$ is the first coordinate of ${\bf x}$.
For example, points in $X^{[2]}$ are like
$$(x_{00},x_{10},x_{01},x_{11})=(x_{\emptyset}, x_{\{1\}},x_{\{2\}},x_{\{1,2\}}).$$

For $x \in X$, we write $x^{[d]} = (x, x,\ldots , x)\in  X^{[d]}$.
The diagonal of $X^{[d]}$ is $\D^{[d]} = \{x^{[d]}: x\in X\}$.
Usually, when $d=1$, denote the diagonal by $\D_X$. A point ${\bf x}
\in X^{[d]}$ can be decomposed as ${\bf x} = ({\bf x'},{\bf  x''})$
with ${\bf x}', {\bf x}''\in X^{[d-1]}$, where ${\bf x}' = (x_{\ep0}
: \ep\in \{0,1\}^{d-1})$ and ${\bf x}''= (x_{\ep1} : \ep\in
\{0,1\}^{d-1})$. We can also isolate the first coordinate, writing
$X^{[d]}_* = X^{2^d-1}$ and then writing a point ${\bf x}\in
X^{[d]}$ as ${\bf x} = (x_\emptyset, {\bf x}_*)$, where ${\bf x}_*=
(x_\ep : \ep\neq \emptyset) \in X^{[d]}_*$.

\subsection{Face transformations}

\begin{de}
Let $\phi: X\rightarrow Y$ and $d\in \N$. Define $\phi^{[d]}:
X^{[d]}\rightarrow Y^{[d]}$ by $(\phi^{[d]}{\bf x})_\ep=\phi x_\ep$
for every ${\bf x}\in X^{[d]}$ and every $\ep\subset [d]$. Let $(X,
T)$ be a system and $d\ge 1$ be an integer. The {\em diagonal
transformation} \index{diagonal transformation} of $X^{[d]}$ is the map $T^{[d]}$. {\em Face
transformations} \index{face transformation} are defined inductively as follows: Let
$T^{[0]}=T$, $T^{[1]}_1=\id \times T$. If
$\{T^{[d-1]}_j\}_{j=1}^{d-1}$ is defined already, then set
$$T^{[d]}_j=T^{[d-1]}_j\times T^{[d-1]}_j, \ j\in \{1,2,\ldots, d-1\},$$
$$T^{[d]}_d=\id ^{[d-1]}\times T^{[d-1]}.$$
\end{de}


The {\em face group} \index{face group} of dimension $d$ is the
group $\F^{[d]}(X)$ \index{$\F^{[d]}(X)$} of
transformations of $X^{[d]}$ generated by the face transformations. We
often write $\F^{[d]}$ instead of $\F^{[d]}(X)$. For $\F^{[d]}$, we
use similar notations to that used for $X^{[d]}$: namely, an element
of the group is written as $S = (S_\ep : \ep\in\{0,1\}^d)$. For
convenience, we denote the orbit closure of ${\bf x}\in X^{[d]}$
under $\F^{[d]}$ by $\overline{\F^{[d]}}({\bf x})$, instead of
$\overline{\O({\bf x}, \F^{[d]})}$.

\subsection{Regionally proximal pairs of order $d$}

First let us define regionally proximal pairs of order $d$.

\begin{de}\index{regionally proximal relation of order $d$ }\index{$\RP^{[d]}$}
Let $(X, T)$ be a t.d.s. and let $d\ge 1$ be an integer. A
pair $(x, y) \in X\times X$ is said to be {\em regionally proximal
of order $d$} if for any $\d  > 0$, there exist $x', y'\in X$ and a
vector ${\bf n} = (n_1,\ldots , n_d)\in\Z^d$ such that $\rho (x, x')
< \d, \rho (y, y') <\d$, and
$$\rho(T^{{\bf n}\cdot \ep}x', T^{{\bf n}\cdot \ep}y') < \d\
\text{for any nonempty $\ep\subset [d]$}.$$ The set of regionally
proximal pairs of order $d$ is denoted by $\RP^{[d]}$ (or by
$\RP^{[d]}(X)$ in case of ambiguity), which is called {\em the
regionally proximal relation of order $d$}.
\end{de}

Moreover, let $\RP^{[\infty]}=\bigcap_{d=1}^\infty\RP^{[d]}(X)$. The
following theorem was proved by Host-Kra-Maass for minimal distal
systems \cite{HKM} and by Shao-Ye  for general minimal systems, see
Theorems 3.1, 3.5, 3.11, Proposition 6.1 and Theorem 6.4 in
\cite{SY}. \index{$\RP^{[\infty]}$}

\begin{thm}\label{ShaoYe}
Let $(X, T)$ be a minimal t.d.s. and $d\in \N$. Then
\begin{enumerate}
\item $(x,y)\in \RP^{[d]}$ if and only if $(x,y,y,\ldots,y)=(x,y^{[d+1]}_*) \in
\overline{\F^{[d+1]}}(x^{[d+1]})$ if and only if $(x,x^{[d]}_*, y,
x^{[d]}_*) \in \overline{\F^{[d+1]}}(x^{[d+1]})$.

\item $(\overline{\F^{[d]}}(x^{[d]}),\F^{[d]})$ is minimal for all
$x\in X$.

\item $\RP^{[d]}(X)$ is an equivalence relation, and so is
$\RP^{[\infty]}.$

\item If $\pi:(X,T)\lra (Y,S)$ is a factor map, then $(\pi\times
\pi)(\RP^{[d]}(X))=\RP^{[d]}(Y).$

\item $(X/\RP^{[d]},T)$ is the maximal $d$-step nilfactor of $(X,T)$.
\end{enumerate}\end{thm}
Note that (5) means that $(X/\RP^{[d]},T)$ is a system of order $d$
and any system of order $d$ factor of $(X,T)$ is a factor of
$(X/\RP^{[d]},T)$.

\begin{rem}
In \cite{SY}, Theorem \ref{ShaoYe} was proved for compact metric
spaces. In fact, one can show that Theorem \ref{ShaoYe} holds for
compact Hausdorff spaces by repeating the proofs sentence by
sentence in \cite{SY}. However, we will describe a direct approach
in Appendix \ref{appendix:CT2}. This result will be used in the next
section.
\end{rem}

\section{Nil$_d$ Bohr$_0$-sets, Poincar\'e sets and $\RP^{[d]}$}

In this section using Theorem A we
characterize $\RP^{[d]}$ using the families $\F_{Poi_d}, \F_{Bir_d}$
and $\F_{d,0}^*$.

\subsection{Nil Bohr-sets}

Recall $\F_{d,0}$ is the family consisting of all Nil$_d$
Bohr$_0$-sets.


\medskip

For $F_1,F_2\in \F_{d,0}$, there are $d$-step nilsystems $(X,T)$,
$(Y,S)$, $(x,y)\in X\times Y$ and $U\times V$ neighborhood of
$(x,y)$ such that $N(x,U)\subset F_1$ and $N(y,V)\subset F_2$. It is
clear that $N(x,U)\cap N(y, V)=N((x,y), U\times V)\in \F_{d,0}$.
This implies that $F_1\cap F_2\in\F_{d,0}$. So we conclude that

\begin{prop}
Let $d \in \N$. Then $\F_{d,0}$ is a filter, and $\F_{d,0}^*$ has
the Ramsey property.
\end{prop}

\subsection{Sets of $d$-recurrence}

\subsubsection{}

Recall that for $d\in \N$, $\F_{Poi_d}$ (resp. $\F_{Bir_d}$) is the
family consisting of all sets of $d$-recurrence
(resp. sets of $d$-topological recurrence).

\begin{rem} It is known that for all integer $d\ge 2$ there exists a set of
$(d-1)$-recurrence that is not a set of $d$-recurrence \cite{FLW}.
This also follows from Theorem~ \ref{several}.
\end{rem}

Recall that a set $S\subset \Z$ is {\em $d$-intersective} \index{$d$-intersective} if every
subset $A$ of $\Z$ with positive density contains at least one
arithmetic progression of length $d+1$ and a common difference in
$S$, i.e. there is some $n\in S$ such that
$$A\cap (A-n)\cap (A-2n)\ldots\cap (A-dn)\neq \emptyset.$$
Similarly, one can define topological $d$-intersective set by
replacing the set with positive density by a syndetic set in the
above definition.

We now give some equivalence conditions of $d$-topological
recurrence.
\begin{prop}\label{birkhoff-equi1} Let $S\subset \Z$. Then the following statements are equivalent:
\begin{enumerate}
\item $S$ is a set of topological $d$-intersective.

\item $S$ is a set of {$d$-topological recurrence}.

\item For any t.d.s. $(X,T)$ there are $x\in X$ and
$\{n_i\}_{i=1}^\infty\subset S$ such that
$$\lim_{i\lra +\infty}T^{jn_i}x=x\ \text{for each}\ 1\le j\le d.$$
\end{enumerate}
\end{prop}
\begin{proof} The equivalence between (1) and (2) was proved in \cite{FLW,
F81}.

$(2)\Rightarrow(3)$. Now assume that whenever $(Y, S)$ is a minimal
t.d.s. and $V\subset Y$ a nonempty open set, there is $n\in S$
such that
$$V\cap T^{-n}V\cap \ldots\cap T^{-dn}V\not=\emptyset.$$
Let $(X, T)$ be a t.d.s., and without loss of generality we assume
that $(X,T)$ is minimal, since each t.d.s. contains a minimal
subsystem. Define for each $j\in\N$
$$W_j=\{x\in X: \exists\ n\in S\
\text{with}\ d(T^{kn}x,x)<\tfrac{1}{j}\ \text{for each}\ 1\le k\le
d\}.$$ Then it is easy to verify that $W_j$ is non-empty, open and
dense. Then any $x\in \bigcap_{j=1}^\infty W_j$ is the point we look
for.

$(3)\Rightarrow(2).$  Let $(X, T)$ be a minimal t.d.s. and
$U\subset X$ a nonempty open set. Then there are $x\in X$ and
$\{n_i\}_{i=1}^\infty\subset S$ such that for each given $1\le k\le
d$, $T^{kn_i}x\lra x$. Since $(X,T)$ is minimal, there is some $l\in
\Z$ such that $x\in V=T^{-l}U$. When $i_0$ is larger enough, we have
$V \cap T^{-n_{i_0}}V\cap \ldots \cap T^{-dn_{i_0}}V\neq \emptyset$,
which implies that $U\cap T^{-n}U\cap \ldots \cap T^{-dn}U\neq
\emptyset$ by putting $n=n_{i_0}$.

\end{proof}

\subsubsection{}

The following fact follows from the Poincar\'e and Birkhoff multiple
recurrent theorems.
\begin{prop}\label{ramseyp}
For all $d\in \N$, $\F_{Poi_d}$ and $\F_{Bir_d}$ have the Ramsey
property.
\end{prop}
\begin{proof}
Let $F\in \F_{Poi_d}$ and $F=F_1\cup F_2$. Assume the contrary that
$F_i\not \in \F_{Poi_d}$ for $i=1,2$. Then there are measure
preserving systems $(X_i,\mathcal{B}_i,\mu_i,T_i)$ and $A_i\in
\mathcal{B}_i$ with $\mu_i(A_i)>0$ such that $\mu_i(A_i\cap
T_i^{-n}A_i\cap\ldots\cap T_i^{-dn}A_i)=0$ for $n\in F_i$, where
$i=1,2$. Set $X=X_1\times X_2$, $\mu=\mu_1\times \mu_2$,
$A=A_1\times A_2$ and $T=T_1\times T_2$. Then we have $$ \mu(A\cap
T^{-n}A\cap \ldots\cap T^{-dn}A)=\mu_1(\bigcap _{i=0}^d
T_1^{-in}A_1) \mu_2(\bigcap _{i=0}^d T_2^{-in}A_2)=0$$ for each
$n\in F=F_1\cup F_2$, a contradiction.

Now let $F\in \F_{Bir_d}$ and $F=F_1\cup F_2$. Assume the contrary
that $F_i\not \in \F_{Bir_d}$ for $i=1,2$. Then there are minimal
systems $(X_i,T_i)$ and non-empty open subsets $U_i$ such that
$U_i\cap T_i^{-n}U_i\cap\ldots\cap T_i^{-dn}U_i=\emptyset$ for $n\in
F_i$, where $i=1,2$. Let $X$ be a minimal subset of $X_1\times X_2$,
$U=(U_1\times U_2)\cap X$ and $T=T_1\times T_2$. Replacing $U_1$ and
$U_2$ by $T_1^{-i_1}U_1$ and $T_2^{-i_2}U_2$ respectively (if
necessary) we may assume that $U\not=\emptyset$ (using the
minimality of $T_1$ and $T_2$). Then we have
$$ U\cap T^{-n}U\cap \ldots\cap T^{-dn}U\subset \bigcap_{i=0}^d
T_1^{-in}U_1\times \bigcap_{i=0}^d T_2^{-in}U_2=\emptyset$$ for each
$n\in F=F_1\cup F_2$, a contradiction.
\end{proof}

\subsection{Nil$_d$ Bohr$_0$-sets and $\RP^{[d]}$}

To show the following result we need several well known facts
(related to distality) from the Ellis enveloping semigroup theory,
see \cite{Au88, V77}. Also we note that the lifting property in
Theorem \ref{ShaoYe} is valid when $X$ is compact and Hausdorff (see
Appendix \ref{appendix:CT2} for more details).

\begin{thm}\label{huang10}
Let $(X,T)$ be a minimal t.d.s. Then $(x,y)\in \RP^{[d]}$ if and
only if $N(x,U)\in \F_{d,0}^*$ for each neighborhood $U$ of $y$.
\end{thm}

\begin{proof}
First assume that $N(x,U)\in \F_{d,0}^*$ for each neighborhood $U$
of $y$. Let $(X_d,S)$ be the maximal $d$-step nilfactor of $(X,T)$
(see Theorem \ref{ShaoYe}) and $\pi:X\lra X_d$ be the projection.
Then for any neighborhood $V$ of $\pi(x)$, we have $N(x,U)\cap
N(\pi(x),V)\not=\emptyset$ since $N(x,U)\in \F_{d,0}^*$. This means
that there is a sequence $\{n_i\}$ such that $$(T\times S)^{n_i}(x,
\pi(x))\lra (y,\pi(x)),\ i\to \infty .$$ Thus, we have $$
\pi(y)=\pi(\lim_{i}T^{n_i}x)=\lim_{i}S^{n_i}\pi(x)=\pi(x),$$ i.e.
$(x,y)\in \RP^{[d]}$.

\medskip
Now assume that $(x,y)\in \RP^{[d]}$ and $U$ is a neighborhood of
$y$. We need to show that if $(Z,R)$ is a $d$-step nilsystem,
$z_0\in Z$ and $V$ is a neighborhood of $z_0$ then $N(x,U)\cap
N(z_0,V)\not=\emptyset$.

Let $$W=\prod_{z\in Z}Z \quad \text{ (i.e. $W=Z^Z$) and
}R^Z:W\rightarrow W$$ with $(R^Z\omega)(z)=R(\omega(z))$ for any
$z\in Z$, where $\omega=(\omega(z))_{z\in Z}\in W$. Note that in
general $(W,R^Z)$ is not a metrizable but a compact Hausdorff
system. Since $(Z,R)$ is a $d$-step nilsystem, $(Z,R)$ is distal.
Hence $(W,R^Z)$ is also distal.

Choose $\omega^*\in W$ with $\omega^*(z)=z$ for all $z\in Z$, and
let $Z_\infty=\overline{\O(\omega^*,R^Z)}$. Then
$(Z_\infty,R^Z)$ is a minimal subsystem of $(W,R^Z)$ since $(W,R^Z)$
is distal. For any $\omega\in Z_\infty$, there exists $p\in E(Z,R)$
such that $\omega(z)=p(\omega^*(z))=p(z)$ for all $z\in Z$. Since
$(Z,R)$ is a distal system, the Ellis semigroup $E(Z,R)$ is
a group (Appendix \ref{appendix:CT2}). Particularly, $p:Z\rightarrow
Z$ is a surjective map. Thus
$$\{\omega(z):z\in Z\}=\{ p(z):z\in Z\}=Z.$$ Hence there
exists $z_\omega\in Z$ such that $\omega(z_\omega)=z_0$.

\medskip

Take a minimal subsystem $(A,T\times R^Z)$ of the product system
$(X\times Z_\infty,T\times R^Z)$. Let $\pi_X: A\rightarrow X$ be the
natural coordinate projection. Then $\pi_X:(A,T\times
R^Z)\rightarrow (X,T)$ is a factor map between two minimal systems.
Since $(x,y)\in \RP^{[d]}(X,T)$, by Theorem \ref{ShaoYe} there exist
$\omega^1,\omega^2\in W$ such that $((x,\omega^1),(y,\omega^2))\in
\RP^{[d]}(A,T\times R^Z)$.

For $\omega^1$, there exists $z_1\in Z$ such that
$\omega^1(z_1)=z_0$ by the above discussion. Let $\pi: A\rightarrow
X\times Z$ with $\pi(u,\omega)=(u,\omega(z_1))$ for $(u,\omega)\in
A$, $u\in X$, $\omega\in W$. Let $B=\pi(A)$. Then $(B,T\times R)$ is
a minimal subsystem of $(X\times Z,T\times R)$, and $\pi:(A,T\times
R^Z)\rightarrow (B,T\times R)$ is a factor map between two minimal
systems. Clearly $\pi(x,\omega^1)=(x,z_0)$,
$\pi(y,\omega^2)=(y,z_2)$ for some $z_2\in Z$, and
$$((x,z_0),(y,z_2))=\pi\times \pi((x,\omega^1),(y,\omega^2))\in
\RP^{[d]}(B,T\times R).$$ Moreover, we consider the projection
$\pi_Z$ of $B$ onto $Z$. Then $\pi_Z:(B,T\times R)\rightarrow (Z,R)$
is a factor map and so $(z_0,z_2)=\pi_Z\times
\pi_Z((x,z_0),(y,z_2))\in \RP^{[d]}(Z,R)$. Since $(Z,R)$ is a system
of order $d$, $z_0=z_2$. Thus $((x,z_0),(y,z_0))\in
B$. Particularly, $N(x,U)\cap
N(z_0,V)=N((x,z_0),U\times V)$ is a syndetic set since $(B,T\times
R)$ is minimal. This completes the proof of theorem.
\end{proof}

\begin{rem}\label{huangrem} From the proof of Theorem \ref{huang10},
we have the following result: if $(X,T)$ is minimal and $(x,y)\in
\RP^{[d]}$ then  $N(x,U)\cap F$ is a syndetic set for each $F\in
\F_{d,0}$ and each neighborhood $U$ of $y$.
\end{rem}


\subsection{Recurrence sets and $\RP^{[d]}$}

Now we can sum up the main result of this section as follows, whose proof depends on Theorem A.

\begin{thm}\label{several}
Let $(X,T)$ be a minimal t.d.s., $d\in\N$ and $x,y\in X$. Then the
following statements are equivalent:

\begin{enumerate}

\item $(x,y)\in \RP^{[d]}$.

\item
$N(x,U)\in \F_{Poi_d}$ for each neighborhood $U$ of $y$.

\item $N(x,U)\in \F_{Bir_d}$ for each
neighborhood $U$ of $y$.

\item $N(x,U)\in \F_{d,0}^*$ for each
neighborhood $U$ of $y$.
\end{enumerate}
\end{thm}

\begin{proof}
First we show that $(1)\Rightarrow (2)$.
Let $U$ be a neighborhood of $y$. We need to show $N(x,U)\in
\F_{Poi_d}$.

Now let $(Y,\mathcal{Y},\mu, S)$ be a measure preserving system and
$A\in \mathcal{Y}$ with $\mu(A)>0$. Let $\mu=\int_{\Omega}
\mu_\omega d m(\omega)$ be an ergodic decomposition of $\mu$. Then
there is $\Omega'\subset \Omega$ with $m(\Omega')>0$ such that for
each $\omega\in \Omega'$, $\mu_\omega(A)>0$. For $\omega\in \Omega'$, set
$$F_\omega=\{n\in\Z:\mu_\omega(A\cap S^{-n}A\cap \ldots\cap S^{-dn}A)>0\}.$$
By Theorem \ref{d-recurrence} there is some subset $M$ with
$BD^*(M)=0$ such that $B=F_\omega \D M$ is a Nil$_d$ Bohr$_0$-set. Hence we
have $N(x,U)\cap (F_\omega\Delta M)$ is syndetic by Remark \ref{huangrem}.
Thus we conclude that there is $n_\omega\not=0$ with $n_\omega\in
N(x,U)\cap F_\omega$ since $BD^*(M)=0$. This implies that there are
$\Omega''\subset \Omega'$ with $m(\Omega'')>0$ and $n\in N(x,U)$
such that for each $\omega\in \Omega''$ one has
$\mu_\omega(A\cap S^{-n}A\cap \ldots\cap S^{-dn}A)>0$ which in turn
implies $\mu(A\cap S^{-n}A\cap \ldots\cap S^{-dn}A)>0$. By the
definition, $N(x,U)\in \F_{Poi_d}$.

It follows from Corollary D that $(2)\Rightarrow (3)
\Rightarrow (4)$. By Theorem \ref{huang10}, one has that
$(4)\Rightarrow (1)$ and completes the proof.
\end{proof}

\section{$SG_d$-sets and $\RP^{[d]}$}
In this section we will describe $\RP^{[d]}$ using the $SG_d$-sets
introduced by Host and Kra in \cite{HK10}. First we recall some
definitions.

\subsection{Sets $SG_d(P)$}\index{$SG_d(P)$}

Recall that for $d\in\N$ and  a (finite or infinite) sequence
$P=\{p_i\}_i$ in $\Z$ the {\em set of sums with gaps of length less
than $d$} of $P$ is the set $SG_d(P)$ of all integers of the form
$$\ep_1p_1+\ep_2p_2+\ldots +\ep_np_n$$ where $n\ge 1$ is an integer,
$\ep_i\in \{0,1\}$ for $1\le i\le n$, the $\ep_i$ are not all equal
to $0$, and the blocks of consecutive $0$'s between two $1$ have
length less than $d$. \index{set of sums with gaps of length less
than $d$}

Note that in this definition, $P$ is a sequence and not a subset of
$\Z$. For example, if $P=\{p_i\}$, then $SG_1(P)$ is the set of all
sums $p_m+p_{m+1}+\ldots +p_n$ of consecutive elements of $P$, and
thus it coincides with the set $\D(S)$ where $S=\{0, p_1, p_1+p_2,
p_1+p_2+p_3,\ldots\}$. Therefore $SG^*_1$-sets are the same as
$\D^*$-sets.

For a sequence $P$, $SG_2(P)$ consists of all sums of the form
$$\sum_{i=m_0}^{m_1}p_i +\sum_{i=m_1+2}^{m_2}p_i+\ldots+
\sum_{i=m_{k-1}+2}^{m_k}p_i+\sum_{i=m_k+2}^{m_{k+1}}p_i$$ where
$k\in \N$ and $m_0,m_1,\ldots, m_{k+1}$ are positive integers
satisfying $m_{i+1}\ge m_i+2$ for $i=1,\ldots,k$, and $m_1\ge m_0$.

\medskip

Recall that for each $d\in \N$, $\F_{SG_d}$ is the family generated by
$SG_d$. \index{$\F_{SG_d}$} Moreover, let $\F_{fSG_d}$ be the
family of sets containing
arbitrarily long $SG_d(P)$ sets with $P$ finite. That is, $A\in
\F_{fSG_d}$ if and only if there are finite sequences $P^i$ with
$|P^i|\lra \infty$ such that $\bigcup_{i=1}^\infty SG_d(P^i)\subset
A$. It is clear that
$$\F_{SG_1}\supset \F_{SG_2}\supset \ldots\supset
\F_{SG_\infty}=:\bigcap_{i=1}^\infty
\F_{SG_i},$$ and
$$\F_{fSG_1}\supset \F_{fSG_2}\supset \ldots\supset
\F_{fSG_\infty}=:\bigcap_{i=1}^\infty \F_{fSG_i}.$$

We now show
\begin{prop} The following statements hold:
\begin{enumerate}
\item $\F_{SG_\infty}=\{A:\exists \ P^i\ \text{infinite for each}\ i\in
\N\ \text{such that}\ A\supset \bigcup_{i=1}^\infty SG_i(P^i)\}.$

\item $\F_{fSG_\infty}=\F_{fip}$.
\end{enumerate}
\end{prop}
\begin{proof} (1). Assume that $A\in \F_{SG_\infty}$. Then $A\in\bigcap_{i=1}^\infty
\F_{SG_i}$ and hence $A\in \F_{SG_i}$ for each $i\in \N$. Thus for
each $i\in\N$ there is $P^i$ infinite such that $A\supset SG_i(P^i)$
which implies that $A\supset\bigcup_{i=1}^\infty SG_i(P^i).$

Now let $B=\bigcup_{i=1}^\infty SG_i(P^i),$ where $P^i$ infinite for
each $i\in\N$. It is clear that $B\subset \F_{SG_i}$ for each $i$
and thus, $B\in\F_{SG_\infty}$. Since $\F_{SG_\infty}$ is a family,
we conclude that $\{A:\exists P^i\ \text{infinite for each}\ i\in
\N\ \text{such that}\ A\supset \bigcup_{i=1}^\infty
SG_i(P^i)\}\subset \F_{SG_\infty}$.

\medskip

(2) It is clear that $\F_{fSG_\infty}\subset \F_{fip}$. Let $A\in
\F_{fip}$ and without loss of generality assume that
$A=\bigcup_{i=1}^\infty FS (P^i)$ with $P^i=\{p_1^i, \ldots,
p_i^i\}$ and $|P^i|\lra\infty$.

Put $A_d=\bigcup_{i=1}^\infty SG_d(P^i)\subset A$ for $d\in\N$. Then
$A_d\in \F_{fSG_d}$ which implies that $A\in \F_{fSG_d}$ for each
$d\ge 1$ and hence $A\in\F_{fSG_\infty}$. That is, $\F_{fip}\subset
\F_{fSG_\infty}$.
\end{proof}

\subsection{$SG_d$-sets and $\RP^{[d]}$}

The following theorem is the main result of this section.

\begin{thm}\label{huang12}
Let $(X,T)$ be a minimal t.d.s. Then for any $d\in\N$,
$(x,y)\in\RP^{[d]}$ if and only if $N(x,U)\in \F_{SG_d}$ for each
neighborhood $U$ of $y$. The same holds when $d=\infty$.
\end{thm}
\begin{proof}
It is clear that if $N(x,U)\in \F_{SG_d}$ for each neighborhood $U$
of $y$, then it contains some $FS(\{n_i\}_{i=1}^{d+1})$ for each
neighborhood $U$ of $y$ which implies that $(x,y)\in\RP^{[d]}$ by
Theorem \ref{ShaoYe}.

\medskip
Now assume that $(x,y)\in\RP^{[d]}$ for $d\ge 1$. Let for $i\ge 2$
$$A_i=:\{0,1\}^i\setminus \{(0,\ldots, 0,0), (0,\ldots,0, 1)\}$$

The case when $d=1$ was proved by Veech \cite{V68} and our method is
also valid for this case. To make the idea of the proof clearer, we
first show the case when $d=2$ and the general case follows by the
same idea.

\medskip

\noindent {\bf I. The case $d=2$.}

Assume that $(x,y)\in\RP^{[2]}$. Then by Theorem \ref{ShaoYe} (1)
for each neighborhood $V\times U$ of $(x,y)$, there are
$n_1,n_2, n_3\in \mathbb{Z}$ such that
$$T^{\ep_1n_1+\ep_2n_2+\ep_3n_3}x\in V \ \text{and}\ T^{n_3}x\in U,$$
for each $(\ep_1,\ep_2,\ep_3)\in A_3.$ For a given $U$, let $\eta>0$
with $B(y,\eta)\subset U$, and take $\eta_i>0$ with
$\sum_{i=1}^\infty\eta_i<\eta$, where $B(y,\eta)=\{z\in X:
\rho(z,y)<\eta\}$.

Choose $n_1^1,n_2^1,n_3^1\in \mathbb{Z}$ such that
$$\rho(T^{n_3^1}x,y)<\eta_1\ \text{and}\ \rho(T^{r}x,x)<\eta_1,$$
for each $r\in E_1$ with
$$E_1=\{\ep_1n_1^1+\ep_2n_2^1+\ep_3n_3^1: (\ep_1,\ep_2,\ep_3)\in
A_3\}.$$ Let
$$S_1=FS(\{n_1^1,n_2^1,n_3^1\}).$$

Choose $n_1^2,n_2^2,n_3^2\in \mathbb{Z}$ such that
$$\rho(T^{n_3^2}x,y)<\eta_2 \text{ and } \max_{s\in
S_1}\rho(T^{s+r}x, T^sx)<\eta_2$$ for each $r\in E_2$ with
$$E_2=\{\ep_1n_1^2+\ep_2n_2^2+\ep_3n_3^2:
(\ep_1,\ep_2,\ep_3)\in A_3\}.$$ Let
$$S_2=FS(\{n_i^j:j=1,2,i=1,2,3\}).$$

Generally when $n_1^i$, $n_2^i$, $n_3^i$, $E_i,S_i$ are defined for $1\le
i\le k$ choose $n_1^{k+1}$, $n_2^{k+1}$, $n_3^{k+1}$ $\in \mathbb{Z}$ such
that
\begin{equation}\label{huang1}
\rho(T^{n_3^{k+1}}x,y)<\eta_{k+1} \text{ and } \max_{s\in
S_k}\rho(T^{s+r}x, T^sx)<\eta_{k+1}. \end{equation} for each $r\in
E_{k+1}$, where
$$E_{k+1}=\{\ep_1n_1^{k+1}+\ep_2n_{2}^{k+1}+\ep_3n_3^{k+1}:
(\ep_1,\ep_2,\ep_3)\in A_3\}.$$ Let
$$S_{k+1}=FS(\{n_i^j:i=1,2,3, 1\le j\le k+1\}).$$

Now we define a sequence $P=\{P_k\}$ such that
$$P_1=n_3^1+n_1^2+n_1^3, P_2=n_3^2+n_2^3+n_2^4, P_3=n_3^3+n_1^4+n_1^5,P_4=n_3^4+n_2^5+n_2^6, \ldots$$
That is, \begin{equation*}
    P_k=n^k_{3}+n^{k+1}_{k\ ({\rm mod}\ 2)}+n^{k+2}_{k\  ({\rm
    mod}\ 2)},
\end{equation*}
where we set $2m\ ({\rm mod}\ 2)=2$ for $m\in \mathbb{N}$.  We
claim that $N(x,U)\supset SG_2(P).$

Let $n\in SG_2(P)$. Then $n=\sum_{j=1}^k P_{i_j},$ where $1\le
i_{j+1}-i_j\le 2$ for $1\le j\le k-1$. By induction for $k$, it is
not hard to show that $n$ can be written as
$$n=a_1+a_2+\ldots+a_{i_k-i_1+3}$$ such that $a_1=n_3^{i_1}$, $a_j\in
E_{j+i_1-1}$ for $j=2,3,\ldots,i_k-i_1+1$ and $a_{i_k-i_1+2}\in \{
n_1^{i_k+1},n_2^{i_k+1},n_1^{i_k+1}+n_2^{i_k+1}\}$,
$a_{i_k-i_1+3}=n_{i_k \ ({\rm mod}\ 2)}^{i_k+2}$. In other words,
$n$ can be written as $n=a_1+a_2+\ldots+a_{i_k-i_1+3}$ with
$a_1=n_3^{i_1}$ and $a_j\in E_{i_1+j-1}$ for $2\le j\le i_k-i_1+3$.

Note that $\sum \limits_{\ell=1}^j a_\ell\in S_{i_1+j-1}$ and
$a_{j+1}\in E_{i_1+j}$ for $1\le j\le i_k-i_1+2$. Thus by
(\ref{huang1}) we have
$$\rho(T^{\sum_{i=1}^ja_i}x,T^{\sum_{i=1}^{j+1}a_i}x)<\eta_{j+i_1}$$
for $1\le j\le i_k-i_1+2$. This implies that
\begin{align*}
\rho(T^nx,y)&\le \rho(T^{\sum_{j=1}^{i_k-i_1+3}a_i}x,
T^{\sum_{j=1}^{i_k-i_1+2}
a_i}x)+\ldots+\rho(T^{n_3^{i_1}+a_2}x,T^{n_3^{i_1}}x)+
\rho(T^{n_3^{i_1}}x,y)\\
&<\sum_{j=0}^{i_k-i_1+2}\eta_{j+i_1}<\eta.
\end{align*}
That is, $n\in N(x,U)$ and hence $N(x,U)\supset SG_2(P).$

\medskip

\noindent {\bf II. The general case.}

Generally assume that $(x,y)\in\RP^{[d]}$ with $d\ge 2$. Then by
Theorem \ref{ShaoYe} (1) for each neighborhood $V\times U$
of $(x,y)$, there are $n_1,n_2,\ldots, n_{d+1}\in \mathbb{Z}$ such
that
$$T^{\ep_1n_1+\ep_2n_2+\ldots+\ep_{d+1}n_{d+1}}x\in V \ \text{and}\ T^{n_{d+1}}x\in U,$$
for each $(\ep_1,\ep_2,\ldots, \ep_{d+1})\in A_{d+1}.$ For a given $U$,
let $\eta>0$ with $B(y,\eta)\subset U$, and take $\eta_i>0$ with
$\sum_{i=1}^\infty\eta_i<\eta$.

Choose $n_1^1,n_2^1,\ldots,n_{d+1}^1\in \mathbb{Z}$ such that
$\rho(T^{n_{d+1}^1}x,y)<\eta_1\ \text{and}\ \rho(T^{r}x,x)<\eta_1$
where $r\in E_1$ with
$$E_1=\{\ep_1n_1^1+\ep_2n_2^1+\ldots+\ep_{d+1}n_{d+1}^1: (\ep_1,\ep_2,\ldots,\ep_{d+1})\in
A_{d+1}\}.$$ Let
$$S_1=FS(\{n^1_1,\ldots,n^1_{d+1}\}).$$

Choose $n_1^2,n_2^2,\ldots,n_{d+1}^2\in\Z$ such that
$$\rho(T^{n_{d+1}^2}x,y)<\eta_2\ \text{and}\ \max_{s\in
S_1}\rho(T^{s+r}x, T^sx)<\eta_2$$ for each $r\in E_2$ with
$$E_2=\{\ep_1n_1^2+\ep_2n_2^2+\ldots+\ep_{d+1}n_{d+1}^2:
(\ep_1,\ep_2,\ldots,\ep_{d+1})\in A_{d+1}\}.$$ Let
$$S_2=FS(\{n_1^1,\ldots, n_{d+1}^1,n_1^2,\ldots,n_{d+1}^2\}).$$

Generally when $n_1^i,\ldots,n_{d+1}^i$, $E_i,S_i$ are defined for
$1\le i\le k$ choose $n_1^{k+1}$, $\ldots$, $n_{d+1}^{k+1}$ $\in \mathbb{Z}$
such that
\begin{equation}\label{huang2}
\rho(T^{n_{d+1}^{k+1}}x,y)<\eta_{k+1}\ \text{and}\ \max_{s\in
S_k}\rho(T^{s+r}x, T^sx)<\eta_{k+1}. \end{equation} for each $r\in
E_{k+1}$, where
$$E_{k+1}=\{\ep_1n_1^{k+1}+\ep_2n_{2}^{k+1}+\ldots+\ep_{d+1}n_{d+1}^{k+1}:
(\ep_1,\ep_2,\ldots,\ep_{d+1})\in A_{d+1}\}.$$ Let
$$S_{k+1}=FS(\{n_i^j:i=1,\ldots, d+1, 1\le j\le k+1\}).$$

Now we define a sequence $P=\{P_k\}$ such that
\begin{eqnarray*}
P_1&=&n_{d+1}^1+n_1^2+\ldots+n_1^{d+1},
P_2=n_{d+1}^2+n_2^3+\ldots+n_2^{d+2},
\ldots,\\ P_d&=&n_{d+1}^d+n_d^{d+1}+\ldots+n_d^{2d},\\
P_{d+1}&=& n_{d+1}^{d+1}+n_1^{d+2}+\ldots+n_1^{2d+1},
P_{d+2}=n_{d+1}^{d+2}+n_2^{d+3}+\ldots+n_2^{2d+2},\ldots,\\
P_{2d}&=&n_{d+1}^{2d}+n_d^{2d+1}+\ldots+n_d^{3d}, \ldots
\end{eqnarray*}
That is,
\begin{equation*}
    P_k=n^k_{d+1}+n^{k+1}_{k\ ({\rm mod}\ d)}+\ldots+n^{k+d}_{k\  ({\rm
    mod}\
    d)},
\end{equation*}
where we set $dm\ ({\rm mod}\ d)=d$ for $m\in \mathbb{N}$.

We claim that $N(x,U)\supset SG_d(P).$ Let $n\in SG_d(P)$ then
$n=\sum_{j=1}^k P_{i_j},$ where $1\le i_{j+1}-i_j\le d$ for $1\le
j\le k-1$. By induction for $k$, it is not hard to show that $n$ can
be written as
$$n=a_1+a_2+\ldots+a_{i_k-i_1+d+1}$$ such that $a_1=n_{d+1}^{i_1}$, $a_j\in
E_{j+i_1-1}$ for $j=2,3,\ldots,i_k-i_1+1$ and
$$a_{i_k-i_1+1+r}\in FS(\{ n^{i_k+r}_\ell:  \ell \in \{ 1,2,\ldots,d\} \setminus \cup_{j=1}^{r-1} \{ i_k+j \ ({\rm mod}\ d)\}\})$$
for $1\le r \le d$. In other words, $n$ can be written as
$n=a_1+a_2+\ldots+a_{i_k-i_1+d+1}$ with $a_1=n_{d+1}^{i_1}$ and
$a_j\in E_{i_1+j-1}$ for $2\le j\le i_k-i_1+d+1$.

 Note that $\sum_{\ell=1}^j a_\ell\in
S_{i_1+j-1}$ and  $a_{j+1}\in E_{i_1+j}$ for $1\le j\le
i_k-i_1+d$. Thus by $(\ref{huang2})$ we have
$$\rho(T^{\sum_{i=1}^ja_i}x,T^{\sum_{i=1}^{j+1}a_i}x)<\eta_{i_1+j}$$
for $1\le j\le i_k-i_1+d$. This implies that
\begin{align*}
\rho(T^nx,y)&\le \rho(T^{\sum_{j=1}^{i_k-i_1+d+1}a_i}x,
T^{\sum_{j=1}^{i_k-i_1+d}
a_i}x)+\ldots+ \rho(T^{n_{d+1}^{i_1}}x,y)\\
&<\sum_{j=0}^{i_k-i_1+d}\eta_{j+i_1}<\eta.
\end{align*}
That is, $n\in N(x,U)$ and hence $N(x,U)\supset SG_d(P)$ which
implies that $N(x,U)\in \F_{SG_d}$. The proof is completed.
\end{proof}



\section{Cubic version of multiple recurrence sets and $\RP^{[d]}$}

Cubic version of multiple ergodic averages was studied in
\cite{HK05}, and also was proved very useful in some other questions
\cite{HK10, HKM}.

In this section we will discuss the question how to describe
$\RP^{[d]}$ using cubic version of multiple recurrence sets. Since
by Theorem \ref{ShaoYe} one can use dynamical parallelepipeds to
characterize $\RP^{[d]}$, it seems natural to describe $\RP^{[d]}$
using the cubic version of multiple recurrence sets.

\subsection{Cubic version of multiple Birkhoff recurrence sets}

First we give definitions for the cubic version of multiple
recurrence sets. We leave  the equivalent statements in viewpoint of
intersective sets to Appendix \ref{appendix:Intersective}.

\subsubsection{Birkhoff recurrence sets}\index{Birkhoff recurrence set}\index{set of topological
recurrence}
First we recall the classical definition. Let $P\subset \Z$. $P$ is
called a {\em Birkhoff recurrence set} (or a {\em set of topological
recurrence}) if whenever $(X, T)$ is a minimal t.d.s. and
$U\subset X$ a nonempty open set, then $P\cap N(U,U)\neq
\emptyset$. Let $\F_{Bir}$ denote the collection of Birkhoff
recurrence subsets of $\Z$. An alternative definition is that for
any t.d.s. $(X,T)$ there are $\{n_i\}\subset P$ and $x\in X$ such
that $T^{n_i}x\lra x$. Now we generalize the above definition to the
higher order.

\begin{de}\label{def-cubic}\index{Birkhoff recurrence set of order $d$}\index{set of topological recurrence of order $d$}
Let $d\in \N$. A subset $P$ of $\Z$ is called a {\em Birkhoff
recurrence set of order $d$} (or a {\em set of topological
recurrence of order $d$}) if whenever $(X, T)$ is a t.d.s. there are
$x\in X$ and $\{n_i^j\}_{j=1}^d\subset P$, $i\in \N$, such that
$FS(\{n_i^j\}_{j=1}^d)\subset P, i\in \N$ and for each given
$\ep=(\ep_1,\ldots,\ep_d)\in \{0,1\}^d$, $T^{m_i}x\lra x$, where
$m_i=\ep_1n_i^1+\ldots+\ep_dn_i^d$, $i\in\N$. A subset $F$ of $\Z$
is a {\em Birkhoff recurrence set of order $\infty$} if it is a
Birkhoff recurrence set of order $d$ for any $d\ge 1$.\index{Birkhoff recurrence set of order $\infty$}

For example, when $d=2$ this means that there are sequence $\{n_i\},
\{m_i\}\subset P$ and $x\in X$ such that $\{n_i+m_i\}\subset P$ and
$T^{n_i}x\lra x, T^{m_i}x\lra x$, $T^{n_i+m_i}x\lra x$.
\end{de}

Similarly we can define (topologically) intersective of order $d$
and intersective of order $d$ (see Appendix
\ref{appendix:Intersective}). We have
\begin{prop}\label{birkhoff-equi2} Let $d\in \N$ and $P\subset \Z$. The following statements are equivalent:
\begin{enumerate}
\item $P$ is a {Birkhoff recurrence set} of order $d$.

\item Whenever $(X, T)$ is a minimal t.d.s. and
$U\subset X$ a nonempty open set,  there are $n_1,\ldots,n_d$ with
$FS(\{n_i\}_{i=1}^d)\subset P$ such that
$$U \cap \big(\bigcap_{n\in FS(\{n_i\}_{i=1}^d)}T^{-n}U \big)\neq \emptyset.$$

\item $P$ is (topologically) intersective of order $d$.
\end{enumerate}
\end{prop}

\begin{proof} $(1)\Leftrightarrow (2)$ follows from the proof of Proposition
\ref{birkhoff-equi1}. See Appendix \ref{appendix:Intersective} for
the proof $(1)\Leftrightarrow (3).$
\end{proof}

\begin{rem}
From the above proof, one can see that for a minimal t.d.s. the set
of recurrent point in the Definition \ref{def-cubic} is residual.
\end{rem}

\subsubsection{Some properties of Birkhoff sequences of order $d$ }

The family of all Birkhoff recurrence
sets of order $d$ is denoted by $\F_{B_d}$. \index{$\F_{B_d}$} We have
$$\F_{B_1}\supset \F_{B_2}\supset \ldots \supset \F_{B_d}\supset \ldots\supset
\F_{B_\infty}=:\bigcap_{d=1}^\infty \F_{B_d}.$$ \index{$\F_{B_\infty}$}

We will show later (after Proposition \ref{longproof}) that
\begin{prop}\label{birkhoff-pro} $\F_{B_\infty}=\F_{fip}.$
\end{prop}

\subsection{Birkhoff recurrence sets and $\RP^{[d]}$}

We have the following theorem

\begin{thm}\label{birkhoff-thm}
Let $(X,T)$ be a minimal t.d.s. Then for any $d\in\N\cup\{\infty\}$,
$(x,y)\in\RP^{[d]}$ if and only if $N(x,U)\in \F_{B_d}$ for each
neighborhood $U$ of $y$.
\end{thm}

\begin{proof} We first show the case when $d\in\N$.

($\Leftarrow$)  Let $d\in \N$ and assume $N(x,U)\in \F_{B_d}.$ Then
there are $FS(\{n_i\}_{i=1}^d)\subset N(x,U)$ such that $U\cap
\bigcap_{n\in FS(\{n_i\}_{i=1}^d)}T^{-n}U\neq \emptyset.$ This means
that there is $y'\in U$ such that $T^ny'\in U$ for any $n\in
FS(\{n_i\}_{i=1}^d)$. Since $T^nx\in U$ for any $n\in
FS(\{n_i\}_{i=1}^d)$, we conclude that $(x,y)\in\RP^{[d]}$ by the
definition.

\medskip

($\Rightarrow$) Assume that $(x,y)\in\RP^{[d]}$ and $U$ is a
neighborhood of $y$. Let $(Z, R)$ be a minimal  t.d.s., $V$ be a
non-empty open subset of $Z$ and $\Lambda\subset X\times Z$ be a
minimal subsystem. Let $\pi:\Lambda\lra X$ be the projection. Since
$(x,y)\in\RP^{[d]}$ there are $z_1,z_2\in Z$ such that
$((x,z_1),(y,z_2))\in \RP^{[d]}(\Lambda, T\times R)$ by Theorem
\ref{ShaoYe}. Let $m\in\N$ such that $R^{-m}V$ be a neighborhood of
$z_2$. Then $U\times R^{-m}V$ is a neighborhood of $(y,z_2)$. By
Theorem \ref{ShaoYe}, there are $n_1,\ldots,n_{d+1}$ such that
$$N((x,z_1),U\times R^{-m}V)\supset FS(\{n_i\}_{i=1}^{d+1}).$$ This
implies that $\bigcap_{n\in FS(\{n_i\}_{i=1}^{d+1})}
R^{-n-m}V\not=\emptyset.$ Thus, $V\cap \bigcap_{n\in
FS(\{n_i\}_{i=1}^{d})} R^{-n}V\not=\emptyset,$ i.e. $N(x,U)\in
\F_{B_d}$.

The case $d=\infty$ is followed from the result for $d\in \N$ and
the definitions.
\end{proof}

\subsection{Cubic version of multiple Poincar\'e recurrence sets}

\subsubsection{Poincar\'e recurrence sets}

Now we give the cubic version of multiple Poincar\'e recurrence
sets.

\begin{de}\index{Poincar\'e recurrence set of order $d$}
For $d\in \N$, a subset $F$ of $\Z$ is a {\em Poincar\'e recurrence
set of order $d$} if for each $(X,\mathcal{B},\mu,T)$ and $A\in
\mathcal{B}$ with $\mu(A)>0$ there are $n_1,\ldots,n_d\in\Z$ such
that $FS(\{n_i\}_{i=1}^d)\subset F$ and
$$\mu(A\cap\big (\bigcap _{n\in
FS(\{n_i\}_{i=1}^d)} T^{-n}A\big ))>0.$$

A subset $F$ of $\Z$ is a {\em Poincar\'e recurrence set of order
$\infty$} if it is a Poincar\'e recurrence set of order $d$ for any
$d\ge 1$.\index{Poincar\'e recurrence set of order
$\infty$}
\end{de}

\begin{rem}
We remark that $F$ is a Poincar\'e recurrence set of order $1$ if
and only if it is a Poincar\'e sequence. Moreover, a Poincar\'e
recurrence set of order $1$ does not imply that it is a Poincar\'e
recurrence set of order $2$. For example, $\{n^k: n\in\mathbb{N}\}$
($k\ge 3$) is a Poincar\'e sequence \cite{F}, it is not a Poincar\'e
recurrence set of order $2$ by the famous Fermat Last Theorem.
\end{rem}

\subsubsection{Some properties of Poincar\'e recurrence
sets of order $d$ }

Let for $d\in \N\cup\{\infty\}$, $\F_{P_d}$ be the family
consisting of all Poincar\'e recurrence sets of order $d$. \index{$\F_{P_d}$}
Thus
$$\F_{P_1}=\F_{Poi}\supset \F_{P_2}\supset \ldots \supset
\F_{P_d}\supset \ldots\supset \F_{P_\infty}=:\bigcap_{d=1}^\infty
\F_{P_d}.$$

We want to show that $\F_{P_\infty}=\F_{fip}$. It is clear that
$\F_{P_\infty}\subset\F_{fip}$. To show $\F_{P_d}\supset\F_{fip},$
we need the following proposition, for a proof see \cite{G} or
\cite{HLY}. \index{$\F_{P_\infty}$}

\begin{prop}\label{infinite-type}
Let $(X,\mathcal{B},\mu)$ be a
probability space, and $\{E_i\}_{i=1}^\infty$ be a sequence of
measurable sets with $\mu(E_i)\ge a>0$ for some constant $a$ and any
$i\in\N$. Then for any $k\ge 1$ and $\ep>0$ there is $N=N(a, k,
\ep)$ such that for any tuple $\{ s_1<s_2<\ldots <s_n\}$ with $n\ge
N$ there exist $1\le t_1<t_2<\ldots<t_k\le n$ with
\begin{align}\label{bds-key}
\mu(E_{s_{t_1}}\cap E_{s_{t_2}}\cap \ldots \cap E_{s_{t_{k}}})\ge
a^k-\ep.
\end{align}
\end{prop}

\medskip

\begin{rem}\label{remark}
To prove Proposition \ref{longproof}, one needs to use Proposition
\ref{infinite-type} repeatedly. To avoid explaining the same idea
frequently, we illustrate how we will use Proposition
\ref{infinite-type} in the proof of Proposition \ref{longproof}
first.

\medskip

For each $j\in\N$, let $\{k_i^j\}_{i=1}^\infty$ be a sequences in $\Z$.
Assume $(X,\mathcal{B},\mu,T)$ is a measure preserving system and
$A\in \mathcal{B}$ with $\mu(A)>0$. { Let $A_1=A$, $a_1=\mu(A_1)$, and
$a_{j+1}=\frac 12 a_j^2$ for all $j\ge 1$. We
will show that there are a decreasing sequence $\{A_j\}_j$ of measurable sets and
a sequence $\{N_j\}\subset \N$
such that for each $j$,  $\mu(A_j)\ge \frac 12a_{j-1}^2=a_j>0$, and for $n\ge N_j$
and any tuple $\{ s(1)<s(2)<\cdots <s(n)\}$ there exist $1\le
t(1,j)<t(2,j)\le n$ with $\mu(T^{-k^j_{s(t(1,j))}}A_j\cap
T^{-k^j_{s(t(2,j))}}A_j)\ge \frac{1}{2}a_j^2=a_{j+1}$.}

Set $E^1_i=T^{-k^1_i}A, i\in \N$. Let $N_1=N(a_1,2,\frac{1}{2}a_1^2)$
be as in Proposition
\ref{infinite-type}. Then for $n\ge N_1$ and any tuple $\{
s(1)<\cdots <s(n)\}$ there exist $1\le t(1,1)<t(2,1)\le n$ with
$\mu(E^1_{s(t(1,1))}\cap E^1_{s(t(2,1))})\ge \frac{1}{2}a_1^2=a_2.$

Fix $t(1,1)<t(2,1)$ for a given tuple $\{s(1)<\cdots <s(n)\}$.
Now let $A_2=A_1\cap T^{-k^1_{s(t(2,1))}+k^1_{s(t(1,1))}}A_1$. Then
$ \mu(A_2)=\mu(E^1_{s(t(1,1))}\cap E^1_{s(t(2,1))})\ge \frac 12
a_1^2=a_2$. Let $E^2_i=T^{-k^2_i}A_2, i\in \N$ and
$N_2=N(a_2,2,\frac{1}{2}a_2^2)$ be as in Proposition
\ref{infinite-type}.  Thus for $n\ge N_2$ and any tuple $\{
s(1)<\cdots <s(n)\}$ there exist $1\le t(1,2)<t(2,2)\le n$ with
$\mu(E^2_{s(t(1,2))}\cap E^2_{s(t(2,2))})  \ge \frac{1}{2}a_2^2=a_3.$

Inductively, assume that $\{E^j_i=T^{-k^{j}_i}A_{j}\}_{i=1}^\infty,
A_j, a_j, N_j$ are defined such that for $n\ge N_j$ and
any tuple $\{s(1)<\cdots <s(n)\}$ there exist $1\le
t(1,j)<t(2,j)\le n$ with $\mu(E^j_{s(t(1,j))}\cap E^j_{s(t(2,j))}) \ge
\frac{1}{2}a_j^2=a_{j+1}$.

Fix $t(1,j)<t(2,j)$ for a given tuple $\{s(1)<\cdots <s(n)\}$. Let
$A_{j+1}=A_j\cap T^{-k^j_{s(t(2,j))}+k^j_{s(t(1,j))} }A_j$. Then
$$ \mu(A_{j+1})=\mu(E^j_{s(t(1,j))}\cap E^j_{s(t(2,j))})\ge 1/2a_j^2=a_{j+1}.$$ Let
$E^{j+1}_i=T^{-k^{j+1}_i}A_{j+1}$, $i\in \N$, and
$N_{j+1}=N(a_{j+1},2,\frac{1}{2}a_{j+1}^2)$ be as in Proposition
\ref{infinite-type}. Then for $n\ge N_{j+1}$ and any tuple
$\{s(1)<\cdots <s(n)\}$  there exist $1\le t(1,j+1)<t(2,j+1)\le n$
with $ \mu(E^{j+1}_{s(t(1,j+1))}\cap E^{j+1}_{s(t(2,j+1))} )\ge
\frac{1}{2}a_{j+1}^2=a_{j+2}.$

\medskip

Note that the choices of $\{N_i\}$ is independent of
$\{k_i^j\}_{i=1}^\infty$. \hfill $\square$
\end{rem}

Now we are ready to show

\begin{prop}\label{longproof}
The following statements hold.
\begin{enumerate}
\item For each $d\in\N$, $\F_{fip}\subset\F_{P_d},$ which implies that
$\F_{P_\infty}=\F_{fip}.$

\item $\F_{SG_d}\subset \F_{P_d}$ for each $d\in\N\cup\{\infty\}$.
Moreover one has $\F_{fSG_d}\subset \F_{P_d}$.

\end{enumerate}
\end{prop}
\begin{proof}
(1) Let $F\in \F_{fip}$. Fix $d\in \N$. Now we  show $F\in
\F_{P_d}$. For this purpose, assume that $(X,\mathcal{B},\mu,T)$ is
a measure preserving system and $A\in \mathcal{B}$ with $\mu(A)>0$.
Since  $F\in \F_{fip}$, there are $p_1,p_2,\ldots , p_{\ell_ d}\in
\mathbb{Z}$  { with $\ell_d=\sum_{i=1}^d N_i$ }such that $F\supset
FS(\{p_i\}_{i=1}^{\ell_d})$, where $N_i$ are chosen as in Remark
\ref{remark} for $(X,\mathcal{B},\mu,T)$ and $A$.

{ Let $A_1=A$, $a_1=\mu(A_1)$, and
$a_{j+1}=\frac 12 a_j^2$ for all $j\ge 1$.}
For $p_1,p_1+p_2,\cdots, p_1+\cdots+p_{N_1}$ by the
argument in Remark ~\ref{remark} (by setting $\{k^1_i\}=\{p_1,p_1+p_2,\ldots\}$
and $(s(1),\ldots, s(N_1))=(1,\ldots,N_1)$) there is
$q_1=p_{i_1^1}+\cdots+p_{i_2^1}$ such that $\mu(A_1\cap
T^{-q_1}A_1)\ge\frac{1}{2}a_1^2=a_2,$ where $1\le
i_1^1<i_2^1\le N_1.$ Let $A_2=A_1\cap T^{-q_1}A_1$.
For $p_{N_1+1},p_{N_1+1}+p_{N_1+2}, \cdots,
p_{N_1+1}+\cdots+p_{N_1+N_2}$, there is
$q_2=p_{i_1^2}+\cdots+p_{i_2^2}$ such that $\mu(A_2\cap
T^{-q_2}A_2)\ge\frac{1}{2}a_2^2=a_3,$ where $N_1+1\le i_1^2<i_2^2\le
N_1+N_2.$ Note that $q_1,q_2, q_1+q_2\in F$.

Inductively we obtain $$N_1+\ldots+N_j+1\le i_1^{j+1}<i_2^{j+1}\le
N_1+\ldots+N_{j+1},\ 0\le j\le d-1.$$ $q_1,\ldots, q_d$ and $A_1,
\ldots, A_q$ with $q_j=\sum_{i=i_1^j}^{i_2^j}p_i$ and {
$A_j=A_{j-1}\cap T^{-q_{j-1}}A_{j-1}$ such that $\mu(A_j)\ge a_j$ and
$\mu(A_j\cap T^{-q_j}A_j)\ge \frac{1}{2}a_j^2=a_{j+1}$.} Thus
$$\mu(A\cap \bigcap_{n\in FS(\{q_i\}_{i=1}^d)}T^{-n}A)\ge\frac 12 a_d^2>0,$$
and it is clear that $F\supset FS(\{q_i\}_{i=1}^d)$. This implies
that $F\in \F_{P_d}$.

Thus  $\F_{P_\infty}\supset\F_{fip}.$ Since it is clear that
$\F_{P_\infty}\subset\F_{fip},$ we are done.


\medskip
(2) Since each $SG_1$-set is a $\Delta$-set, it is a
Poincar\'e recurrence set (this is easy to be checked by Poincar\'e
recurrence Theorem \cite{F81}). We first show the case when $d=2$
which will illustrate the general idea. Then we give the proof for
the general case.

Let $F\in SG_2$. Then there is $P=\{P_i\}_{i=1}^\infty\subset \Z$
with $F=SG_2(P)$. Let $(X, \mathcal{B},\mu,T)$ be a measure
preserving system and $A\in\mathcal{B}$ with $\mu(A)>0$. Set $A_1=A$
and $a_1=\mu(A_1)$.

Let $$q_1=\sum_{i=1}^{N_2}P_{2i-1},
q_2=\sum_{i=N_2+1}^{2N_2}P_{2i-1},\ \ \ldots,\ \text{and}\
q_{N_1}=\sum_{i=(N_1-1)N_2+1}^{N_1N_2}P_{2i-1},$$ where
$N_1=N(a_1,2,\frac{1}{2}a_1^2)$ and $N_2=N(a_2,2,\frac{1}{2}a_2^2)$
are chosen as in Remark \ref{remark} for $(X,\mathcal{B},\mu,T)$ and
$A$. Consider the sequence $q_1,q_1+q_2, \ldots, q_1+q_2+\ldots
+q_{N_1}$. Then as in Remark \ref{remark} there are $1\le i_1\le j_1\le
N_1$ such that $\mu(A_2)\ge \frac{1}{2}\mu(A)^2$, where $A_2=A_1\cap
T^{-n_1}A_1$ and $n_1=\sum_{i=i_1}^{j_1}q_i $. Note that
$$n_1=P_{2(i_1-1)N_2+1}+P_{2(i_1-1)N_2+3}+ \ldots + P_{2j_1N_2-1}.$$

Now consider the sequence
$$P_{2(i_1-1)N_2}, P_{2(i_1-1)N_2}+P_{2(i_1-1)N_2+2},
\ldots,P_{2(i_1-1)N_2}+P_{2(i_1-1)N_2+2}+\ldots+P_{2i_1N_2}.$$ It
has $N_2+1$ terms. So as in Remark \ref{remark} there are $1\le
i_2\le j_2\le N_2$ such that $\mu(A_2\cap T^{-n_2}A_2)\ge \frac 12 a_2^2$, where
$n_2=\sum_{i=(i_1-1)N_2+i_2}^{(i_1-1)N_2+j_2} P_{2i} $. Note that
$n_1,n_2,n_1+n_2\in F$ by the definition of $SG_2(P)$. It is easy to
verify that
$$\mu(A\cap T^{-n_1}A\cap T^{-n_2}A\cap T^{-n_1-n_2}A)\ge \frac 12
a_2^2>0.$$ Hence $F\in \F_{P_2}$.

\medskip

Now we show the general case. Assume that $d\ge 3$ and let $F\in
SG_d$. We show that $F\in \F_{P_d}$.

Since $F\in SG_d$, there is $P=\{P_i\}_{i=1}^\infty\subset \Z$ with
$F=SG_d(P)$. Let $(X, \mathcal{B},\mu,T)$ be a measure preserving
system and $A\in\mathcal{B}$ with $\mu(A)>0$. Set $A_1=A$. Let
$N_1,\ldots,N_d$ be the numbers as defined in Remark \ref{remark}
for $(X,\mathcal{B},\mu,T)$, $A$ and let $M_i=\prod_{j=i}^d N_j$ for
$1\le i\le d$.

Let $$q_1^1=\sum_{i=1}^{M_2}P_{di-(d-1)},\
q_2^1=\sum_{i=M_2+1}^{2M_2}P_{di-(d-1)},\ \ldots,\
q_{N_1}^1=\sum_{i=(N_1-1)M_2+1}^{M_1}P_{di-(d-1)}.$$ Consider the
sequence $q_1^1,q_1^1+q_2^1, \ldots, q_1^1+q_2^1+\ldots +q_{N_1}^1$.
Then as in Remark \ref{remark} there are $1\le i_1\le j_1\le N_1$ such
that $\mu(A_2)\ge \frac{1}{2}a_1^2$, where $A_2=A_1\cap
T^{-n_1}A_1$ and $n_1=\sum_{i=i_1}^{j_1}q_i^1$.

Let $m_1=(i_1-1)M_2$. Note that there is $t_1\ge M_2-1$ such that
$$n_1=\sum_{i=i_1}^{j_1}q_i^1=P_{dm_1+1}+P_{dm_1+d+1} +\ldots +
P_{dm_1+t_1d+1}.$$

Now consider $$q_1^2=\sum_{i=m_1+1}^{m_1+M_3}P_{di-(d-2)},\
q_2^2=\sum_{i=m_1+M_3+1}^{m_1+2M_3}P_{di-(d-2)},\ \ldots,\
q_{N_2}^2=\sum_{i=m_1+(N_2-1)M_3+1}^{m_1+M_2}P_{di-(d-2)}.$$

Now consider $q_1^2,q_1^2+q_2^2, \ldots, q_1^2+q_2^2+\ldots
+q_{N_2}^2$. It has $N_2$ terms. So as in Remark \ref{remark} there
are $1\le i_2\le j_2\le N_2$ such that $\mu(A_3)\ge
\frac{1}{2}a_2^2$, where $A_3=A_2\cap T^{-n_2}A_2$ and
$n_2=\sum_{i=i_2}^{j_2} q_i^2$. Let $m_2=m_1+(i_2-1)M_3$. Note that
$n_1,n_2,n_1+n_2\in F$ and there is $t_2\ge M_3-1$ such that
$$n_2=\sum_{i=i_2}^{j_2}q_i^2=P_{dm_2+2}+P_{dm_2+d+2} +\ldots +
P_{dm_2+t_2d+2}.$$ Note that $n_2$ has at least $M_3$ terms.

Inductively for $1\le k\le d-1$ we have $1\le i_k\le j_k\le N_k$ and
$$n_k=\sum_{i=i_k}^{j_k}q_i^k=P_{dm_k+k}+P_{dm_k+d+k} +\ldots +
P_{dm_k+t_kd+k},$$ where $t_k\ge M_{k+1}-1$. Also we have
$A_k=A_{k-1}\cap T^{-n_{k-1}}A_{k-1}$ with $\mu(A_k)\ge \frac 12
a_{k-1}^2$, and $FS(\{n_j\}_{j=1}^{k})\subset F$.

Especially, when $k=d$, we get $1\le i_d\le j_d\le N_d$ and
$n_d=\sum_{i=i_d}^{j_d}P_{di}$. By the definition of $SG_d$ we get
that $FS(\{n_i\}_{i=1}^d)\subset F$. From the definition of $A_j,
j=1,2,\ldots, d$, one has
$$\mu(A\cap \bigcap_{n\in FS(\{n_i\}_{i=1}^d)}T^{-n}A)\ge\frac 12 a_d^2>0,$$
which implies that $F\in \F_{P_d}$. The proof is completed.
\end{proof}

\subsubsection{}

\noindent{\it {Proof of Proposition \ref{birkhoff-pro}:}} It is
clear that $\F_{B_\infty}\subset\F_{fip}.$ Since $\F_{fip}\subset
\F_{P_\infty}\subset\F_{B_\infty}$ (by Proposition \ref{longproof}
and the obvious fact that $\F_{P_d}\subset \F_{B_d}$) we have
$\F_{B_\infty}=\F_{fip}.$

\subsection{Poincar\'e recurrence sets and $\RP^{[d]}$}

\begin{thm}\label{poincare} Let $(X,T)$ be a minimal t.d.s. Then
for each $d\in\N\cup \{\infty\}$, $(x,y)\in \RP^{[d]}$ if and only
if $N(x,U)\in \F_{P_d}$ for any neighborhood $U$ of $y$.
\end{thm}

\begin{proof} We first show the case when $d\in\N$.
$(\Leftarrow)$ Since $\F_{P_d}\subset \F_{B_d}$, it follows from
Theorem \ref{birkhoff-thm}.

\medskip
$(\Rightarrow)$ Assume that $(x,y)\in \RP^{[d]}$ and $U$ is a
neighborhood of $y$. By Theorem \ref{huang12}, $N(x,U)\in
\F_{SG_d}$. Then by Proposition \ref{longproof} we have
$N(x,U)\in\F_{P_d}.$

\medskip

The case $d=\infty$ follows from the case $d\in \N$ and definitions.
\end{proof}

\section{Conclusion}

Now we sum up the results of previous three sections.
Note that $\F_{Bir_\infty}$ and $\F_{Poi_\infty}$ can be defined
naturally. Since $\F_{1,0}\subset \F_{2,0}\subset \ldots$ we define
$\F_{\infty,0}=:\bigcup_{d=1}^\infty \F_{d,0}$. Another way to do
this is that one follows the idea in \cite{D-Y} to define
$\infty$-step nilsystems and view $\F_{\infty,0}$ as the family
generated by all Nil$_\infty$ Bohr$_0$-sets. It is easy to check
that Theorem \ref{several} holds for $d=\infty$.

Thus we have
\begin{thm}\label{rpd}
Let $(X,T)$ be a minimal t.d.s. and $x,y\in X$. Then the following
statements are equivalent for $d\in\N\cup\{\infty\}$:

\begin{enumerate}

\item $(x,y)\in \RP^{[d]}$.

\item $N(x,U)\in \F_{d,0}^*$ for each
neighborhood $U$ of $y$.

\item $N(x,U)\in \F_{Poi_d}$ for each
neighborhood $U$ of $y$.

\item $N(x,U)\in \F_{Bir_d}$ for each neighborhood $U$ of $y$.

\item $N(x,U)\in \F_{SG_d}$ for each neighborhood $U$ of $y$.

\item $N(x,U)\in \F_{fSG_d}$ for each neighborhood $U$ of $y$.

\item $N(x,U)\in \F_{B_d}$ for each
neighborhood $U$ of $y$.

\item $N(x,U)\in \F_{P_d}$ for each neighborhood
$U$ of $y$.

\end{enumerate}

\end{thm}

\chapter{$d$-step almost automorpy and recurrence sets}\label{chapter-AA}

In the previous chapter we obtain some characterizations of
regionally proximal relation of order $d$. In the present section we
study $d$-step almost automorpy.

\section{Definition of $d$-step almost automorpy}

\subsubsection{}
First we recall the notion of $d$-step almost automorphic systems
and give its structure theorem.

\begin{de}\index{$d$-step almost automorphic point}
Let $(X,T)$ be a t.d.s. and $d\in \N\cup\{\infty\}$.  A point $x\in
X$ is called a {\em $d$-step almost automorphic} point (or $d$-step
AA point for short) if $\RP^{[d]}(Y)[x]=\{x\}$, where
$Y=\overline{\{T^nx:n\in \mathbb{Z}\}}$ and $\RP^{[d]}(Y)[x]=\{y\in
Y: (x,y)\in \RP^{[d]}(Y)\}$.

A minimal t.d.s. $(X,T)$ is called {\em $d$-step almost automorphic}
($d$-step AA for short) if it has a $d$-step almost automorphic
point.
\end{de}

\begin{rem} Since
$$\RP^{[\infty]}\subset \ldots \subset \RP^{[d]}\subset \RP^{[d-1]}\subset
\ldots \subset \RP^{[1]},$$ we have
$$\text{AA}=\text{1-step AA}\Rightarrow \ldots \Rightarrow\text{(d-1)-step AA}
\Rightarrow\text{d-step AA}\Rightarrow \ldots \Rightarrow
\infty\!-\!\text{step AA}.$$
\end{rem}

\subsubsection{}
The following theorem follows from Theorem \ref{ShaoYe}.

\begin{thm}[Structure of $d$-step almost
automorphic systems]\label{thm-AA}
Let $(X,T)$ be a minimal t.d.s. Then $(X,T)$ is a $d$-step almost
automorphic system for some $d\in \N\cup\{\infty\}$ if and only if
it is an almost one-to-one extension of its maximal $d$-step
nilfactor $(X_d, T)$.

\[
\begin{CD}
X @>{T}>> X\\
@V{\pi}VV      @VV{\pi}V\\
X_d @>{T }>> X_d
\end{CD}
\]

\end{thm}

\medskip

\subsection{$1$-step almost automorphy}
First we recall some classical results about almost automorphy.

Let $(X,T)$ be a minimal t.d.s.. In \cite{V68} it is proved that
$(x,y)\in \RP^{[1]}$ if and only if for each neighborhood $U$ of
$y$, $N(x,U)$ contains some $\D$-set, see also Theorem
\ref{huang12}. Similarly, we have that for a minimal system $(X,T)$,
$(x,y)\in \RP^{[1]}$ if and only if for each neighborhood $U$ of
$y$, $N(x,U)\in \F_{Poi}$ \cite{HLY}, see also Theorem \ref{rpd}.

Using these theorems and the facts that $\F_{Poi}$ and $\F_{Bir}$
have the Ramsey property, one has

\begin{thm}\label{thm6.4}
Let $(X,T)$ be a minimal t.d.s. and $x\in X$. Then the following
statements are equivalent:
\begin{enumerate}
\item $x$ is AA.

\item $N(x,V)\in \F_{Poi}^*$ for each
neighborhood $V$ of $x$.

\item $N(x,V)\in \F_{Bir}^*$ for each
neighborhood $V$ of $x$.

\item  $N(x,V)\in \Delta^*$ for each
neighborhood $V$ of $x$. \cite{V65,F}
\end{enumerate}
\end{thm}

We will not give the proof of this theorem since it is a special
case of Theorem~ \ref{AAgeneral}.

\subsection{$\infty$-step almost automorphy}

In this subsection we give one characterization for $\infty$-step
AA. Following from Theorem \ref{ShaoYe}, one has

\begin{prop}\label{shaoye}
Let $(X,T)$ be a minimal t.d.s. and $d\ge 1$. Then

\begin{enumerate}
\item $(x,y)\in \RP^{[d]}$ if and only if $N(x,U)$ contains a finite IP-set of
length $d+1$ for any neighborhood $U$ of $y$, and thus

\item $(x,y)\in \RP^{[\infty]}$ if and only if $N(x,U)\in \F_{fip}$ for any
neighborhood $U$ of $y$.
\end{enumerate}
\end{prop}

To show the next theorem we need the following lemma which should be
known, see for example Huang, Li and Ye \cite{HLY2}.

\begin{lem}
$\F_{fip}$ has the Ramsey property.
\end{lem}

We have the following

\begin{thm}\label{aainfty}
Let $(X,T)$ be a minimal t.d.s. Then $(X,T)$ is $\infty$-step AA if
and only if there is $x\in X$ such that $N(x,V)\in \F_{fip}^*$ for
each neighborhood $V$ of $x$.
\end{thm}

\begin{proof}
Assume that there is $x\in X$ such that $N(x,V)\in \F_{fip}^*$ for
each neighborhood $V$ of $x$. If there is $y\in X$ such that
$(x,y)\in\RP^{[\infty]}$, then by Proposition \ref{shaoye} for any
neighborhood $U$ of $y$, $N(x,U)\in\F_{fip}$. This implies that
$x=y$, i.e. $(X,T)$ is $\infty$-step AA.

Now assume that $(X,T)$ is $\infty$-step AA, i.e. there is $x\in X$
such that $\RP^{[\infty]}[x]=\{x\}$. If for some neighborhood $V$ of
$x$, $N(x,V)\not\in \F_{fip}^*$, then $N(x,V^c)$ contains finite
IP-sets of arbitrarily long lengths.

Let $U_1=V^c$. Covering $U_1$ by finitely many closed balls
$U_1^1,\ldots, U_1^{i_1}$ of diam $\le 1$. Then there is $j_1$ such
that $N(x,U_1^{j_1})$ contains finite IP-sets of arbitrarily long
lengths. Let $U_2=U_1^{j_1}$. Covering $U_1$ by finitely many closed
balls $U_2^1,\ldots, U_2^{i_2}$ of diam $\le \frac{1}{2}$. Then
there is $j_2$ such that $N(x,U_2^{j_2})$ contains finite IP-sets of
arbitrarily long lengths. Let $U_3=U_2^{j_2}$. Inductively, there
are a sequence of closed balls $U_n$ with diam $\le \frac{1}{n}$
such that $N(x,U_n)$ contains finite IP-sets of arbitrarily long
lengths. Let $\{y\}=\bigcap U_n$. It is clear that $(x,y)\in
\RP^{[\infty]}$ with $y\not=x$, a contradiction. Thus $N(x,V)\in
\F_{fip}^*$ for each neighborhood $V$ of $x$.
\end{proof}

\section{Characterization of $d$-step almost automorphy}

Now we use the results built in previous sections to get the
following characterization for $d$-step AA via recurrence sets.

\begin{thm}\label{AAgeneral}
Let $(X,T)$ be a minimal t.d.s., $x\in X$ and
$d\in\N\cup\{\infty\}$. Then the following statements are
equivalent:
\begin{enumerate}
\item $x$ is a $d$-step AA point.

\item $N(x,V)\in \F_{d,0}$ for each neighborhood $V$ of $x$.

\item $N(x,V)\in \F_{Poi_d}^*$ for each
neighborhood $V$ of $x$.

\item $N(x,V)\in \F_{Bir_d}^*$ for each
neighborhood $V$ of $x$.

\end{enumerate}
\end{thm}

\begin{proof}
Roughly speaking this theorem follows from Theorem \ref{rpd}, the
fact $\F_{d,0}^*, \F_{Poi_d}$ and $\F_{Bir_d}$ have the Ramsey
property, and the idea of the proof of Theorem \ref{aainfty}. We
show that $(1)\Leftrightarrow(2)$, and the rest is similar.

$(1)\Rightarrow (2)$: Let $x$ be a $d$-step AA point. If (2) does
not hold, then there is some neighborhood $V$ of $x$ such that
$N(x,V)\not \in \F_{d,0}$. Then $N(x,V^c)=\Z\setminus N(x,V)\in
\F^*_{d,0}$. Since $\F^*_{d,0}$ has the Ramsey property, similar to
the proof of Theorem \ref{aainfty} one can find some $y\in V^c$ such
that $N(x,U)\in \F^{*}_{d,0}$ for every neighborhood $U$ of $y$. By
Theorem \ref{rpd}, $y\in \RP^{[d]}[x]$. Since $y\neq x$, this
contradicts the fact $x$ being $d$-step AA.

$(2)\Rightarrow (1)$: If $x$ is not $d$-step AA, then there is some
$y\in \RP^{[d]}[x]$ with $x\neq y$. Let $U_x$ and $U_y$ be
neighborhoods of $x$ and $y$ with $U_x\cap U_y=\emptyset$. By $(2)$
$N(x,U_x)\in \F_{d,0}$. By Theorem \ref{rpd}, $N(x,U_y)\in
\F^*_{d,0}$. Hence $N(x,U_x)\cap N(x,U_y)\neq \emptyset$,  which
contradicts the fact that $U_x\cap U_y=\emptyset$.
\end{proof}

\appendix

\chapter{}

\section{The Ramsey properties} \label{appendix:Ramsey}

Recall that a family $\F$ has the {\em Ramsey property} means that if
$A\in\F$ and $A=\cup_{i=1}^n A_i$ then one of $A_i$ is still in
$\F$. In this section, we show that $\F_{SG_2}$ does not have the
Ramsey property. \index{Ramsey property}

\begin{thm}
$\F_{SG_2}$ does not have the Ramsey property.
\end{thm}

\begin{proof}
Let $P=\{p_1, p_2, \ldots \}$ be a subsequence of $\N$ with
$p_{i+1}>2(p_1+\ldots +p_i)$. The assumption that
$p_{i+1}>2(p_1+\ldots +p_i)$ ensures that each element of $SG_2(P)$
has a unique expression with the form of $\sum_i p_{j_i}$.

Now divide the set $SG_2(P)$ into the following three sets:
\begin{equation*}
\begin{split}
B_1&=\{p_{2n-1}+\ldots+p_{2m-1}:n\le m\in\N\}=SG_1(\{p_1,p_3,\ldots\}),\\
B_2&=\{p_{2n}+\ldots+p_{2m}:n\le m\in\N\}=SG_1(\{p_2,p_4,\ldots\}),\\
B_0&=SG_2(P)\setminus (B_1\cup B_2).
\end{split}
\end{equation*}

We show that $B_i\not\in \F_{SG_2}$ for $i=0,1,2.$  In fact, we will
prove that for each $i=0,1,2$ there do not exist $a_1\le a_2\le a_3$ such
that
\begin{equation}
    a_1,a_2,a_3,a_1+a_2,a_2+a_3,a_1+a_3\in B_i, \tag{$*$}
\end{equation}
which obviously implies that $B_i\not\in \F_{SG_2}$ for $i=0,1,2.$
\medskip

\noindent (1). First we show $B_2\not\in \F_{SG_2}$. The proof
$B_1\not\in \F_{SG_2}$ follows similarly. Assume the contrary, i.e.
there exist $a_1\le a_2\le a_3$ such that
\begin{equation*}
    a_1,a_2,a_3,a_1+a_2,a_2+a_3,a_1+a_3\in B_2.
\end{equation*}
Let
\begin{equation*}
\begin{split}
a_1&=p_{2n_1}+\ldots+p_{2m_1},\ n_1\le m_1;\\
a_2&=p_{2n_2}+\ldots+p_{2m_2},\ n_2\le m_2;\\
a_3&=p_{2n_3}+\ldots+p_{2m_3},\ n_3\le m_3.
\end{split}
\end{equation*}
Since $a_1\le a_2\le a_3$ and the assumption that $p_{i+1}>2(p_1+\ldots
+p_i)$, one has that $m_1\le m_2\le m_3$. Since $a_1+a_2, a_2+a_3\in
B_2$, one has that $n_2=m_1+1$ and $n_3=m_2+1$. Hence $n_3=m_2+1\ge
n_2+1=m_1+2$, i.e. $n_3>m_1+1$. Thus $$a_1+a_3\not \in B_2,$$ a
contraction!

\medskip

\noindent (2). Now we show $B_0\not\in \F_{SG_2}$. Assume the
contrary, i.e. there exist $a_1\le a_2\le a_3$ such that
\begin{equation*}
    a_1,a_2,a_3,a_1+a_2,a_2+a_3,a_1+a_3\in B_0.
\end{equation*}
Let
\begin{equation*}
\begin{split}
a_1&=p_{i^1_1}+p_{i^1_2}+\ldots+p_{i^1_{k_1}};\\
a_2&=p_{i^2_1}+p_{i^2_2}+\ldots+p_{i^2_{k_2}};\\
a_3&=p_{i^3_1}+p_{i^3_2}+\ldots+p_{i^3_{k_3}},
\end{split}
\end{equation*}
where $i^r_1<i^r_2<\ldots <i^r_{k_r}$, $i^r_{j+1}\le i^r_j+2$ for
$1\le j\le k_r-1$, and there are both even and odd numbers in
$\{i^r_1,i^r_2,\ldots,i^r_{k_r}\}$ $(r=1,2,3)$.

Since there are both even and odd numbers in
$\{i^r_1,i^r_2,\ldots,i^r_{k_r}\}$ $(r=1,2,3)$ and $i^r_{j+1}\le
i^r_j+2$ for $1\le j \le k_r-1$, there exist $1\le j_r\le k_r-1$
such that $i^r_{j+1}=i^r_{j_r}+1$. Since $a_1\le a_2\le a_3$ and the
assumption that $p_{i+1}>2(p_1+\ldots +p_i)$, one has that
$i^1_{k_1}\le i^2_{k_2}\le i^3_{k_3}$. Note that we have
\begin{equation*}
    i^1_1<i^1_2<\ldots <i^1_{j_1}<i^1_{j_1+1}=i^1_{j_1}+1<\ldots<i^1_{k_1},
\end{equation*}
\begin{equation*}
    i^2_1<i^2_2<\ldots <i^2_{j_2}<i^2_{j_2+1}=i^2_{j_2}+1<\ldots<i^2_{k_2},
\end{equation*}
\begin{equation*}
    i^3_1<i^3_2<\ldots <i^3_{j_3}<i^3_{j_3+1}=i^3_{j_3}+1<\ldots<i^3_{k_3}.
\end{equation*}

The condition $a_1+a_2\in B_0$ implies that
\begin{equation}
    i^1_{j_1+1}<i^2_1\le i^1_{k_1}+2; \ i^1_{k_1}<i^2_{j_2}. \tag{a}
\end{equation}
In fact if $i^2_1<i^1_{j_1}$, then the gap $\{i^1_{j_1},
i^1_{j_1}+1\}$ is missing in the term of $a_2$ and it contradicts
the assumption $a_2\in SG_2(P)\in \F_{SG_2}$. The statement
$i^1_{k_1}<i^2_{j_2}$ follows by the same argument.

Similarly, using the assumptions $a_2+a_3\in B_0$ and $a_1+a_3\in
B_0$, one has
\begin{equation}
    i^2_{j_2+1}<i^3_1\le i^2_{k_2}+2; \ i^2_{k_2}<i^3_{j_3}. \tag{b}
\end{equation}
and
\begin{equation}
    i^1_{j_1+1}<i^3_1\le i^1_{k_1}+2; \ i^1_{k_1}<i^3_{j_3}. \tag{c}
\end{equation}
From (a), we have that $i_{k_1}^1<i^2_{j_2}$; and from (b), we have
$i^2_{j_2+1}=i^2_{j_2}+1< i^3_1$. Hence we have $i^3_1\ge
i^1_{k_1}+3$, which contradicts (c). The proof is completed.
\end{proof}

\section{Compact Hausdorff systems}\label{appendix:CT2}

In this section we discuss compact Hausforff systems, i.e. the
systems with phase space being compact Hausdorff. The reason for
this is not generalization for generalization's sake, but rather
that we have to deal with non-metrizable systems. For example, we
will use (in the proof of Theorem \ref{huang10}) an important tool
named Ellis semigroup which is a subspace of an uncountable product
of copies of the phase space and therefore in general not
metrizable.

\subsection{Compact Hausdorff systems}

In the classical theory of abstract topological dynamics, the basic
assumption about the system is that the space is a compact Hausdorff
space and the action group is a topological group. In this paper, we
mainly consider the compact metrizable system under $\Z$-actions,
but in some occasions we have to deal with compact Hausdorff spaces
which are non-metrizable. Note that each compact Hausdorff space is
a uniform space, and one may use the uniform structure replacing the
role of a metric, see for example the Appendix of \cite{Au88}.

First we recall a classical equality concerning regionally proximal
relation in compact Hausdorff systems.
A {\em compact Hausdorff system} is a pair $(X,T)$, where $X$ is a
compact Hausdorff space and $T:X\rightarrow X$ is a homeomorphism.
Let $(X,T)$ be a compact Hausdorff system and $\mathcal{U}_X$ be the
unique uniform structure of $X$. 
The {\em regionally proximal relation} on $X$ is defined
by \index{regionally proximal relation}
\begin{equation*}
    \RP=\bigcap_{\a \in \mathcal{U}_X}\overline{ \bigcup_{n\in \Z}(T\times
    T)^{-n}\a}
\end{equation*}

\subsection{Ellis semigroup} \index{Ellis semigroup}

A beautiful characterization of distality was given by R. Ellis
using so-called enveloping semigroup. Given a compact Hausdorff
system $(X,T)$, its {\it enveloping semigroup} (or {\em Ellis
semigroup}) \index{enveloping semigroup} $E(X,T)$ is defined as the closure of the set $\{T^n: n
\in \Z\}$ in $X^X$ (with its compact, usually non-metrizable,
pointwise convergence topology). Ellis showed that a compact
Hausdorff system $(X,T)$ is distal if and only if $E(X,T)$ is a
group if and only if every point in $(X^2, T\times T)$ is minimal
\cite{E69}.

\subsection{Limits of Inverse systems}

Suppose that  every $\ll$ in a set $\Lambda$ directed by the
relation $\le$ corresponds a t.d.s. $(X_\ll,T_\ll)$, and that for
any $\ll,\xi\in \Lambda$ satisfying $\xi\le \ll$ a factor map
$\pi^\ll_\xi: (X_\ll,T_\ll)\rightarrow (X_\xi,T_\xi)$ is defined;
suppose further that $\pi^\xi_\tau\pi^\ll_\xi=\pi^\ll_\tau$ for all
$\ll, \xi,\tau\in \Lambda$ with $\tau\le \xi\le \ll$ and that
$\pi^\ll_\ll={\rm id}_X$ for all $\ll\in \Lambda$. In this situation
we say that the family $\{X_\ll,
\pi^\ll_\xi,\Lambda\}=\{(X_\ll,T_\ll), \pi^\ll_\xi,\Lambda\}$ is an
{\em inverse system of the systems $(X_\ll,T_\ll)$}; and the
mappings $\pi^\ll_\xi$ are called {\em bonding mappings} of the
inverse system.\index{inverse system}

Let $\{X_\ll, \pi^\ll_\xi,\Lambda\}$ be an inverse system. The {\em
limit of the inverse system  $\{X_\ll, \pi^\ll_\xi,\Lambda\}$} is
the set
\begin{equation*}
    \Big\{(x_\ll)_\ll\in \prod_{\ll\in \Lambda}X_\ll: \pi^\ll_\xi(x_\ll)=
    x_\xi\ \text{for all $\xi\le \ll\in \Lambda$}\Big\},
\end{equation*}
and is denoted by $\varprojlim \{X_\ll, \pi^\ll_\xi,\Lambda\}$. Let
$X=\varprojlim \{X_\ll, \pi^\ll_\xi,\Lambda\}$. For each $\ll\in
\Lambda $, let $\pi_\ll: X\rightarrow X_\ll,
(x_\sigma)_\sigma\mapsto x_\ll$ be the projection mapping.

\medskip

A well known result is the following (see for example \cite{Ke}):

\begin{lem}\label{inverse}
Each compact Hausdorff system is the inverse limit of topological
dynamical systems.
\end{lem}

\subsection{The regionally proximal relation of order $d$ for
compact Hausdorff systems}

The definition of the regionally proximal relation of order $d$ for
compact Hausdorff systems is similar to the metric case.

\begin{de}\index{regionally proximal relation of order $d$}
Let $(X, T)$ be a compact Hausdorff system, $\mathcal{U}_X$ be the
unique uniform structure of $X$ and let $d\ge 1$ be an integer. A
pair $(x, y) \in X\times X$ is said to be {\em regionally proximal
of order $d$} if for any $\a  \in \mathcal{U}_X$, there exist $x',
y'\in X$ and a vector ${\bf n} = (n_1,\ldots , n_d)\in\Z^d$ such
that $(x, x') \in  \a, (y, y') \in \a$, and $$ (T^{{\bf n}\cdot
\ep}x', T^{{\bf n}\cdot \ep}y') \in \a \ \text{for any $\ep\in
\{0,1\}^d$, $\ep\not=(0,\ldots,0)$},
$$ where ${\bf n}\cdot \ep = \sum_{i=1}^d \ep_in_i$. The set of all
regionally proximal pairs of order $d$ is denoted by $\RP^{[d]}(X)$,
which is called {\em the regionally proximal relation of order $d$}.
\end{de}

By Lemma \ref{inverse}, each compact Hausdorff system is the inverse
limit of topological dynamical systems. Recall the definition of the
product uniformity. Let $(X_\ll, \U_\ll)_{\ll\in \Lambda}$ be a
family of uniform spaces and let $Z=\prod_{\ll\in \Lambda}X_\ll$.
The uniformity on $Z$ (the product uniformity) is defined as
follows. If $F=\{\ll_1,\ldots,\ll_m\}$ is a finite subset of the
index set $\Lambda$ and $\a_{\ll_j}\in \U_{\ll_j}$ $(j=1,\ldots,m)$,
let
$$\Phi_{\a_{\ll_1},\ldots,\a_{\ll_m}}=\{(x,y)\in Z\times Z:
(x_{\ll_j},y_{\ll_j})\in \a_{\ll_j} ,\ j=1,\ldots, m\}.$$ The
collection of all such sets $\Phi_{\a_{\ll_1},\ldots,\a_{\ll_m}}$
for all finite subsets $F$ of $\Lambda$ is a base for the product
uniformity. From this and the definition of the regionally proximal
relation of order $d$, one has the following result.

\begin{prop}
Let $(X,T)$ be a compact Hausdorff system and $d\in \N$. Suppose
that $X=\varprojlim \{X_\ll, \pi^\ll_\xi,\Lambda\}$, where $(X_\ll,
T_\ll)_{\ll\in \Lambda}$ are t.d.s.. Then
$$\RP^{[d]}(X)=\varprojlim
\{\RP^{[d]}(X_\ll), \pi^\ll_\xi\times \pi^\ll_\xi,\Lambda\}.$$
\end{prop}

Thus combining this proposition with Theorem \ref{ShaoYe}, one has

\begin{thm}
Let $(X, T)$ be a minimal compact Hausdorff system and $d\in \N$.
Then
\begin{enumerate}

\item $\RP^{[d]}(X)$ is an equivalence relation, and so is
$\RP^{[\infty]}.$

\item If $\pi:(X,T)\lra (Y,S)$ is a factor map, then $(\pi\times
\pi)(\RP^{[d]}(X))=\RP^{[d]}(Y).$

\item $(X/\RP^{[d]},T)$ is the maximal $d$-step nilfactor of $(X,T)$.
\end{enumerate}
\end{thm}

Note that for a compact Hausdorff system $(X,T)$ we say that it is a
system of order $d$ for some $d\in\N$ if it is an inverse limit of
$d$-step minimal nilsystems.

\section{Intersective}\label{appendix:Intersective}
It is well known that $P$ is a Birkhoff recurrence set if and only
if $P\cap(F-F)\not=\emptyset$ for each $F\in\F_{s}$. To give a
similar characterization we have
\begin{de}\index{topologically intersective of order $d$}
A subset $P$ is {\em topologically intersective of order $d$} if for each
$F\in \F_{s}$ there are $n_1,\ldots,n_d$ with
$FS(\{n_i\}_{i=1}^d)\subset P$ and $a\in F$ with
$a+FS(\{n_i\}_{i=1}^d)\subset F$, i.e. $F\cap\bigcap_{n\in
FS(\{n_i\}_{i=1}^d)} (F-n)\not=\emptyset$.
\end{de}

\begin{thm}A subset $P$ is topologically intersective of order $d$ if and
only if it is a Birkhoff recurrence set of order $d$.
\end{thm}

\begin{proof} Assume that $P$ is a Birkhoff recurrence set of order $d$.
Let $F\in\F_s$. Then $1_F\in \{0,1\}^{\Z}$. Let $(X,T)$ be a minimal
subsystem of $(\overline{\O(1_F,T)},T)$, where $T$ is the shift.
Since $F\in \F_s$, $[1]$ is a non-empty open subset of $X$. By the
definition there are $n_1,\ldots,n_d$ with
$FS(\{n_i\}_{i=1}^d)\subset P$ such that $[1] \cap
\big(\bigcap_{n\in FS(\{n_i\}_{i=1}^d)}T^{-n}[1] \big)\neq
\emptyset.$ It implies that there is $a\in F$ with
$a+FS(\{n_i\}_{i=1}^d)\subset F$ and hence $P$ is topologically
intersective of order $d$.

\medskip
Assume that $P$ is topologically intersective of order $d$. Let
$(X,T)$ be a minimal t.d.s. and $U$ be an open non-empty subset.
Take $x\in U$, then $F=N(x,U)\in \F_s$. Thus there are
$n_1,\ldots,n_d$ with $FS(\{n_i\}_{i=1}^d)\subset P$ and $a\in F$
with $a+FS(\{n_i\}_{i=1}^d)\subset F$. It follows that $U \cap
\big(\bigcap_{n\in FS(\{n_i\}_{i=1}^d)}T^{-n}U \big)\neq \emptyset.$
\end{proof}

It is well known that $P$ is a Poincar\'e recurrence set if and only
if $P\cap(F-F)\not=\emptyset$ for each $F\in\F_{pubd}$. To give a
similar characterization we have

\begin{de} A subset $P$ is intersective of order $d$ if for each $F\in
\F_{pubd}$ there are $n_1,\ldots,n_d$ with
$FS(\{n_i\}_{i=1}^d)\subset P$ and $a\in F$ with
$a+FS(\{n_i\}_{i=1}^d)\subset F$.
\end{de}

\begin{thm}
A subset is intersective of order $d$ if and only if it is a
Poincar\'e recurrence set of order $d$.
\end{thm}
\begin{proof}
Assume that $P$ is intersective of order $d$. Let $(X,\mathcal{B},
\mu,T)$ be a measure preserving system and $A\in\mathcal{B}$ with
$\mu(A)>0$. By the Furstenberg corresponding principle, there exists
$F\subset \mathbb{Z}$ such that $D^*(F)\ge \mu(A)$ and
\begin{equation}\label{F1}
\{\alpha\in \F(\mathbb{Z}):\bigcap_{n\in \alpha}
(F-n)\not=\emptyset\}\subset \{\alpha\in
\F(\mathbb{Z}):\mu(\bigcap_{n\in \alpha} T^{-n}A)>0\},
\end{equation}
where $\F(\mathbb{Z})$ denotes the collection of non-empty finite
subsets of $\mathbb{Z}$. Since $P$ is intersective of order $d$,
there are $n_1,\ldots,n_d$ with $FS(\{n_i\}_{i=1}^d)\subset P$ and
$a\in F$ with $a+FS(\{n_i\}_{i=1}^d)\subset F$, i.e. $F\cap \bigcap
\limits_{n\in FS(\{n_i\}_{i=1}^d)} (F-n)\not=\emptyset$. By
(\ref{F1}) $P\in \F_{P_d}$.

Now assume that $P\in\F_{P_d}$ and $F\in \F_{pubd}$. Then by the
Furstenberg corresponding principle, there are a measure preserving
system $(X,\mathcal{B}, \mu,T)$ and $A\in\mathcal{B}$ such that
$\mu(A)=BD^*(F)>0$ and
\begin{equation}\label{F2}
BD^*(\bigcap_{n\in \alpha} (F-n))\ge \mu(\bigcap_{n\in \alpha}
T^{-n}A)
\end{equation}
for all $\alpha\in \F(\Z)$. Since $P\in\F_{P_d}$, there are
$n_1,\ldots,n_d$ with $FS(\{n_i\}_{i=1}^d)\subset P$ and $\mu(A\cap
\bigcap \limits_{n\in FS(\{n_i\}_{i=1}^d)}T^{-n}A)>0.$ This implies
$F\cap \bigcap \limits_{n\in FS(\{n_i\}_{i=1}^d)}
(F-n)\not=\emptyset$ by (\ref{F2}).
\end{proof}

\backmatter
\bibliographystyle{amsalpha}

\printindex
\end{document}